\documentclass[a4paper,reqno,12pt]{amsart}
\usepackage[utf8]{inputenc}

\usepackage[T1]{fontenc}
\usepackage{lmodern}
\usepackage[english]{babel}
\usepackage{amsmath,a4wide} 
\allowdisplaybreaks
\usepackage{xfrac}
\usepackage{esint}
\usepackage{stmaryrd,mathrsfs,bm,amsthm,mathtools,yfonts,amssymb,color,braket,booktabs,graphicx,graphics,amsfonts}
\usepackage{latexsym,microtype,indentfirst}
\usepackage{hyperref}
\hypersetup{colorlinks=true,
	linkcolor=blue,
	citecolor=blue,
	urlcolor=black,
	linktoc=all}
    \usepackage{cleveref}
\usepackage{xcolor}
\usepackage{courier}
\usepackage{constants}
\usepackage{comment}
\usepackage{multicol}
\usepackage[dvipsnames]{xcolor}

\usepackage{scalerel}[2016-12-29]
\def\stretchint#1{\vcenter{\hbox{\stretchto[440]{\displaystyle\int}{#1}}}}

\usepackage{todonotes}
\usepackage{cancel}
\usepackage{color,soul}
\usepackage[dvipsnames]{xcolor}
\usepackage{mathtools}
\usepackage{subcaption}
\usepackage{aliascnt} 

\newcommand{\Na}{\mathbb{N}} 
\newcommand{\R}{\mathbb{R}} 

\newcommand{\eps}{\varepsilon} 
\newcommand{\spt}{\mathrm{spt}} 
\newcommand{\Lip}{\mathrm{Lip}} 
\newcommand{\Ha}{\mathcal{H}} 
\newcommand{\Leb}{\mathcal{L}} 
\newcommand{\pa}{\partial} 
\newcommand{\Int}{{\rm int}\,} 

\newcommand{\mres}{\mathbin{\vrule height 1.6ex depth 0pt width 
0.13ex\vrule height 0.13ex depth 0pt width 1.3ex}}

\newcommand{\abs}[1]{{\left|#1\right|}}
\newcommand{\norm}[1]{\left\lVert#1\right\rVert} 
\newcommand{\normpiccolo}[1]{\mathinner{\mathopen{\vstretch{1.21}{\|}} #1 \mathclose{\vstretch{1.21}{\|}}}}

\newcommand{\V}{\mathbf{V}} 
\newcommand{\RV}{\mathbf{RV}}
\newcommand{\IV}{\mathbf{IV}} 
\newcommand{\var}{\mathbf{var}} 
\newcommand{\E}{\mathcal{E}} 
\newcommand{\op}{\mathcal{OP}}
\newcommand{\cA}{\mathcal{A}}
\newcommand{\cB}{\mathcal{B}}
\renewcommand{\d}{\mathrm{d}} 
\newcommand{\bG}{\mathbf{G}} 

\newcommand{\rC}{\mathrm{C}} 

\renewcommand{\d}{\mathrm{d}}

\newcommand{\ssubset}{\subset\joinrel\subset}
\DeclareMathOperator*{\esssup}{ess\,sup}

\newconstantfamily{eps}{
symbol=\varepsilon}
\newconstantfamily{con}{
symbol=c}

\renewcommand{\set}[1]{{\left\{#1\right\}}} 

\newcommand{\pair}[1]
{\left\langle#1\right\rangle}

\newcommand{\define}{\vcentcolon =}

\renewcommand{\l}{\ell} 

\newcommand{\bigint}[1]{\stretchint{6.5ex}_{\mkern-11mu #1}} 

\newcommand{\T}[1]{\mathcal T\left(#1\right)}

\newcommand{\fb}{\mathrm{fb}}
\newcommand{\h}{\mathrm h}

\newcommand{\hepsfb}{\h_\eps^\fb}

\newcommand{\htildeepsfb}{\widetilde \h_\eps^\fb}

\newcommand{\classEvc}{{\bf E}^{\mathrm{vc}}} 
\newcommand{\Deltajvc}{\Delta_j^{\mathrm{vc}}} 

\newcommand{\opn}{\mathcal{OP}^N} 

\newcommand{\interior}[1]{\rm{int}\left( #1 \right)}
\newcommand{\clos}[1]{\mathrm{clos}\left( #1 \right)}
\newcommand{\abspiccolo}[1]{\bigl|#1\bigr|}
\newcommand{\setpiccolo}[1]{\mathinner{\mathopen{\vstretch{1.21}{\{}} #1 \mathclose{\vstretch{1.21}{\}}}}}

\newcommand{\leb}{\mathcal L^{n+1}}
\newcommand{\closedhalf}{\R^{n+1}_{\ge 0}}
\newcommand{\opns}{\mathcal{OPS}^N} 
\newcommand{\virgolette}[1]{``#1’’}
\newcommand{\eltwoloc}{L^2_{\mathrm{loc}}}
\renewcommand{\emptyset}{\varnothing}

\usepackage[margin=1.in]{geometry}

\newtheorem{theorem}{Theorem}[section]
\newtheorem{corollary}[theorem]{Corollary}

\theoremstyle{definition}
\newtheorem*{notation}{Notation}
\newtheorem{example}[theorem]{Example}

\newaliascnt{assumption}{theorem}
\newaliascnt{definition}{theorem}
\newaliascnt{lemma}{theorem}
\newaliascnt{remark}{theorem}
\newaliascnt{proposition}{theorem}
\newtheorem{assumption}[assumption]{Assumption}
\newtheorem{definition}[definition]{Definition}
\newtheorem{lemma}[lemma]{Lemma}
\newtheorem{remark}[remark]{Remark}
\newtheorem{proposition}[proposition]{Proposition}
\aliascntresetthe{assumption}
\aliascntresetthe{definition}
\aliascntresetthe{lemma}
\aliascntresetthe{remark}
\aliascntresetthe{proposition}

\numberwithin{equation}{section}
\numberwithin{subsection}{section}

\title[Free boundary Brakke]{On the existence of canonical multi-phase free boundary Brakke flows: a case study} 

\author[A. Scapin]{Alessandro Scapin}
\address{Dipartimento di Matematica, Universit\`{a} degli Studi di Milano, Via Saldini 50, I-20133 Milano (MI), Italy}
\email{alessandro.scapin1@unimi.it}
\begin{document}
\bibliographystyle{plain}
\begin{abstract}
    We establish the global-in-time existence of a codimension $1$ canonical multi-phase free boundary Brakke flow in the upper halfspace  which is integer rectifiable up boundary, starting from a countably $n$-rectifiable set. Under a suitable uniform density ratio assumption on the initial datum, we show that the free boundary carries no positive mass for some short time.
\end{abstract}
\maketitle

\makeatletter
   \providecommand\@dotsep{5}
   \makeatother
   \relax

\section{Introduction}
A mean curvature flow (abbreviated hereafter as MCF) is a family of surfaces $\set{\Gamma(t)}_{t\ge0}$ moving with normal velocity equal to the mean curvature of $\Gamma(t)$ at each point and time. This is arguably the most fundamental geometric flow involving extrinsic curvature, as it arises as the gradient flow of the area functional, making it some sort of geometric analogue of the heat equation. The initial value problem for the MCF with a smooth closed initial datum $\Gamma_0$ is locally well-posed in time, until the appearance of singularities. A large number of generalized solutions past singularities have been proposed since the 1970s: as a non-exhaustive list, we mention the viscosity solutions and the associated level set flows \cite{ChenGigaGoto, EvansSpruck}, Brakke flows \cite{Brakke, KimTonegawa}, BV solutions \cite{LuckhausSturzenhecker} and $L^2$ flows \cite{MugnaiRoger}. In the present paper, we  focus on the Brakke flow. 

Besides the boundaryless MCF, it is natural to study the MCF in a given domain with boundary conditions, such as Dirichlet or Neumann. The Brakke flow with Dirichlet boundary conditions has been studied by Stuvard-Tonegawa \cite{STDirichlet}, whereas we are interested here in the study of the \emph{free boundary} MCF, namely a family of surfaces moving by mean curvature and attaching to the boundary of the domain orthogonally, at least in a weak measure-theoretic sense. The aim of the present paper is to establish the global-in-time existence of a multi-phase free boundary canonical Brakke flow in the upper halfspace starting from a rectifiable initial datum. By multi-phase we mean that the evolving surfaces, for any given time, are the boundaries of finitely many open sets, possibly empty, $\leb$-partitioning the upper halfspace $\R^{n+1}_+$ (henceforth referred to as \textit{grains}). The attribute ``canonical’’ refers instead to a notion introduced by Stuvard-Tonegawa \cite{STCanonical}, where the evolution of the surfaces is coupled with the evolution of the grains, yielding a BV-type relation. In particular, this notion prevents the possibility of a \emph{sudden vanishing} of the flow, which Brakke's inequality alone would allow. These solutions are occasionally referred to as varifold-BV solutions (see $\mbox{e.g.} \ $\cite{WSUniq}). 

The free boundary MCF with a smooth, compact, immersed hypersurface initial datum was originally studied by Stahl \cite{Stahl}, until the appearance of singularities. A weak notion of MCF with generalized 90-degree angle condition was then developed in the context of the level set solutions by Sato \cite{SatoLevelSet} and Giga-Sato \cite{GigaSato}. The first notion of free boundary Brakke flow was introduced by Mizuno-Tonegawa \cite{MizunoTonegawa} as the limit of solutions to the parabolic Allen-Cahn equation with Neumann boundary conditions in a strictly convex domain; the convexity assumption was then removed by Kagaya \cite{Kagaya}. Later, Edelen \cite{Edelen} studied many properties of free boundary Brakke flows, such as the compactness of the class, existence of tangent flows and White-type local regularity, and proved an existence result in any codimension by means of the elliptic regularization scheme proposed by Ilmanen \cite{IlmanenElliptic}. Concerning BV solutions, Hensel-Laux \cite{HenselLauxBVAngle} proved the existence of a BV flow with general contact angle (in particular of $90$-degree) in the framework of Allen-Cahn equation, conditional on the assumption of convergence of the energy in the iteration. All the aforementioned approaches yield a two-phase flow; the main advantage of our result is that it produces \textit{multi-phase} solutions to the free boundary Brakke flow that are integer rectifiable up to the boundary hyperplane $H_0$. The analysis of the corresponding problem in general domains is beyond the scope of the present study and it will be addressed in future works.

Additionally, we show that under suitable density ratio assumption on the initial datum, our solution has zero mass on $H_0$ for some short time, providing, to the best of our knowledge, a first partial result in this direction. We will further comment on the possibility of concentration of mass on the free boundary, which is arguably one of the most interesting and challenging problems regarding free boundary Brakke flows, after explaining the general construction of the flow. See \autoref{section: concentration of mass}.

Though somewhat technical, in order to clarify the setting of the problem at this point and to state our main results, we introduce the assumptions on the initial surface $\Gamma_0$.
\begin{assumption} \label{ass:main}
Integers $n\geq 1$ and $N\geq 2$ are fixed. Let 
\begin{gather*}
    H_0\define \set{x=(x_1,\ldots,x_{n+1})\in \R^{n+1} \ \colon \ x_{n+1}=0}\, ,\\
    \R^{n+1}_+\define \set{x=(x_1,\ldots,x_{n+1})\in \R^{n+1} \ \colon \ x_{n+1}>0}\, .
\end{gather*}
Consider the following set of assumptions:
\begin{itemize}
\item[(A1)] $\Gamma_0\subset \R^{n+1}_+$ is a relatively closed, countably $n$-rectifiable set with finite $n$-dimensional Hausdorff measure \footnote{We can actually assume that the surface measure of $\Gamma_0$ is infinite, as long as it grows at most exponentially at infinity; see \autoref{appendix: weight}. }; 
\item[(A2)] $E_{0,1},E_{0,2},\ldots,E_{0,N}$ are non-empty, open and mutually disjoint subsets of $\R^{n+1}_+$, such that $\R^{n+1}_+\setminus\Gamma_0=\bigcup_{i=1}^N E_{0,i}$\, ;
\smallskip
\item[(A3)] The sets 
\begin{equation*}
    B_{0,i}=\set{x\in H_0 \, \colon \, \exists \,  r=r(x)\, \colon \,  B_{r}(x)\cap \R^{n+1}_+ \subseteq E_{0,i} }
\end{equation*}
satisfy 
\begin{equation*}
    \Ha^n\left(H_0 \setminus \bigcup_{i=1}^N B_{0,i}\right)=0\, .
\end{equation*}
\end{itemize}
\end{assumption}
The main results of the present paper can then be roughly stated as follows.
\begin{theorem}
\label{theorem: main thm introduction}
    Under \autoref{ass:main}, there exists a free boundary Brakke flow starting from $\Gamma_0$, consisting of varifolds $\set{V_t}_{t\ge 0}$ which are integer rectifiable up to the boundary $H_0$ for $\mbox{a.e.} \ t$. Furthermore, for every $i=1,\ldots,N$, there exists a flow of grains $\set{E_i(t)}_{t\ge 0}$ starting from $E_{0,i}$, evolving under a generalized free boundary BV law with the same velocity.
\end{theorem}

\begin{theorem}
Under the assumptions of \autoref{theorem: main thm introduction}, if we further assume a suitable uniform density ratio condition on $\Gamma_0$, then there exists a time $T_0>0$ such that the corresponding free boundary Brakke flow is unit density for $\mbox{a.e.}\ t\in [0,T_0)$ and the grains evolve under a free boundary BV law (not merely \emph{generalized}). Moreover, $\norm{V_t}(H_0)=0$ for $\mbox{a.e.}\ t\in [0,T_0)$. 
\end{theorem}
While it is reasonable to expect that $\norm{V_t}(H_0)=0$ for \emph{every} $t\in[0,T_0)$, a rigorous proof remains elusive.

\subsection{Strategy of the proof and technical assumption (A3)}
The idea for obtaining a multi-phase free boundary Brakke flow is to adapt the approximation scheme by Kim-Tonegawa \cite{KimTonegawa} by suitably modifying the approximate velocity. Indeed, grains should move by the usual (smoothed) mean curvature away from the hyperplane $H_0$, and by its tangential component along $H_0$. If we reflect the initial datum across $H_0$, the (smoothed) mean curvature of the symmetrization naturally satisfies these properties, allowing to start the approximation scheme proposed by \cite{KimTonegawa}. On the other hand, one cannot run the scheme of Kim-Tonegawa as a black box, since the other essential step of the approximation, namely the regularization via (volume controlled) area reducing Lipschitz deformation, might disrupt the symmetry at each stage. We must then further modify the scheme imposing the symmetry of the area reducing Lipschitz deformations, so that the symmetry of the grains along the iteration is enforced. This yields the existence of a symmetric Brakke flow (in the usual sense) starting from the symmetrization of the initial datum $\Gamma_0$, where the \emph{free boundary} Brakke flow is obtained, roughly, as its restriction to the (closed) upper halfspace. Due to the symmetry condition imposed on the Lipschitz deformation, a delicate point is to show that the free boundary Brakke flow is integer rectifiable up to the boundary.

The flatness of the ``barrier’’ plays a pivotal role here, as it enables us to globally reflect anything and avoid seeing errors related to its curvature. Another key technical aspect concerns how to precisely define the reflection of the initial grains. Given the grain $E_{0,i}$, its reflection should be defined as $E_{0,i}^S= E_{0,i} \cup \sigma(E_{0,i}) \cup  \widetilde B_{0,i}$, where $\sigma$ is the reflection map across $H_0$ and $\widetilde B_{0,i}\subset H_0$ should intuitively be the common boundary of $E_{0,i}$ and $\sigma(E_{0,i})$, that is the region on $H_0$ that is effectively wetted by $E_{0,i}$. Since the reflection is only a tool to obtain a tangential velocity, the $\widetilde B_{0,i}$ should be chosen wisely in order to make sure that this process does not lead to any subset (of positive $\Ha^n$ measure) of $H_0$ becoming part of the evolving varifolds. The role of assumption (A3) is precisely to allow us to accomplish this task, by choosing $\widetilde B_{0,i}=B_{0,i}$. We show the following pathological example to explain what may happen if (A3) fails.

\begin{example}
Suppose that the initial datum is given as in \autoref{fig: Cantor}: first, consider any smooth bounded curve having boundary on $H_0$. In the region of $\R^2$ delimited by this curve, take the segment $[0,1]\subset H_0$ and begin the construction of a ``fat Cantor’’ set by removing a segment in the middle. Next, use this removed segment as a basis for a rectangle of some small height. Then, continue the construction of the \virgolette{fat Cantor} set, placing a rectangle on each newly removed region, with a height smaller than the previous one. Moreover, assign the same grain label to all such rectangles. By iterating this procedure and carefully choosing the height of the rectangles, one can achieve finite surface measure, obtaining an open partition that is admissible in $\R^2_+$ in the sense of \cite{KimTonegawa}, as it consists of finitely many (not connected) open sets which are mutually disjoint, whose boundaries have finite area (length) and are countably rectifiable in $\R^2_+$. 

Starting from this initial datum, we cannot define any reasonable symmetrization. Indeed, the points belonging to the \virgolette{fat Cantor} set in $H_0$ cannot be in the interior of the symmetrization; thus, the symmetrization would have positive boundary measure on $H_0$, a scenario we want to avoid.
\begin{figure}
    \centering
    \includegraphics[width=0.6\linewidth]{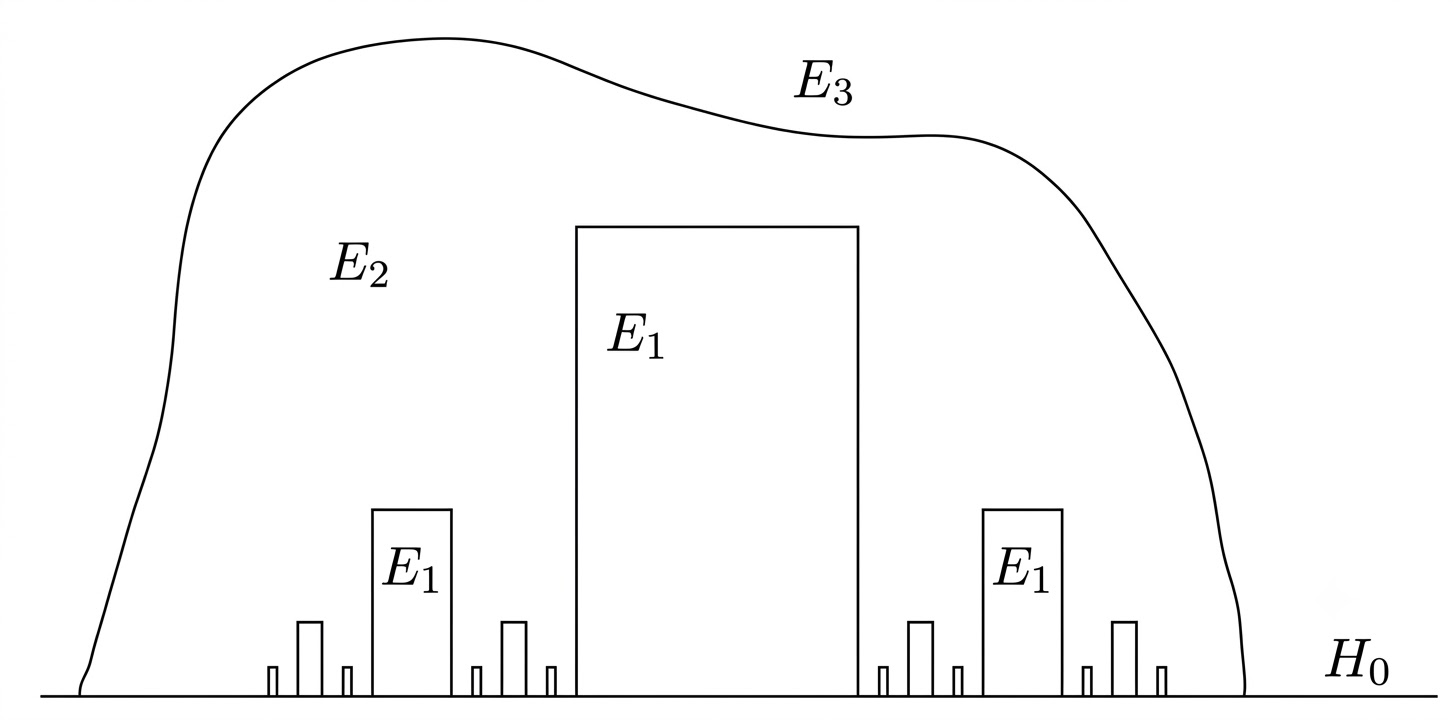}
    \caption{A pathological example}
    \label{fig: Cantor}
\end{figure}
    
\end{example}
Assumption (A3) rules out this possibility, making such an initial datum unacceptable. However, this condition still sounds reasonable, since a point $x\in H_0$ belongs to $H_0 \setminus \bigcup_{i=1}^N B_{0,i}$ if, intuitively, there are multiple grains meeting at $x_0$, and therefore one would expect the set of such points to be of dimension $n-1$.

\subsection{Concentration of mass on the free boundary}
\label{section: concentration of mass}
 The definitions of free boundary varifolds and free boundary Brakke flows allow for the possibility that the free boundary carries positive mass: in fact, forbidding this possibility would prevent the classes from being compact (see, for instance, the example described by Edelen \cite{Edelen}). On the other hand, one could expect that if the initial datum has no mass on the free boundary, this property is preserved over time, possibly under some additional assumptions on the geometry of the domain, such as convexity or mean convexity. This problem is however quite challenging and the possibility of mass approaching the free boundary is actually one of the main reasons we need to restrict our analysis to the halfspace; see \autoref{appendix: non flat domains} for further details. Actually, within the setting of the present paper, we think this mass concentration at the boundary phenomenon cannot be generically excluded. By \cite[Proposition $4.4$]{Edelen}, given any free boundary Brakke flow in the halfspace, its symmetrization is a (boundaryless) Brakke flow in $\R^{n+1}$. Vice versa, in \autoref{section: symmetric Brakke flow} we show that any symmetric Brakke flow in $\R^{n+1}$ is naturally associated to a free boundary Brakke flow in the halfspace. From this perspective the question becomes: can a symmetric Brakke flow concentrate mass on its hyperplane of symmetry after some time? We suspect that the answer to this question might be affirmative, mainly because a hyperplane is a (stable) minimal surface and any MCF is expected to evolve towards some minimal surface. Although this is far from being a proof, we made some simulations using Brakke's surface evolver in $\R^2$ which seem to go in this direction. See \autoref{fig: surface evolver}. The idea is the following: suppose the initial datum is made of a tiny \virgolette{eye} with high curvature, attached to a very large ball, with small mean curvature. We expect the \virgolette{eye} to shrink to a subset of $H_0$ very fast and that the solution carries this region for some time, until the big ball will become very small and the whole solution will shrink to a point (by inclusion principle). At the level of the approximation scheme, if we end up in the scenario shown in the fourth picture of \autoref{fig: surface evolver}, the area-reducing Lipschitz maps would just crush and \virgolette{destroy} that region, since the grains are symmetric and therefore that region would not be part of the reduced boundary of any grain. However, it is not clear to us how to infer that the limit flow actually behaves similarly. In any case, this example should at least show that one cannot expect any uniform-in-time control over the varifold mass around the free boundary. More precisely, for any sufficiently small $\rho>0$, the initial datum (the first picture of \autoref{fig: surface evolver}) has mass of order $\rho$ in a $\rho$--tubular neighborhood of $H_0$, whereas after some time (third picture of \autoref{fig: surface evolver}) the mass in the same $\rho$--tubular neighborhood will be of order $1$.
\begin{figure}[h!]
    \centering
    \begin{subfigure}[b]{0.48\textwidth}
        \centering
        \includegraphics[width=\textwidth]{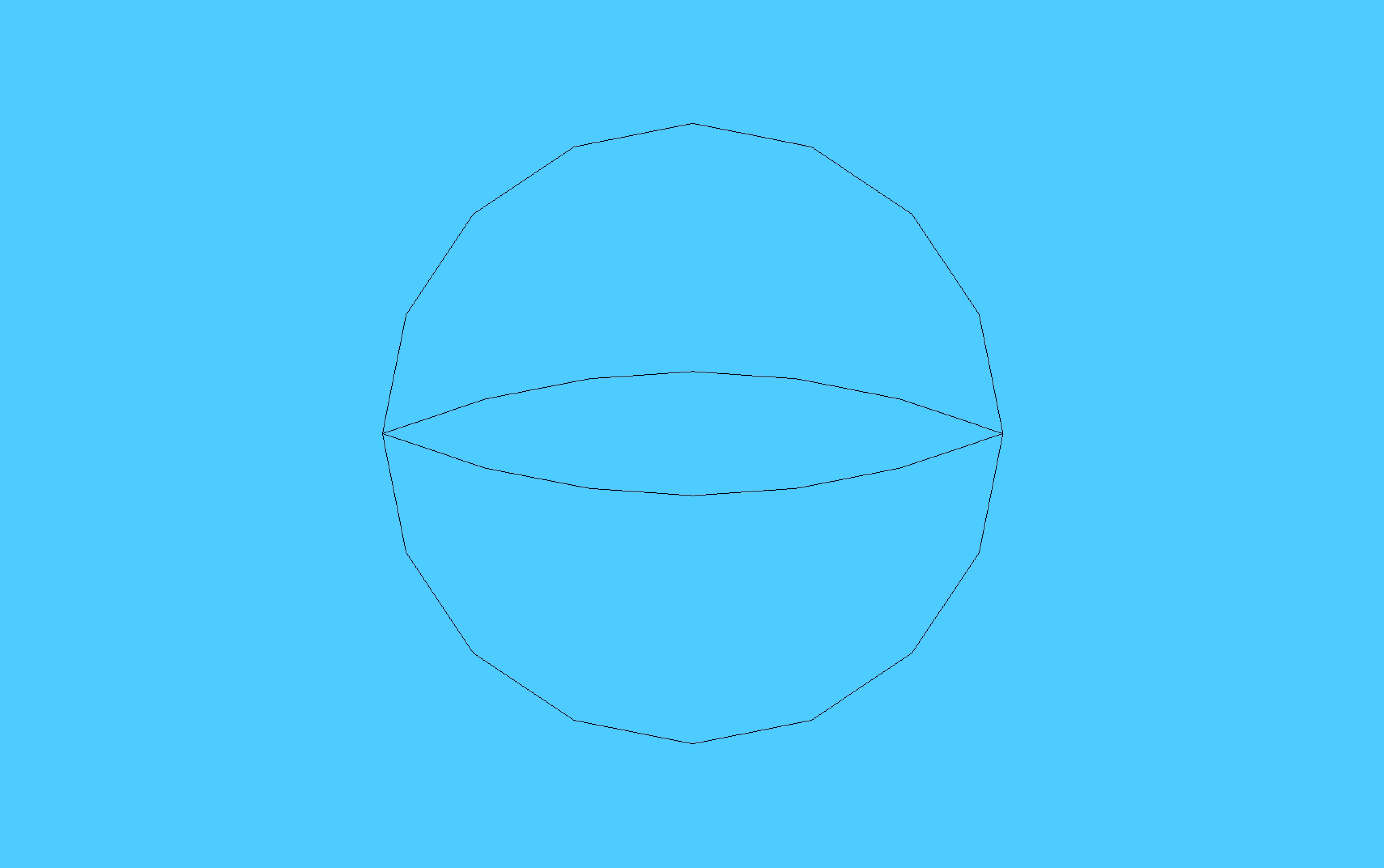}
    \end{subfigure}
    \hfill
    \begin{subfigure}[b]{0.48\textwidth}
        \centering
        \includegraphics[width=\textwidth]{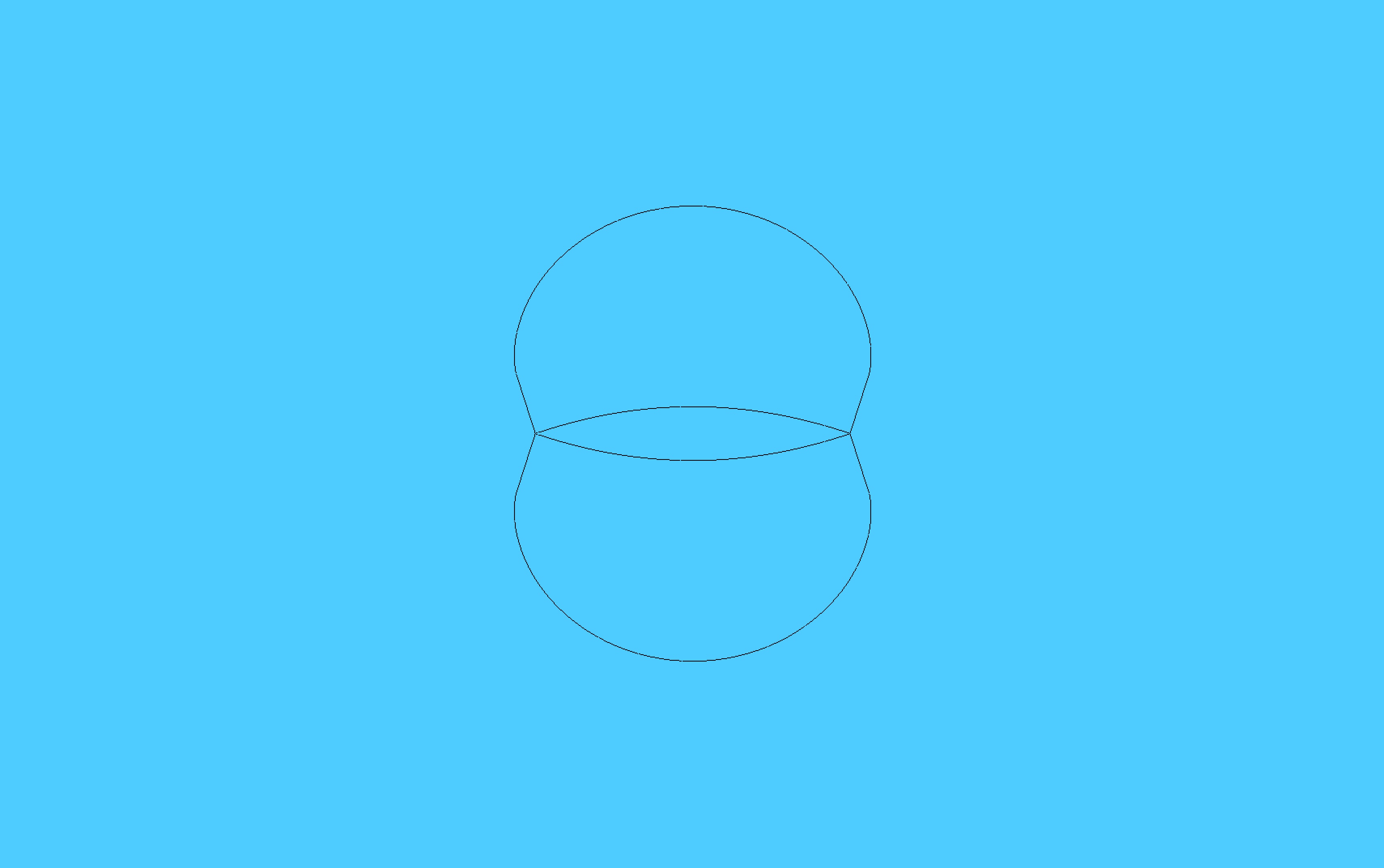}
    \end{subfigure}
\par \vspace{0.3cm}

    \begin{subfigure}[b]{0.48\textwidth}
        \centering
        \includegraphics[width=\textwidth]{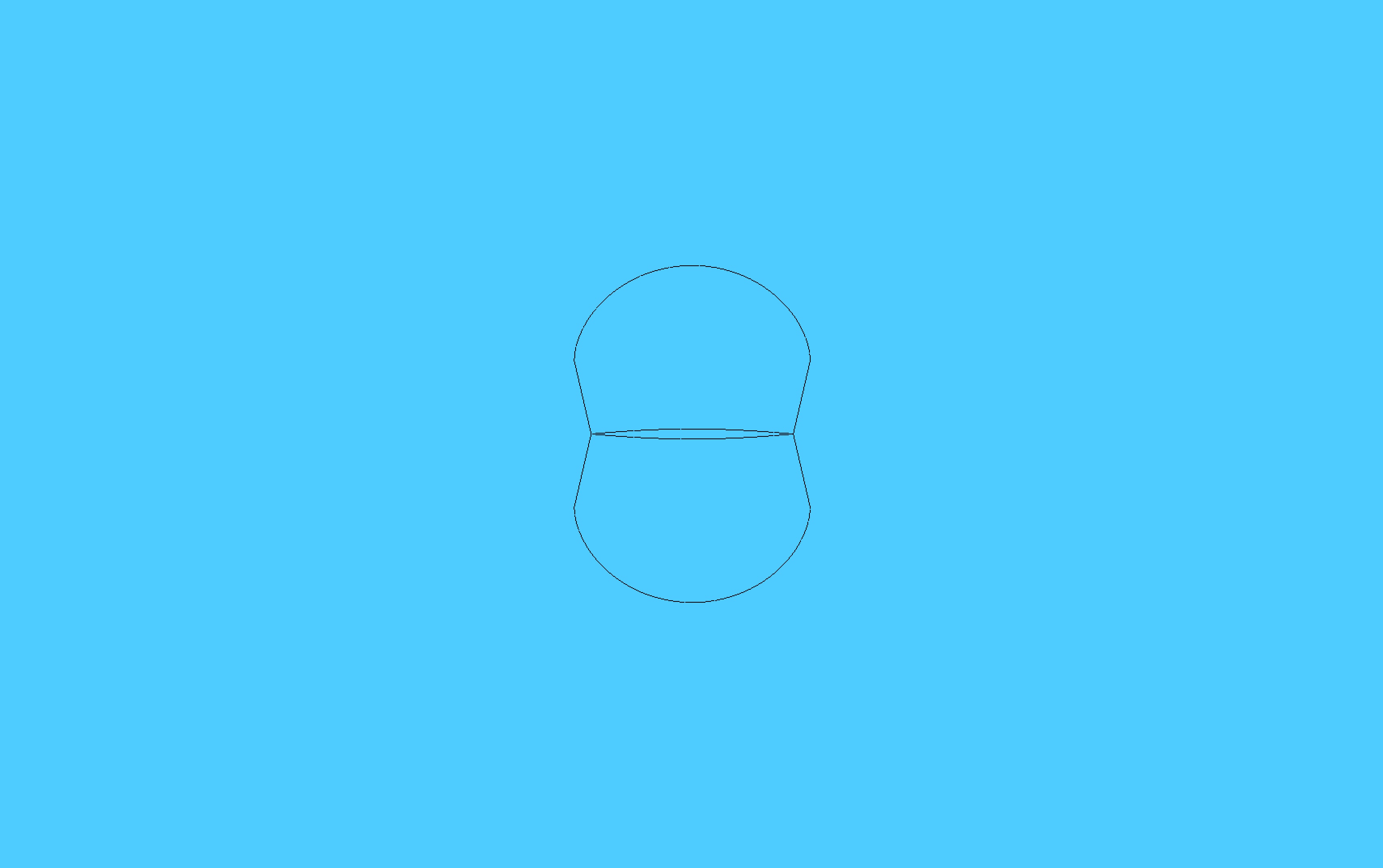}
    \end{subfigure}
    \hfill
    \begin{subfigure}[b]{0.48\textwidth}
        \centering
        \includegraphics[width=\textwidth]{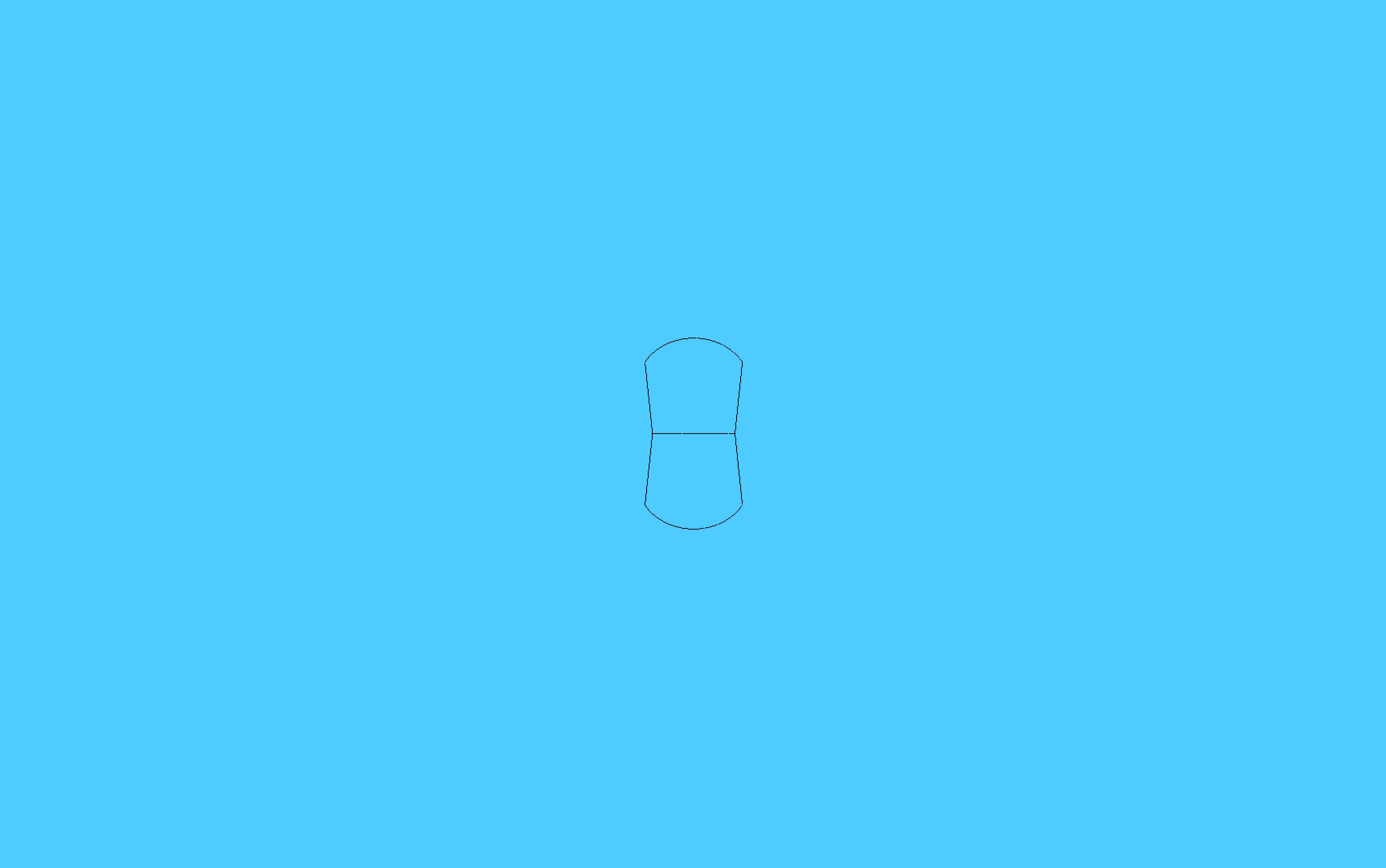}
    \end{subfigure}
    \caption{Evolution of a symmetric network: the \virgolette{eye} region shrinks to a subset of $H_0$ illustrating the potential for mass concentration on the free boundary}
  \label{fig: surface evolver}
    \label{fig:quattro_immagini}
\end{figure}

\section*{Acknowledgments}
The author wishes to express his gratitude to his supervisor Salvatore Stuvard and to Yoshihiro Tonegawa for several useful discussions. This research was carried out during the author's visit to the Institute of Science Tokyo, whose warm hospitality is gratefully acknowledged. This work was supported by the Italian Ministry of University and Research (MUR) through
the FIS 2 project \emph{SiGmA: \virgolette{Singularities in Geometric Analysis:
Minimal Surfaces and Mean Curvature Flows}}, project code FIS-2023-02962 (CUP G53C25000120001). 
\section{Definitions, notations and main results}

Given $x\in \R^{n+1}$ and $r>0$, $U_r(x)$ and $B_r(x)$ denote the open and closed ball of radius $r$ and center $x$. Given $A\subset \R^{n+1}$, $\interior A$ denote the interior of $A$ and $\clos{A}$ or $\overline A$ the closure of $A$. The symbol $\leb$ will denote the Lebesgue measure in $\R^{n+1}$; the symbol $\Ha^k$ will denote the $k$-dimensional Hausdorff measure on $\R^{n+1}$, normalized so that $\Ha^{n+1}=\leb$. Given a measure $\mu$ on $\R^{n+1}$ and a $\mu$-measurable set $\Gamma$, $\mu \mres A$ denotes the restriction of $\mu$ to $\Gamma$.
\subsection{Varifolds}

The symbol $\bG(n+1,k)$ will denote the Grassmannian of (unoriented) $k$-dimensional linear planes in $\R^{n+1}$.  Given $S \in \bG(n+1,k)$, we shall often identify $S$ with the orthogonal projection operator onto it. For any open $U\subseteq \R^{n+1}$, the symbol $\V_k(U)$ will denote the space of $k$-dimensional {\it varifolds} in $U$, namely the space of Radon measures on $\bG_k(U) \define U \times \bG(n+1,k)$ (see \cite{AllardFirstVariation,SimonLectures} for a comprehensive treatment of varifolds). To any given $V \in \V_k(U)$ one associates a Radon measure $\|V\|$ on $U$, called the {\it weight} of $V$, and defined by projecting $V$ onto the first factor in $\bG_k(U)$, explicitly:
\begin{equation*}
\norm V(\varphi) \define  \int_{\bG_k(U)} \varphi(x) \, dV(x,S) \quad \mbox{for every $\varphi \in C_c(U)$}\,.
\end{equation*}
A set $\Gamma \subset \R^{n+1}$ is {\it countably $k$-rectifiable} if it can be covered by countably many Lipschitz images of $\R^k$ into $\R^{n+1}$ up to a $\Ha^k$-negligible set. We say that $\Gamma$ is (locally) {\it $\Ha^k$-rectifiable} if it is $\Ha^k$-measurable, countably $k$-rectifiable, and $\Ha^k\mres\Gamma$ is (locally) finite. If $\Gamma \subset \R^{n+1}$ is locally $\Ha^k$-rectifiable, and $\theta \in L^{1}_{loc}(\Ha^k \mres\Gamma)$ is a positive function on $\Gamma$, then there is a $k$-varifold canonically associated to the pair $(\Gamma,\theta)$, namely the varifold $\var(\Gamma,\theta)$ defined by
\begin{equation} \label{varGammatheta}
\var(\Gamma,\theta)(\varphi) \define \int_\Gamma \varphi(x, T_x \Gamma) \, \theta(x)\, d\Ha^k(x) \quad \mbox{for every } \varphi \in C_c(\bG_k(U))\,,
\end{equation}
where $T_x\Gamma$ denotes the approximate tangent plane to $\Gamma$ at $x$, which exists $\Ha^k$-a.e. on $\Gamma$. Any varifold $V \in \V_k(U)$ admitting a representation as in \eqref{varGammatheta} is said to be \emph{rectifiable}, and the space of rectifiable $k$-varifolds in $U$ is denoted by 
${\bf RV}_k(U)$. If $V = \var(\Gamma,\theta)$ is rectifiable and $\theta$ is integer valued at $\Ha^k$-a.e. $x \in \Gamma$, then we say that $V$ is an \emph{integral} $k$-dimensional varifold in $U$: the corresponding space is denoted $\IV_k(U)$. If $\theta\equiv 1$, $V$ is said to be unit-density. We say that $V\in \V_k(\clos{ U})$ (respectively $\RV(\clos{U})$ or $\IV(\clos{U}$) if $V\in \V_k(\R^{n+1})$ (respectively respectively $\RV(\R^{n+1})$ or $\IV(\R^{n+1})$)  and $\spt\,  V \subseteq \clos{U}$.\\

Given $V\in \V_k(U)$, its first variation is defined as
\begin{equation*}
\delta V(g) = \int_{\bG_k(U)} \nabla g(x) \cdot S \, dV(x,S) \quad \mbox{for every $g \in C^1_c(U;\R^{n+1})$}\,,
\end{equation*}
where, after identifying $S \in \bG(n+1,k)$ with the orthogonal projection operator $\R^{n+1} \to S$,
\[
\nabla g \cdot S = {\rm trace}(\nabla g^T \circ S) = \sum_{i,j=1}^{n+1} S_{ij} \, \frac{\partial g_i}{\partial x_j} = {\rm div}_S\,  g\,.
\]

$\norm{\delta V}$ denotes the total variation of $\delta V$. If the variation $ \delta V$ can be extended to a bounded linear functional on $C_c(U,\R^{n+1})$, we say that $V$ has bounded first variation in $U$. In this case, $\delta V$ is naturally associated with a unique $\R^{n+1}$-valued measure on $U$ by means of the Riesz representation theorem. If such a measure is absolutely continuous with respect to the weight $\norm V$, by the Lebesgue-Radon-Nikod\'ym differentiation theorem, we have
\begin{equation*}
    \delta V (g)=-\int_{\R^{n+1}} \h(V,\cdot) \cdot g \, d\norm V  \quad \mathrm{for \ every \ } g\in C_c(U;\R^{n+1})\, ,
\end{equation*}
for some $\norm V$ measurable and locally $\norm V$-integrable vector field $\h$, called the \emph{generalized mean curvature} of $V$.

\subsection{BV functions and sets of finite perimeter}
Given an open set $U \subseteq \R^{n+1}$, we say that a function $f \in L^1(U)$ has bounded variation in $U$, written $f \in {\rm BV}(U)$, if 
\begin{equation*}
\sup\set{ \int_U f \, {\rm div}\,g \, dx \, \colon \, g \in C^1_c(U;\R^{n+1}) \mbox{ with } \norm{g}_{C^0} \leq 1 } < \infty\,.
\end{equation*}
If $f\in {\rm BV}(U)$, then there exists an $\R^{n+1}$-valued Radon measure on $U$, called measure derivative of $f$ and denoted $\nabla f$, such that
\begin{equation*}
\int_U f\,{\rm div}\,g\, dx = - \int_U g \cdot d\nabla f \qquad \mbox{for all $g \in C^1_c(U;\R^{n+1})$}\,.
\end{equation*}
We say that $f \in {\rm BV}_{{\rm loc}}(U)$ if $f \in {\rm BV}(U')$ for all $U' \ssubset U$.\\

For a set $E \subset \R^{n+1}$, $\chi_E$ is the characteristic (or indicator) function of $E$, defined by $\chi_E(x)=1$ if $x\in E$ and $\chi_E(x)=0$ otherwise. We say that $E$ has locally finite perimeter in $\R^{n+1}$ if $\chi_E \in {\rm BV}_{{\rm loc}}(\R^{n+1})$. When $E$ is a set of locally finite perimeter, then the measure derivative $\nabla \chi_E$ is the associated Gauss-Green measure, and its total variation $\norm{\nabla \chi_E}$ is the perimeter measure; by De Giorgi's structure theorem, $\norm{ \nabla \chi_E} = \Ha^n \mres{\pa^\ast E}$, where $\pa^\ast E$ is the reduced boundary of $E$, and $\nabla\chi_E = - \nu_E\, \norm{\nabla \chi_E} = - \nu_E\, \Ha^n \mres{\pa^\ast E}$, where $\nu_E$ is the outer pointing unit normal vector field to $\pa^\ast E$. We refer to \cite{MaggiSets, AFP} for a complete treatment.

\subsection{Free boundary varifolds}
\begin{definition}
Let $U$ be an open set with boundary $\partial U$ of class $C^2$. We define
\begin{equation*}
 \T{\partial U}\define \set{g\in C_c(\R^{n+1}; \R^{n+1}) \, \colon \, g(x)\cdot \nu_{\partial U} (x)=0\ \mathrm{for \ every }\ x\in \partial U}\, ,
\end{equation*}
where $\nu_{\partial U}$ is the outer unit normal to $\partial U$. We say that a varifold $V\in \V_k(\clos U )$ has {\it free boundary} at $\partial U$ if there exists a $\norm V$-measurable vector field $\h^\fb\in L^1_{\operatorname{loc}}(\norm V; \R^{n+1})$ which is $\norm V$-$\mbox{a.e.}$ tangent to $\partial U$ such that
    \begin{equation*}
    \delta V(g) =-\int_{\overline U} \h^\fb \cdot g \, d\norm V \quad \mathrm{for \ every \ } \,  g\in \T{\partial U} \cap C^1_c(\R^{n+1};\R^{n+1})\, ,
    \end{equation*}
    that is, if 
    \begin{equation}
    \label{e:free boundary varifold with tangent vectors}
    \begin{split}
     &\int_{\bG_k(\overline U)} S\cdot \nabla g \, dV(x,S)=-\int_{\overline U} \h^\fb \cdot g \, d\norm V \\
     & \qquad \mathrm{for \ every \ } \,  g\in \T{\partial U} \cap C^1_c(\R^{n+1};\R^{n+1})\, .
     \end{split}
    \end{equation}   
    We call such a $\h^\fb$ the \emph{generalized free boundary mean curvature} of $V$.
\end{definition}
The condition that $\h^\fb$ is $\norm V$-a.e. tangent to $\partial U$ implies the uniqueness of the free boundary mean curvature whenever it exists. Indeed, otherwise every sum $\h^\fb+K$ with $K$ orthogonal to $\partial U$ and identically $0$ in $U$ would satisfy \eqref{e:free boundary varifold with tangent vectors}.

Notice that this definition allows for the possibility that $\norm V(\partial U) >0$, even in the smooth case. Otherwise, this class would not be compact. This allows some counterintuitive phenomena: for instance, $S^1$ counted with constant multiplicity is a varifold having free boundary at itself, with $\h^\fb=0$.

We refer to \cite{DeMasiRectifiability} for some motivation for this weak notion of orthogonality and for further properties.

\subsection{(Standard) Brakke flows}
\begin{definition}[Brakke flow]
Let $0 < T \leq \infty$, and let $U \subset \R^{n+1}$ be an open set. A \emph{$k$-dimensional Brakke flow} in $U$ is a one-parameter family of varifolds $\set{V_t}_{t\in\left[0,T\right)}$ in $U$ such that all of the following hold: 
\begin{enumerate}
\item[(a)]
For a.e.~$t\in \left[0,T\right)$, $V_t\in{\bf IV}_k(U)$;
\item[(b)] For a.e.~$t\in \left[0,T\right)$, $\delta V_t$ is locally bounded and absolutely continuous with respect to $\norm{V_t}$;
\item[(c)] The generalized mean curvature $\h(\cdot,V_t)$ (which exists for a.e.~$t$ by (b)) satisfies $\h(\cdot,V_t) \in L^2_{{\rm loc}}(\norm{V_t};\R^{n+1})$, and for every compact set $K \subset U$ and for every $t < T$ it holds $\sup_{w \in \left[0,t\right]} \norm{V_w} (K) < \infty$;
\item[(d)] $\h(V_t,\cdot)\in \eltwoloc(\norm{V_t}\times dt)$;
\item[(e)]
For all $0\leq t_1<t_2<T$ and $\varphi\in C_c^1(U\times\left[0,T\right);\mathbb R_+)$, it holds
\begin{equation}
\label{brakineq}
\begin{split}
&\norm{V_{t_2}}(\varphi(\cdot,t_2)) - \norm{V_{t_1}}(\varphi(\cdot,t_1)) \\ 
&\qquad\leq \int_{t_1}^{t_2} \int_{U} \left( - \varphi (x,t) \,\abs{\h(x,V_t)}^2 + \nabla\varphi(x,t) \cdot \h(x,V_t) + \frac{\partial\varphi}{\partial t}(x,t) \right)    d\norm{V_t}(x)\,dt\,.
\end{split}
\end{equation}
\end{enumerate}
\end{definition}

The inequality \eqref{brakineq} is typically referred to as \emph{Brakke's inequality}.It provides an elegant weak formulation of the velocity: indeed, if a family of smooth surfaces $\set{\Gamma(t)}_{t\ge 0}$ satisfies \eqref{brakineq} with respect to the naturally associated weight measures, then $\set{\Gamma(t)}_{t\ge 0}$ is an MCF in the classical sense. See, for instance, \cite[Proposition $2.1$]{TonegawaBook}.

\subsection{Free boundary Brakke flows}
\begin{definition}
Let $0 < T \leq \infty$, and let $U \subset \R^{n+1}$ be an open set with $\partial U$ of class $C^2$. A \emph{$k$-dimensional free boundary Brakke flow} in $U$ is a one-parameter family of varifolds $\set{V_t}_{t\in\left[0,T\right)}\in \V_k(\clos{U})$ such that all of the following hold:
  \begin{enumerate}
\item[(a)]
For a.e.~$t\in [0,T)$, $V_t\in \IV_k(\clos{U})$;
\item[(b)]
For a.e.~$t\in [0,T)$, $V_t$ has free boundary at $\partial U$;
\item[(c)] The generalized free boundary mean curvature $\h^\fb(\cdot,V_t)$ satisfies $\h^\fb(\cdot,V_t) \in L^2_{{\rm loc}}(\norm{V_t};\R^{n+1})$, and for every compact set $K\subset \clos{U}$ and for every $t<T$ it holds $\sup_{w\in[0,t]} \norm{V_w}(K)<\infty$.
\item[(d)] $\h^\fb(V_t,\cdot)\in \eltwoloc(\norm{V_t}\times dt)$;
\item[(e)]
For all $0\leq t_1<t_2<T$ and $\varphi\in C_c^1(\R^{n+1}\times[0,T);\R_+)$ with $\nabla \varphi(\cdot, t) \in \T{\partial U}$,
\begin{equation}
\label{brakineqfb}
\begin{split}
\norm{V_{t_2}}(\varphi&(\cdot,t_2)) - \norm{V_{t_1}}(\varphi(\cdot,t_1))\\
&\leq \int_{t_1}^{t_2}\int_{\R^{n+1}}\left( -\varphi(x,t) \abspiccolo{\h^\fb(x,V_t)}^2+\nabla \varphi (x,t) \cdot \h^\fb (x,V_t) + \frac{\partial \varphi}{\partial t}(x,t) \right)\, d\norm{V_t}(x)\, dt\, .
\end{split}
\end{equation}

\end{enumerate}
While \eqref{brakineq} is unaffected by the behavior of the varifolds on $\partial U$, \eqref{brakineqfb} is sensitive to it. The tangentiality of the test functions encodes the $90$-degree angle condition. 
\end{definition}
\noindent \cite[Proposition 4.2]{Bao2026} shows that actually any \emph{integral} free boundary Brakke flow satisfies Brakke's inequality even for test functions whose gradient do not satisfy the tangentiality condition. However, this holds only when the integrality of the flow is known a priori, as the result relies on Brakke's perpendicularity theorem (see \cite[Section $5.8$]{Brakke}).
\subsection{Free boundary BV flow}
We introduce another weak notion of MCF, which is related to the motion of hypersurfaces given as the boundaries of a finite collection of sets of finite perimeter. Given a domain $U$ and $N\ge 2$, we say that $\set{E_1,\ldots,E_N}$ is an $\mathcal L^{n+1}$-partition of $\R^{n+1}$ if $E_i\subseteq U$ for every $i$, they are pairwise disjoint, and $\mathcal L^{n+1}(U\setminus (\cup_{i=1}^N E_i))=0$.
\begin{definition}
Suppose $N \ge 2$ is an integer, and let $0 < T \leq  \infty$. $N$ one-parameter families $\{E_i (t)\}_{t\in\left[0,T\right)}$ ($i=1,\ldots,N$) identify a free boundary $\mathrm{BV}$ solution for multi-phase MCF in $U$ if all of the following hold:
\begin{enumerate}
    \item For a.e.~$t \in \left[0,T\right)$, $\{E_1(t),\ldots,E_N(t)\}$ is an $\Leb^{n+1}$-partition of $U$, $E_i(t)$ is a set of locally finite perimeter, and, setting $I_{i,j}(t)\define\pa^\ast E_i (t)\cap \pa^\ast E_j(t)\cap U$ for $i \neq j$, 
    \begin{equation*}
        \esssup_{t \in \left[0,T\right)} \sum_{i,j=1,\,i \neq j}^N \Ha^n (I_{i,j}(t)) < \infty\,;
    \end{equation*}
    \item There exist scalar functions $v_1,\ldots,v_N$ such that
    \begin{equation*}
        \int_0^T \int_{\pa^\ast E_i(t)\cap U} |v_i(x,t)|^2\,d\Ha^n(x)\,dt < \infty \quad \mbox{for every $i$}\,,
    \end{equation*}
    and with the property that
    \begin{equation*}
        \begin{split}
& \left.\int_{E_i(t)} \varphi(x,t) \, dx\right|_{t=t_1}^{t_2} =  \int_{t_1}^{t_2} \int_{E_i(t)} \frac{\partial\varphi}{\partial t}(x,t) \, dx\,dt  
\\
& \qquad+ \int_{t_1}^{t_2} \int_{\pa^\ast E_i(t)\cap U} \varphi(x,t) \, v_i(x,t)\,d\Ha^n(x)\,dt 
\end{split}
    \end{equation*}
    for a.e.~$0 \le t_1 < t_2 < T$ and for all $\varphi \in C^1_c(\R^{n+1} \times \left[0,T\right))$ having $\nabla \varphi (\cdot, t)\in \T{\partial U}$;
    \item Setting $\nu_i(x,t) = \nu_{E_i(t)}(x)$ for the outer unit normal to the reduced boundary of $E_i(t)$ at $x$, it holds 
    \begin{equation*}
        v_i(\cdot,t)\,\nu_i (\cdot,t) = v_j(\cdot, t)\,\nu_j(\cdot,t) \quad \mbox{$\Ha^n$-a.e.~on $I_{i,j}(t)$, for a.e. $0 \leq t < T$}\,;
    \end{equation*}
    \item The functions $v_i$ further satisfy
    \begin{equation*} 
        \sum_{i\neq j} \int_0^T \int_{I_{i,j}(t)} {\rm div}\,g - (\nu_i \otimes \nu_i) \cdot \nabla g\,d\Ha^n\,dt = - \sum_{i \neq j} \int_0^T \int_{I_{i,j}(t)} v_i \, \nu_i \cdot g \,  d\Ha^n\,dt
    \end{equation*}
    for all vector fields $g \in C^1_c(\R^{n+1}\times \left[0,T\right];\R^{n+1})\cap \T{\partial U}$;
\item The following inequality holds for 
a.e.~$0\leq t< T$:
\begin{equation*}
    \sum_{i,j=1,\,i \neq j}^N \Ha^n(I_{i,j}(t)) + \sum_{i,j=1,\,i \neq j}^N \int_0^t\int_{I_{i,j}(s)} |v_i(x,s)|^2\,d\Ha^n(x)\,ds \leq \sum_{i,j=1,\,i \neq j}^N \Ha^n(I_{i,j}(0))\,.
\end{equation*}
\end{enumerate}
If only $(1)-(3)$ are verified, we speak about \it{generalized} free boundary $BV$ flow in $U$.
\end{definition}

\subsection{Main results}
We are now ready to state precisely the main results of the present paper. 
\begin{theorem}
    \label{thm: main theorem precise}
    Let $\Gamma_0$ and $E_{0,1},\ldots, E_{0,N}$ be as in \autoref{ass:main}. There exists an $n$-dimensional Brakke flow $\set{V_t}_{t\ge 0}$ in $\R^{n+1}_+$ with free boundary at $H_0$ such that the following assertions hold:
    \begin{itemize}
    \item[(1)] $\norm{V_0}= \Ha^n\mres{\Gamma_0}$;
    \item[(2)] If $\Ha^n(\bigcup_{i=1}^N (\partial E_{0,i} \setminus \partial^\ast E_{0,i}))=0$, then $\lim\limits_{t\to 0+}\norm{V_t}=\norm{V_0}$.
    \item[(3)] $\displaystyle \norm{V_t}(\R^{n+1})+\int_0^t \int_{\R^{n+1}} \abspiccolo{\h^\fb(\cdot, V_s)}^2 \, d\norm{V_s}\, ds \leq \Ha^n(\Gamma_0) \quad \mbox{for all} \ t>0$.
    \end{itemize}
    Moreover, for each $i=1,\ldots,N$, there exists a family of open sets $\set{E_i(t)}_{t\ge 0}$ in $\R^{n+1}_+$ such that, setting 
    \begin{equation*}
        \Gamma(t)\define \R^{n+1}_+ \setminus \bigcup_{i=1}^N E_i(t)\,,
    \end{equation*}
    the following assertions hold:
    \begin{itemize}
        \item[(i)] $E_i(0)=E_{0,i}$ for $i=1,\ldots,N$.
        \item[(ii)] $E_1(t),\ldots, E_N(t)$ are pairwise disjoint for $t\in \R_+$.
        \item[(iii)] For every $0<T<\infty$, the families $\set{E_i(t)}_{t\in[0,T)}$ define a generalized free boundary $\mathrm{BV}$ flow in $\R^{n+1}_+$, with scalar velocities given by
        \begin{equation*}
            v_i(x,t)= \h^\fb(x,V_t)\cdot \nu_i(x,t)\,,
        \end{equation*}
        where $\nu_i(x,t)$ is the outer normal to the reduced boundary of $E_i(t)$ at the point $x$.
        \item[(iv)] If $V_t$ is a unit density flow in $\R^{n+1}_+\times (t_1,t_2)$, then $\set{E_i(t)}_{t\in (t_1,t_2)}$ actually defines a free boundary $\mathrm{BV}$ flow in $\R^{n+1}_+$.
    \end{itemize}
\end{theorem}
In the following statement, we adopt the notation $\sigma$ for the reflection map about $H_0$ (see \eqref{e: def sigma})
\begin{theorem}
\label{theorem: unit density precise}
     Under \autoref{ass:main}, further assume that there exist $r_0>0$ and $\delta_0>0$ such that 
\begin{equation}
\label{e: assumption density ratio}
    \sup_{x\in \R^{n+1}, \, 0<r<r_0} \frac{\Ha^n(\Gamma_0\cap B_r(x))+ \Ha^n(\Gamma_0\cap B_r(\sigma(x))}{\omega_n \, r^n}<2-\delta_0\, .
\end{equation}
Then, there exists $T_0=T_0(n,r_0,\delta_0, \Ha^n(\Gamma_0))\in (0,\infty)$ such that $\set{V_t}_{t\in[0,T_0)}$ in \autoref{theorem: main thm introduction} is unit density for $\mbox{a.e.} \ t\in[0,T_0)$ and the families $\set{E_i(t)}_{t\in[0,T_0)}$ define a free boundary $BV$ flow. Moreover, we have $\norm{V_t}(H_0)=0$ for $\mbox{a.e.}\ t\in [0,T_0)$.
\end{theorem}
\begin{remark}
    The numerator of \eqref{e: assumption density ratio} can be equivalently written as 
    \begin{equation*}
    \Ha^n\left( \left(\Gamma_0 \cup \sigma(\Gamma_0) \right) \, \cap B_r(x)\right)\, ,
    \end{equation*}
   meaning we are actually prescribing a uniform $2-\delta_0$ density ratio on $\Gamma_0 \cup \sigma(\Gamma_0).$ Thus, $\Gamma_0$ must satisfy a uniform $2-\delta_0$ density ratio when it is far from $H_0$, a uniform $1-\frac{\delta_0}{2}$ density ratio on $H_0$ and, intuitively, it cannot approach the hyperplane ``too tangentially’’.
\end{remark}
\section{Preliminaries}
\subsection{Reflections and symmetry}
Set $H_0=\set{x_{n+1}=0}$, $\R^{n+1}_+\define \set{x_{n+1}>0}$ and $\R^{n+1}_- \define \set{x_{n+1}<0}$. We also introduce 
\begin{equation*}
    \closedhalf\define \clos{\R^{n+1}_+}=\set{x_{n+1}\ge 0}\, .
\end{equation*}
For any $A\subset \R^{n+1}$, we also define 
\begin{equation*}
    A_+\define A \cap \R^{n+1}_+ \ \mbox{and} \ A_-\define A \cap \R^{n+1}_-\, .
\end{equation*}
We will sometimes use the symbol $A^+$ or $A^-$ instead, depending on the presence of multiple subscripts. Let $\d$ denote the signed distance function from $H_0$, that is
\begin{equation*}
    \d(x)= \begin{cases*}
        \mbox{dist}(x,H_0) &if $x_{n+1}\ge 0$\\
        -\mbox{dist}(x,H_0) &if $x_{n+1}< 0$
    \end{cases*}
\end{equation*}
so that $\d$ is a linear function and $\nabla \d(x)$ is the constant vector $e_{n+1}$. Let $\pi(x)$ be the nearest point projection of $x\in \R^{n+1}$ onto $H_0$, namely \begin{equation*}
\pi(x_1,\ldots,x_n,x_{n+1})\define x-\d (x) \nabla \d(x)= (x_1,\ldots,x_n,0)\, .
\end{equation*} We define the reflection across $H_0$ of a point $x$ to be 
\begin{equation}
\label{e: def sigma}
    \sigma (x)\define2\pi(x)-x=x-2 \d(x)\, \nabla \d(x)\, ;
\end{equation}
in coordinates, $\sigma(x_1,\ldots,x_n,x_{n+1})=(x_1,\ldots,x_n,-x_{n+1})$. We denote by $Q\define D \sigma= \mbox{id}-2\, \nabla \d\otimes \nabla \d$, a constant orthogonal matrix such that $\sigma(x)=Q\left[x\right]$.
We give the notion of symmetry for several objects. The dependence on the plane $H_0$ is henceforth implicit.
\begin{definition}
\label{def: symmetryc objects}
\begin{itemize}
    \item[(a)]A set $A\subseteq \R^{n+1}$ is said to be symmetric if $x\in A$ if and only if $\sigma(x)\in A$.
\item[(b)] Given a vector field $g\in C(\R^{n+1}, \R^{n+1})$, we define its reflection as $g^S(x)\define \sigma(g(\sigma(x)))$ for every $x\in \R^{n+1}$. $g$ is said to be symmetric if $g(x)=g^S(x)$ for every $x\in \R^{n+1}$, that is 
\begin{equation*}
  g(x)=\sigma(g(\sigma(x))) \ \mbox{for all} \ x\in \R^{n+1}\, ,
\end{equation*}
or, equivalently,
\begin{equation*}
    \sigma(g(x))=g(\sigma(x)) \ \mbox{for all} \ x\in \R^{n+1}\, ;
\end{equation*}
note that this implies that $g$ is tangential along $H_0$.
\item[(c)] A Radon measure $\mu^S$ in $\R^{n+1}$ is said to be symmetric if $\sigma_\sharp (\mu^S)=\mu^S$. This is equivalent to say that there exists another Radon measure $\mu$ supported on $\closedhalf$ such that $\mu^S=\mu+\sigma_\sharp \mu$. Indeed, given $\mu^S$, one can define $\mu=\mu^S\mres{\R^{n+1}_+}+\frac12\mu^S\mres{H_0}$ that satisfies the described property. We analogously define symmetric varifolds. Given a symmetric varifold $V^S\in \V_n(\R^{n+1})$ such that $V^S=V+\sigma_\sharp V$ for $V\in \V_n(\closedhalf)$, for any function $f\in C_c(\bG_n(\R^{n+1}))$, we have
\begin{equation*}
    \sigma_\sharp V (f) = \int_{\bG_n(\R^{n+1})} f(\sigma(y), Q\,S\, Q) \, dV(y,S)
\end{equation*}
and then 
\begin{equation}
\label{e: integration on the reflection}
    V^S(f)= \int_{\bG_n(\R^{n+1})} f(y,S)+f(\sigma(y), Q\,S\, Q) \, dV(y,S)\, ,
\end{equation}
where we are making an abuse of notation by denoting by the same symbol the varifold measure and the associated linear functional on the Grassmanian. \\
Given such a $V$, we also refer to $V^S$ as its \emph{symmetrization}.
    \end{itemize}
\end{definition}
\begin{lemma}
\label{lemma: limit of symmetric is symmetric}
    Let $\setpiccolo{\mu^S_k}_{k=1}^\infty$ be a sequence of symmetric Radon measures on $\R^{n+1}$ and let $\mu^S$ be a Radon measure on $\R^{n+1}$ such that $\mu_k^S\rightharpoonup \mu^S$ in the sense of measures. Then, $\mu^S$ is symmetric.
    \begin{proof}
        Let $\varphi \in C_c(\R^{n+1})$. By the symmetry of $\mu_k^S$ and since $\varphi(\sigma(\cdot))$ is still a compactly supported continuous function, we have
        \begin{equation*}
            \sigma_\sharp (\mu_k^S)(\varphi)=\mu_k^S(\varphi(\sigma))\longrightarrow \mu^S(\varphi(\sigma))=(\sigma_\sharp \mu^S)( \varphi)\, ,
        \end{equation*}
        that is the thesis.
    \end{proof}
\end{lemma}
\subsection{Classes of symmetric test functions and vector fields}

Define, for every $j \in \Na$, the classes $\cA_j$ and $\cB_j$ as follows:

\begin{equation*} 
\begin{split}
\cA_j \define \{ \varphi &\in C^2(\R^{n+1}; \R^+) \, \colon \, \varphi(x) \leq 1, \, \abs{\nabla \varphi(x)} \leq j\, \varphi(x),\\
& \normpiccolo{\nabla^2\varphi(x)} \leq j \ \mbox{and} \ \varphi(x)=\varphi(\sigma(x)) \ \mbox{for every} \  x\in \R^{n+1} \} ,
\end{split}
\end{equation*}
\begin{equation*} 
\begin{split}
\cB_j \define \{ g &\in C^2(\R^{n+1}; \R^{n+1}) \, \colon \, \abs{g(x)} \leq j,\,\, \abs{\nabla g(x)} \leq j\, , \\ 
& \normpiccolo{\nabla^2 g(x)} \leq j\, , \sigma(g(x))=g(\sigma(x)) \ \mbox{for every} \ x \in \R^{n+1} \ \mbox{and} \ \norm{g}_{L^2} \leq j\}\,.
\end{split}
\end{equation*}
The properties of functions $\varphi \in \cA_j$ and vector fields $g \in \cB_j$ are precisely as in \cite[Lemma 4.6, Lemma 4.7]{KimTonegawa}, plus the tangentiality along $H_0$ obtained by symmetry. We record them all in the following lemma.

\begin{lemma} 
Let $x,y \in \R^{n+1}$ and $j \in \Na$. For every $\varphi \in \cA_j$, the following properties hold:
\begin{gather}
\varphi(x)  \leq \varphi(y) \exp(j\, \abs{x-y})\,, \label{e:Gronwall} \\
\nonumber \abs{\varphi(x) - \varphi(y)} \leq j \, \abs{x-y} \varphi(y) \exp(j\, \abs{x-y})\,,  \\
\nonumber \abs{\varphi(x) - \varphi(y) - \nabla\varphi(y) \cdot (x-y)} \leq j \abs{x-y}^2 \varphi(y) \exp(j\, \abs{x-y}) \,,\\
\nonumber \nabla \varphi(x) \cdot \nabla d(x) =0 \ \mbox{for} \ x\in H_0\, .
\end{gather}

Also, for every $g \in \cB_j$:
\begin{gather*}
\abs{g(x) - g(y)} \leq j\, \abs{x-y}\, ,\\
g(x)\cdot \nabla d(x)=0 \ \mbox{for} \ x\in H_0\, .
\end{gather*}

\end{lemma}
\subsection{Open partitions} 
Let $U\subseteq \R^{n+1}$ be an open set.
\begin{definition}\label{def:op}
For $N \geq 2$, an  \emph{open partition} of $U$ in $N$ elements is a finite and ordered collection $\E = \{E_{i}\}_{i=1}^N$ of subsets $E_{i} \subseteq U$ such that:
\begin{itemize}
\item[(a)] $E_1,\ldots,E_N$ are open and mutually disjoint;
\item[(b)] $\Ha^n(U \setminus \bigcup_{i=1}^N E_i) < \infty$;
\item[(c)] $\left(\bigcup_{i=1}^N \partial E_i \right)\cap U$ is countably $n$-rectifiable.
\end{itemize}
\noindent
The set of all open partitions of $U$ of $N$ elements will be denoted $\op^N(U)$. If $U=\R^{n+1}$, we simply denote the class by $\opn.$

If $U=\R^{n+1}$ and $\E$ further satisfy
\begin{itemize}
    \item [(d)] $E_i$ is symmetric for each $i=1,\ldots,N$,
\end{itemize}
we say that $\E$ is a \emph{symmetric} open partition. This class will be denoted by $\opns(\R^{n+1})$ or simply $\opns$.
\end{definition}
\noindent
Note that some of the $E_i$ may be empty and that we are not assuming them to be connected. Condition $(c)$ is not redundant, since in general only the \emph{reduced} boundary of the set $E_i$ is countably rectifiable, but this may differ from the topological boundary. 
\begin{notation}
Given $\E =\set{E_i}_{i=1}^N\in \op^N(U)$, we set
\begin{equation*} 
\partial\E \define \var\left( \left(\bigcup_{i=1}^N \partial E_i \right) \cap U,\,  1\right) \in \IV_{n}(\R^{n+1})\,.
\end{equation*}
\end{notation}
\noindent We introduce the notion of symmetrization for a partition of the halfspace.
\begin{definition}
\label{def: symmetrization of grains}
Let $\E=\set{E_i}_{i=1}^N\in \op^N(\R^{n+1}_+)$. We define the sets 
\begin{equation}
\label{e: def Bi}
    B_{i}\define\set{x\in H_0 \, \colon \, \exists \,  r= r(x)\, \colon \,  B_r(x)\cap \R^{n+1}_+ \subseteq E_{i} }
\end{equation}
and the \emph{symmetrization} of $\E$ as the family of sets $\E^S\define\setpiccolo{E_i^S}_{i=1}^N$, where 
\begin{equation*}
E_i^S\define E_i \cup \sigma (E_i) \cup B_i\, .
\end{equation*}
\end{definition}
\begin{lemma}
\label{lem: symmetrization is opn}
    Given $\E=\set{E_i}_{i=1}^N\in \opn(\R^{n+1}_+)$, let $\E^S$ be the family of sets defined in \autoref{def: symmetrization of grains}. Then, $\E^S\in \opns(\R^{n+1})$, provided $\Ha^n\left(H_0 \setminus \left(\bigcup_{i=1}^N B_i\right)\right)<\infty$.
    \begin{proof}
    The family $\E^S$ is pairwise disjoint by definition. We show that each $E_i^S$ is open and, as a first step, we show that $B_i$ is open in the topology of $H_0$. In particular, we claim that for any given $x\in B_i$, points of $H_0$ sufficiently close to $x$ are also contained in $B_i$. Indeed, let $x\in B_i$, $r=r(x)$ be as in \eqref{e: def Bi}. If $y\in H_0$ is such that $\abs{y-x}<\frac r2 $, then $B_{\frac r2}(y) \cap \R^{n+1}_+\subset B_{r}(x) \cap \R^{n+1}_+\subseteq E_i,$ yielding the claim. Let now $z\in E_i^S$; we need to prove that there exists a ball $U_\rho(z) \subseteq E_i^S$. If $z\in E_i$ or $z\in \sigma(E_i)$, such ball must exist since these are open sets. If $z\in B_i$, by definition we have that for sufficiently small $\rho$ (depending on $z$), $B_\rho(z)\cap \R^{n+1}_+\subseteq E_i;$ thus, it must be $B_\rho (z) \cap \R^{n+1}_- \subseteq \sigma(E_i)$. The previous claim shows that $B_\rho (z) \cap H_0 \subseteq B_i$, up to possibly halving $\rho$. This shows that $U_\rho (z) \subseteq E_i^S$, that is (a) of \autoref{def:op}.

    We want to show (b) of \autoref{def:op}, that is the boundary measure is finite. We compute
    \begin{equation}
    \label{e: definition symmetrization proof total mass}
        \begin{split}
            \Ha^n&\left( \R^{n+1} \setminus \bigcup_{i=1}^N E_i^S \right) \leq \Ha^n \left(\R^{n+1}_+ \setminus  \bigcup_{i=1}^N E_i^S\right) + \Ha^n\left( \R^{n+1}_- \setminus \bigcup_{i=1}^N E_i^S\right)+\Ha^n\left( H_0\setminus \bigcup_{i=1}^N E_i^S\right)\\
            &= \Ha^n \left(\R^{n+1}_+ \setminus  \bigcup_{i=1}^N E_i\right) + \Ha^n\left( \R^{n+1}_- \setminus \bigcup_{i=1}^N \sigma(E_i)\right)+\Ha^n\left( H_0\setminus \bigcup_{i=1}^N B_i\right)<\infty\, ;
        \end{split}
    \end{equation}
    the first term is finite due to $\E\in \opn(\R^{n+1}_+)$ and it equals the second one since $\sigma$ is a diffeomorphism. Finally, the latter is finite by assumption.

    \noindent The countably $n$-rectifiablity can be showed decomposing $\bigcup_{i=1}^N \partial E_i^S$ as in \eqref{e: definition symmetrization proof total mass}.
    \end{proof}
\end{lemma}

\subsection{Symmetric admissible functions and volume controlled Lipschitz deformations}
\begin{definition} \label{def:E-admissible}
Given $\E=\set{E_i}_{i=1}^N \in \opns(\R^{n+1})$, a function $f \colon \R^{n+1} \to \R^{n+1}$ is \emph{symmetric $\E$-admissible} if it is Lipschitz continuous and satisfies the following. Let $\widetilde{E}_i \define \Int(f(E_i))$ for $i = 1,\dots,N$. Then:
\begin{itemize}
\item[(a)] $\set{\widetilde{E}_i}_{i=1}^N$ are mutually disjoint;
\item[(b)] $\R^{n+1} \setminus \bigcup_{i=1}^N \widetilde{E}_i \subset f\left(\bigcup_{i=1}^N \partial  E_i\right)$;
\item[(c)] $\sup_{x\in \R^{n+1}} \abs{f(x)-x}<\infty$;
\item[(d)] $\sigma(f(x))=f(\sigma(x))$ for every $x\in \R^{n+1}$.
\end{itemize}
\end{definition}
This class is non empty, as it contains the identity map.
\begin{lemma}(See \cite[Lemma $4.4$]{KimTonegawa})\label{l:preserving partitions}
For $\E = \set{\E_i}_{i=1}^N \in \opns(\R^{n+1})$, let $f$ be symmetric $\E$-admissible. If we define $\widetilde\E \define \{\widetilde E_{i}\}_{i=1}^N$ with $\widetilde E_i \define \Int(f( E_i))$, then $\widetilde\E \in \opns(\R^{n+1})$. 
\end{lemma}
\begin{notation}
If $\E \in \opns$ and $f \in \Lip(\R^{n+1}; \R^{n+1})$ is a symmetric $\E$-admissible function, then the resulting open partition $\widetilde \E \in \opns$ will be denoted $f_{\star} \E$. 
\end{notation}
Among all the symmetric $\E$-admissible functions, we introduce a particular class.
\begin{definition}\label{ar8}
For $\E=\set{E_i}_{i=1}^N\in\opns(\R^{n+1})$, define
$\classEvc(\E, j,\operatorname{Sym})$ to be the set of all the functions $f\,\colon\,\R^{n+1}\to \R^{n+1}$ which are symmetric $\E$-admissible functions such that, writing $\setpiccolo{\widetilde E_i}_{i=1}^N\define 
f_{\star}\mathcal E$, the following hold:
\begin{itemize}
\item[(a)] $\abs{f(x)-x}\leq \frac{1}{j^2}$ for all $x\in\R^{n+1}$,
\item[(b)]
$\displaystyle \mathcal L^{n+1}(\widetilde E_i\triangle E_i)\leq \frac{\|\partial\mathcal E\|(\R^{n+1})-\|\partial f_{\star}\mathcal E\|(\R^{n+1})}{j} \ \  \mbox{for all} \ i=1,\ldots,N \, $;
\item[(c)] $\|\partial f_{\star}\mathcal E\|(\varphi)\leq \|\partial \E\|(\varphi)$
for all $\varphi\in \mathcal A_j\, $;
\item[(d)] $\Ha^n\left( H_0 \cap  \left(\bigcup_{i=1}^N \partial \widetilde E_i\right) \right)=0\, $.
\end{itemize}
Furthermore, we define 
\begin{equation}\label{ar1}
\begin{split}
\Deltajvc \norm{\partial\mathcal E}(\R^{n+1})&\define \inf_{f\in{ \classEvc(\mathcal{E}, j, \operatorname{Sym})}}
\set{\norm{\partial f_{\star}\E}(\R^{n+1})-\norm{\partial\mathcal E}( \R^{n+1})}\leq 0\, ,
\end{split}
\end{equation}
\end{definition}
\noindent This class may be empty in general, due to condition (d). However, if $\E=\set{E_i}_{i=1}^N$ satisfies $\Ha^n\left( H_0 \cap  \left(\bigcup_{i=1}^N \partial  E_i\right) \right)=0$, which will always be the case for us, the identity function belongs to this class.

\begin{lemma}
\label{lem: massive drop area implies admissible}
    Let $\E=\set{E_i}_{i=1}^N\in \opns$, $j\in \Na$, $C_1,\ldots,C_K$ be finitely many pairwise disjoint compact symmetric sets and let $f\colon \R^{n+1}\to \R^{n+1}$ be a symmetric $\E$-admissible function such that
    \begin{itemize}
        \item [(a)] $\set{x\, \colon \, f(x) \neq x}\cup \set{f(x) \, \colon \, f(x) \neq x}\subseteq \bigcup_{\l=1}^K C_\l \, $,
        \item [(b)] $\norm{\partial f_\star \E}(C_\l)\leq \exp(-j\, \mathrm{diam}\, \bigl( C_\l\cap \closedhalf\bigr)) \norm{\partial \E}(C_\l)$ for every $\l=1,\ldots,K\, $.
    \end{itemize}
    Then, we have $\|\partial f_{\star}\mathcal E\|(\varphi)\leq \|\partial \E\|(\varphi)$
for all $\varphi\in \mathcal A_j$.
\begin{proof}
    For any $\varphi \in \cA_j$, 
    \begin{equation*}
        \begin{split}
            \norm{\partial f_\star \E}(\varphi)- \norm{\partial \E}(\varphi) &= \sum_{\l=1}^K \norm{\partial f_\star \E}\mres{C_\l}(\varphi)- \norm{\partial \E}\mres{C_\l}(\varphi)\\
            &\leq \sum_{\l=1}^K \max_{C_\l} \varphi \, \norm{\partial f_\star E}(C_\l)- \min_{C_\l} \varphi \, \norm{\partial \E}(C_\l)\\
            &=\sum_{\l=1}^K \max_{C_\l\cap \R^{n+1}_{\ge 0}} \varphi \, \norm{\partial f_\star E}(C_\l)- \min_{C_\l\cap \closedhalf} \varphi \, \norm{\partial \E}(C_\l)\\
            &\leq \sum_{\l=1}^K \min_{C_\l} \varphi \left( \exp\left(j\, \mathrm{diam}\, \bigl(C_\l\cap \closedhalf\bigr)\right) \norm{\partial f_\star \E}(C_\l)- \norm{\partial \E}(C_\l)\right)\\
            &\leq 0\, ,
        \end{split}
    \end{equation*}
    where we used \eqref{e:Gronwall} in the penultimate line and (b) in the last line.
\end{proof}
\end{lemma}
\subsection{Smoothing of varifolds and first variations}
We let $\psi \in C^\infty(\R^{n+1})$ be a radially symmetric function such that 
\begin{equation*}
    \begin{split}
        &\psi(x)=1 \ \text{for} \ \abs x \leq \frac{1}{2}\, , \qquad \qquad  \psi(x)=0 \ \text{for} \ \abs x \ge 1\\
        & 0 \leq \psi(x) \leq 1\, , \quad \abs{ \nabla \psi(x)} \leq 3\, , \quad \normpiccolo{\nabla ^2 \psi(x)}\leq 9 \ \mathrm{for all} \ x\in \R^{n+1}\, ,
    \end{split}
\end{equation*}
and we define, for each $\eps \in (0,1)$,
\begin{equation*}
    \widehat \Phi_\eps(x) \define \frac{1}{\left( 2\pi \eps^2\right)^\frac{n+1}{2}} \exp\left(-\frac{\abs{x}^2}{2\eps^2}\right) 
    \end{equation*}
    and
    \begin{equation}
    \label{e:definition_Phi_eps}
    \Phi_\eps(x)\define c(\eps)\,  \psi(x) \, \widehat \Phi_\eps(x)\, ,
\end{equation}
where the constant $c(\eps)$ is chosen in such a way that 
\begin{equation*}
    \int_{\R^{n+1}} \Phi_\eps(x) \, dx=1\, .
\end{equation*}
\begin{definition}
Given $V\in \V_n(\R^{n+1})$, we introduce $\Phi_\eps \ast \|V\|$ as the measure on $\R^{n+1}$ defined by
\begin{equation*}
(\Phi_\eps \ast \norm V) (\varphi) \define \norm V (\Phi_\eps \ast \varphi) = \int_{\R^{n+1}} \int_{\R^{n+1}} \Phi_\eps (x-y) \, \varphi(y) \, dy \, d\norm V(x)
\end{equation*}
for all $\varphi \in C_c(\R^{n+1})$, identified with the smooth function 
\begin{equation*}
(\Phi_\eps\ast\norm V)(x) \define \int_{\R^{n+1}} \Phi_\eps(y-x) \, d\norm V(y)
\end{equation*}
by means of the identity
\begin{equation*}
(\Phi_\eps \ast \norm V) (\varphi) = \langle \Phi_\eps \ast \norm V\,,\,\varphi\rangle_{L^2(\R^{n+1})}\,.
\end{equation*}
Analogously, the smoothing of $\delta V$ by means of the convolution kernel $\Phi_\eps$ is the vector field $\Phi_\eps * \delta V \in C^\infty(\R^{n+1}, \R^{n+1})$ defined by
    \begin{equation*}
        (\Phi_\eps \ast \delta V)(x) \define  \int_{\bG_n(\R^{n+1})} S(\nabla \Phi_\eps(y-x)) \, dV(y,S)\, ,
    \end{equation*}
    in such a way that 
    \begin{equation*}
        \delta V(\Phi_\eps * g)= \pair{\Phi_\eps * V, g} \quad \forall \, g\in C^1_c(\R^{n+1}, \R^{n+1})\, .
    \end{equation*}
\end{definition}

\subsection{Tangential smoothed mean curvature vector }
Given $V\in \V_n(\R^{n+1})$ and $\eps \in (0,1)$ we define 
\begin{equation*}
    \widetilde \h_\eps(\cdot, V) \define - \frac{ \Phi_\eps \ast \delta V}{\Phi_\eps \ast \norm V + \eps }
\end{equation*}
and
\begin{equation*}
     \h_\eps(\cdot, V)\define \Phi_\eps \ast \widetilde \h_\eps\, .
\end{equation*}
If $V\in \V_n(\R^{n+1}_+)$, let $V^S=V+\sigma_\sharp V$ be its symmetrization and we further introduce 
\begin{equation*}
    \htildeepsfb(\cdot,V) \define - \frac{ \Phi_\eps \ast \delta V^S}{\Phi_\eps \ast \norm{ V^S} + \eps }
\end{equation*}
and 
\begin{equation*}
    \hepsfb(\cdot, V)\define \Phi_\eps \ast \htildeepsfb\, .   
\end{equation*}
We record some $L^\infty$ properties of the smoothed mean curvature vector, as in \cite[Lemma $5.1$]{KimTonegawa}.

\begin{lemma} \label{l:smc estimates}
For every $M > 0$, there exists a constant $\eps_1 \in \left( 0,1 \right)$, depending only on $n$, $M$ such that the following holds. Let $V \in \V_{n}(\R^{n+1})$ such that $\norm{V}(\R^{n+1}) \leq M$; then, for every $x\in \R^{n+1}$ and $\eps \in \left( 0, \eps_1 \right)$, we have:
\begin{gather}\label{e:h in L infty}
\abs{\h_\eps(x, V)} \leq  \eps^{-2}\,,\\
\notag \norm{\nabla \h_\eps(x,V)} \leq  \eps^{-4}\,,\\
\notag \norm{ \nabla^2 \h_\eps(x,V) } \leq \eps^{-6}\,.
\end{gather}
\end{lemma}

\noindent We now state a key lemma showing that the approximate mean curvature associated to a symmetric partition is really symmetric, despite the presence of convolutions. This is where the flatness of $H_0$ plays a crucial role, as we will discuss in \autoref{appendix: non flat domains}.
\begin{lemma}
\label{lem: symmetry of mean curvature}
    Let $V\in \V_n(\R^{n+1}_+)$ and $V^S=V+\sigma_\sharp V\in \V_n(\R^{n+1})$ its symmetrization. Then, for any $x\in \R^{n+1}$ we have
    \begin{equation*}
        \sigma( \htildeepsfb (x))= \htildeepsfb (\sigma(x))
    \end{equation*}
    and
    \begin{equation*}
   \sigma(\hepsfb(x)=\hepsfb(\sigma(x))\, .
    \end{equation*}
    \begin{proof}
        By \eqref{e: integration on the reflection}, we have 
        \begin{align*}
         \left(\Phi_\eps\ast \delta V^S \right)(\sigma (x))&=\int_{\bG_n(\R^{n+1})} S'\left[ \nabla \Phi_\eps(y-\sigma(x)\right] \, dV^S(y,S')\\
         &= \int_{\bG_n(\R^{n+1})} S\left[ \nabla \Phi_\eps (y-\sigma(x))\right]+ Q\,S\,Q\left[ \nabla \Phi_\eps (\sigma(y)-\sigma(x))\right] dV(y,S)\\
         &=\vcentcolon I_1+I_2
        \end{align*}
        and 
        \begin{align*}
        Q\left[ \left(\Phi_\eps \ast \delta V^S\right)(x) \right] &=\int_{\bG_n(\R^{n+1})} Q \left[ S'\left[ \nabla \Phi_\eps (y-x) \right]\right] \, dV^S (y,S')\\
        &= \int_{\bG_n(\R^{n+1})} Q \left[ S\left[ \nabla \Phi_\eps (y-x) \right]\right] \, d V (y,S) \\
        & \quad +\int_{\bG_n(\R^{n+1})} Q \left[ Q\,S\, Q\left[ \nabla \Phi_\eps (\sigma(y)-x) \right]\right] \, d V (y,S)\\
        &= \int_{\bG_n(\R^{n+1})} Q \left[ S\left[ \nabla \Phi_\eps (y-x) \right]\right] + S\left[Q\left[ \nabla \Phi_\eps (\sigma(y)-x) \right]\right] \, d V (y,S)\\
        &=\vcentcolon I_3+I_4\, .
        \end{align*}
        We claim that $I_1=I_4$ and $I_2=I_3.$ By definition of the kernel \eqref{e:definition_Phi_eps},
        \begin{equation}
        \label{e: gradient kernel}
            \nabla \Phi_\eps (\sigma(y)-x)= c(\eps) \left(\widehat \Phi_\eps(\sigma(y)-x) \nabla \psi(\sigma(y)-x)+ \psi(\sigma(y)-x) \nabla \widehat \Phi_\eps(\sigma(y)-x)\right)\, .
        \end{equation}
        Since both $\psi$ and $\widehat \Phi_\eps$ are radially symmetric functions and $\sigma$ is an isometry, we have
        \begin{equation*}
            \widehat \Phi_\eps(\sigma(y)-x)= \widehat \Phi_\eps( \sigma( \sigma(y)-x))) = \widehat \Phi_\eps (y-\sigma(x))\, 
             \end{equation*}
             and similarly
             \begin{equation*}
                 \psi(\sigma(y)-x)= \psi(y-\sigma(x)).\, 
             \end{equation*}
     Note that for any radially symmetric function of the form $\eta(x)= g(\abs{x})$, it holds 
     \begin{equation*}
\nabla \eta(x)= g'(\abs{x})\, \frac{x}{\abs{x}}= g'(\abs{\sigma(x)}) \, \frac{x}{\abs{\sigma(x)}}
     \end{equation*}
     and 
     \begin{equation}
     \label{e: simmetry of gradient}
         Q\left[ \nabla \eta (x)\right] =g'(\abs{\sigma(x)}) \, \frac{Q[x]}{\abs{\sigma(x)}}=\nabla \eta(\sigma (x))\, .
     \end{equation}
     By \eqref{e: gradient kernel}-\eqref{e: simmetry of gradient} (applied to both $\psi$ and $\widehat \Phi_\eps)$, we have 
     \begin{equation*}
         Q\left[ \nabla \Phi_\eps(\sigma(y)- x)\right]= \nabla \Phi_\eps(y- \sigma(x))\, 
     \end{equation*}
     yielding that $I_1=I_4$. A similar argument allows to show that also $I_2=I_4$. Thus, we have
     \begin{equation}
     \label{e: symmetry of variation}
         Q\left[\left(\Phi_\eps \ast \delta V^S\right) (x)\right] = \left(\Phi_\eps \ast \delta V^S\right)(\sigma(x))\, .
     \end{equation}
   Moreover,
     \begin{equation}
     \label{e: symmetry of mass}
     \begin{split}
        \Phi_\eps \ast \normpiccolo{V^S}(x)&= \int_{\R^{n+1}} \Phi_\eps(y-x)+ \Phi_\eps(\sigma(y)-x) \, d\norm{V}(y) \\
        &= \int_{\R^{n+1}} \Phi_\eps(\sigma(y)-\sigma(x))+ \Phi_\eps(y-\sigma(x)) \, d\norm{V}(y)\\
        &= \Phi_\eps \ast \normpiccolo{V^S}(\sigma(x))\, .
        \end{split}
     \end{equation}
     \eqref{e: symmetry of variation} and \eqref{e: symmetry of mass} show that 
     \begin{equation*}
         Q\left[ \htildeepsfb(x)\right] = \htildeepsfb (\sigma(x))
     \end{equation*}
     and then, by linearity, also 
     \begin{equation*}
         Q\left[ \hepsfb (x)\right]= \hepsfb (\sigma(x))\, .
     \end{equation*}
    \end{proof}
\end{lemma}

\section{Existence of a symmetric Brakke flow in $\R^{n+1}$}
In this section we show how to construct a symmetric Brakke flow in $\R^{n+1}$ starting from the symmetrization of the initial datum $\Gamma_0\subset \R^{n+1}_+$, by suitably modifying the approximation scheme developed in \cite{KimTonegawa, STCanonical}. The rectifiability and the integrality of the solution will be addressed in the next sections. 
\subsection{The construction of the approximate flow}
Consider an initial rectifiable set $\Gamma_0\subset \R^{n+1}_+$ with a corresponding finite open partition $\E_0$ of $\R^{n+1}_+$ consisting of $N$ elements as in \autoref{ass:main}; symmetrize it as described in \autoref{def: symmetrization of grains} to obtain a symmetric open partition of $\R^{n+1}$ (see \autoref{lem: symmetrization is opn}).

For every natural number $j$ and for times $t \in \left[0,j\right]$, we define open partitions $\E_j^S (t) = \{E_{j,1}^S(t),\ldots,E_{j,N}^S(t)\}$ according to the following rule: given the initial datum $\E_0^S\in \opns(\R^{n+1})$, we set
\begin{eqnarray} \label{app:0} 
    \E_j^S(0)&=&\E_0^S \,, \\ \label{app:future}
    \E_j^S(t) &=& \E_{j,k}^S \qquad \mbox{for all $t \in \left( (k-1) \Delta t_j, k \Delta t_j \right]$}\,.
\end{eqnarray}
In \eqref{app:future}, the epoch length is $\Delta t_j = 2^{-p_j}$ for some $p_j \in \Na$, and $k\in\{1,\ldots,j\,2^{p_j}\}$. For each $k$, the open partition $\E_{j,k}^S$ is obtained from the open partition $\E_{j,k-1}^S$ (with the convention $\E_{j,0}^S = \E_0^S$) through successive modifications, encoded in the following two-step algorithm:
\begin{itemize}
    \item[(1)] First, one chooses $f_1\in \E^{vc}(\E_{j,k-1}^S, j, \operatorname{Sym})$ with the property that
\begin{equation}
\label{e: def f1}
\normpiccolo{\partial (f_1)_{\star} \E_{j,k-1}^S }(\R^{n+1}) - \normpiccolo{ \partial \E_{j,k-1}^S}(\R^{n+1}) \leq  (1-j^{-5}) \, \Deltajvc \normpiccolo{ \partial \E_{j,k-1}^S}(\R^{n+1})\,,
\end{equation}
and sets
\begin{equation*}
(\E^S_{j,k})^\ast \define (f_1)_\star(\E_{j,k-1}^S)\, ;
\end{equation*}
thus, in particular,
\begin{equation*}
(E^S_{j,k,i})^\ast \define {\rm int}(f_1(E_{j,k-1,i})) \quad \mbox{for every $i \in \set{1,\ldots,N}$}\,.
\end{equation*} 

\item[(2)] Next, one defines the map

\begin{equation}
\label{e: def f2}
 f_2(x) \define x + \Delta t_j \, \hepsfb(x)\,,
\end{equation} 
where $\eps_j\in (0,1)$, $\hepsfb(x)= \h_\eps\left(x, \partial (\E_{j,k}^S)^\ast\right)$ is the $\eps_j$ smoothed mean curvature vector of the multiplicity one varifold $\pa(\E^S_{j,k})^\ast$. Notice that $ f_2$ is a diffeomorphism of $\R^{n+1}$ due to \autoref{l:smc estimates} as soon as $\Delta t_j \ll \eps_j^4$. We set
 \begin{equation*}
 \E_{j,k}^S \define (f_2)_{\star}(\E^S_{j.k})^\ast\, .
 \end{equation*}
\end{itemize}
\begin{remark}
    The symmetrization process does not take place at every step, but rather only at the beginning of this procedure. The symmetry is indeed preserved along the iteration, as we will discuss in \autoref{prop: symmetry along iteration}.
\end{remark}
\begin{theorem}
\label{theorem: existence of limit symmetric measures}
There is a constant $c_2=c_2(n) \gg 1$ with the following property. Let $\Gamma_0\subset \R^{n+1}_+$ and $\E_0 \in \opn(\R^{n+1}_+)$ be as in \autoref{ass:main} and let $\E_0^S$ be the symmetrization of $\E_0$. Then there exist 
\begin{itemize}
\item a subsequence $j_\l$ of $\Na\, $,
\item reals $\eps_{j_\l} \in \left( 0, j_\ell^{-6}\right)$ with $\lim\limits_{\ell\to\infty} \eps_{j_\ell}=0\,$,
\item integers $p_{j_\ell} \in \mathbb N$ with $\Delta t_{j_\ell}=2^{-p_j} \in \left( 2^{-1} \eps_{j_\ell}^{c_2}, \eps_{j_\ell}^{c_2} \right]\,$,
\item a family $\{\mu_t^S\}_{t \ge 0}$ of symmetric Radon measures on $\R^{n+1}$,
\item a family $\E(t)^S=\{E_1^S(t),\ldots,E_N^S(t)\}_{t \ge 0}\,$ of symmetric open sets
\end{itemize}
such that the approximating flow of open partitions $\E_{j_\ell}^S(t)$ defined by \eqref{app:0}-\eqref{app:future} satisfies for all $T < \infty$,
\begin{align*}
    &\limsup_{\ell \to \infty} \sup_{t \in \left[0,T\right]} \normpiccolo{ \partial \E_{j_\ell}^S(t)}(\R^{n+1}) \leq \normpiccolo{ \partial\E_0^S}(\R^{n+1})\, ,\\
    &\limsup_{\ell\to\infty} \bigint{0}^T \left( \bigint{\R^{n+1}} \frac{\big| \Phi_{\eps_{j_\ell}} \ast \delta (\partial \E_{j_\ell}^S(t))\big|^2}{\Phi_{\eps_{j_\ell}} \ast \normpiccolo{\partial \E_{j_\ell}^S(t)} + \eps_{j_\ell}} \, dx - \frac{1}{\Delta t_{j_\ell}}\, \Delta_{j_\ell}^{vc}\normpiccolo{\partial \E_{j_\ell}^S(t)}(\R^{n+1}) \right) \, dt < \infty \,,    \\
    &\lim_{\ell\to\infty} j_\ell^{2(n+1)} \,\Delta_{j_\ell}^{vc} \normpiccolo{\partial\E_{j_\ell}^S(t)} (\R^{n+1}) = 0 \quad \mbox{for a.e.} \ t \in \R_+\,,\\
    &\lim_{\ell\to\infty} \normpiccolo{ \partial \E_{j_\ell}^S(t)} (\varphi) = \mu_t^S (\varphi) \quad \mbox{for all $\varphi \in C_c (\R^{n+1})$ and any $t \in \R^+$}\,, \\
    &  \mbox{$\chi_{E^S_{j_\ell,i}(t)} \to \chi_{E^S_i(t)^S}$ in $L^1_{{\rm loc}}(\R^{n+1})$ as $\ell \to \infty$ for every $i \in \{1,\ldots,N\}$ and any $t \in\R^+$}\,.
\end{align*}
Furthermore, the following assertions hold:
\begin{itemize}
    \item[(a)] There exists a subset $Z \subset \R_+$ with $\Leb^1(Z)=0$ such that, for every $t \in \R_+\setminus Z$, $\mu_t^S$ is integer rectifiable in $\R^{n+1}$, with even density for $\normpiccolo{\mu_t^S}-\mbox{a.e.} \ x\in H_0$;
    \item[(b)] If $V_t^S$ is defined to be an arbitrary varifold in $\V_n(\R^{n+1})$ with $\normpiccolo{V_t^S}=\mu_t^S$ also for $t \in Z$, then the family $\setpiccolo{V_t^S}_{t \ge 0}$ satisfies 
    \begin{itemize}
        \item[(b1)] $\normpiccolo{V_0^S}=\Ha^n\mres{\Gamma_0^S}$, where $\Gamma_0^S$ is the symmetrization of $\Gamma_0$;
        \item[(b2)] $\displaystyle \int_0^\infty \int_{\R^{n+1}} \abs{\h(\cdot, V_w^S)}\, d\norm{V^S_w}  dw<\infty \quad \mbox{for all} \ t>0$,
        \item[(b3)] $\displaystyle \normpiccolo{V_t^S}(\R^{n+1})+\int_0^t \int_{\R^{n+1}} \abspiccolo{\h(\cdot, V_w^S)}^2\, d\normpiccolo{V_w^S}\, dw\leq \Ha^n(\Gamma_0^S)$ for all $t>0$;
        \item[(b4)] $\setpiccolo{V_t^S}_{t\ge 0}$ is a Brakke flow in $\R^{n+1}$;
    \end{itemize}
    \item[(c)] Setting 
    \begin{equation*}
        \Gamma^S(t)\define \R^{n+1} \setminus \bigcup_{i=1}^N E^S_i(t)\,,
    \end{equation*}
    the flow of grains $E_i^S(t)$ satisfies the following:
    \begin{itemize}
        \item[(c1)] $E_i^S(0)=E_{0,i}^S$ for $i=1,\ldots,N$.
        \item[(c2)] $E_1^S(t),\ldots, E_N^S(t)$ are pairwise disjoint for $t\in \R_+$.
        \item[(c3)] For every $0<T<\infty$, the families $\{E^S_i(t)\}_{t\in[0,T)}$ define a generalized $\mathrm{BV}$ flow in $\R^{n+1}$ with respect to symmetric tests, with scalar velocities $v_i^S$ given by
        \begin{equation*}
            v_i^S(x,t)= \h(x,V^S_t)\cdot \nu_i(x,t)\, ,
            \end{equation*}
        where $\nu_i(x,t)$ is the outer unit normal to $E_i^S(t)$ at the point $x$.
        \item[(c4)] Suppose that, for $0\leq t_1<t_2<\infty$, the Brakke flow is unit density for $t\in(t_1,t_2)$. Then, the families $\{E^S_i(t)\}_{t\in[0,T)}$ actually define a $\mathrm{BV}$ flow in $\R^{n+1}$ with respect to symmetric tests (not only \emph{generalized}).
    \end{itemize}
    \item[(d)] If $\Ha^n(\partial E_{0,i}^S \setminus \partial^\ast E_{0,i}^S)=0$, then $\lim\limits_{t\to 0^+} \normpiccolo{V_t^S}=\normpiccolo{V_0^S}$.
\end{itemize}
\end{theorem}
The proof of the existence of this flow can be mostly carried out by repeating \textit{verbatim} the arguments in \cite{KimTonegawa, STCanonical}. We only need to show that the symmetry is preserved throughout the approximating procedure, yielding symmetric limit varifolds, and that rectifiability and integrality can be reproduced in this setting. Indeed, as we mentioned in the introduction, the original proofs require the application of some particular tailor-made area reducing Lipschitz maps, which may not be symmetric in general. 

\subsection{Symmetry along the iteration}
We show that the varifolds associated to the approximating open partitions are symmetric and thus any limiting measure of such objects must be symmetric too by \autoref{lemma: limit of symmetric is symmetric}.
\begin{proposition}
\label{prop: symmetry along iteration}
    Consider the time-discrete procedure defined by \eqref{app:0}-\eqref{app:future}. Then, for any $j,k,i$, the set $E_{j,k,i}^S$ is symmetric with respect to $H_0$ and
        \begin{equation}
    \label{e: no mass on the boundary along iteration}
    \Ha^n(H_0 \cap \partial\E^S_{j,k})=0\, .
    \end{equation}
    In addition, if we denote by $\E_{j,k}=\set{E_{j,k,i}}_{i=1}^N\in \opn(\R^{n+1}_+)$, where $E_{j,k,i}\define E_{j,k,i}^S\cap \R^{n+1}_+$, we have 
    \begin{equation}
    \label{e: symmetry of varifolds along iteration}
        \partial \E_{j,k}^S = \partial \E_{j,k} + \sigma_\sharp( \partial \E_{j,k})\, .
    \end{equation}
  \begin{remark}
      In \eqref{e: symmetry of varifolds along iteration} we are regarding $\E_{j,k}$ as a partition of $\R^{n+1}_+$, thus considering the boundaries of the grains \emph{inside} $\R^{n+1}_+$.
  \end{remark}
    \begin{proof}
If $k=0$, the sets $E_{j,0,i}^S$ are symmetric by construction and \eqref{e: no mass on the boundary along iteration} is given by \autoref{ass:main} $(A3)$. In fact, if $x\in B_{0,i}$ for some $i$, by \autoref{def: symmetrization of grains} and \autoref{lem: symmetrization is opn}, $x$ is an interior point of $E_{0,i}^S$, and it cannot be a point of $\partial \E^S_{0}$. Hence,
\begin{equation*}
     \Ha^n(H_0 \cap \partial\E^S_{0})\leq \Ha^n\left(H_0\setminus \left(\bigcup_{i=1}^N B_{0,i} \right)\right)=0.
\end{equation*}
Given the conclusions for the step $k-1$, we show their validity for the step $k$. Since by inductive hypothesis $\Ha^n(H_0 \cap \partial\E^S_{j,k-1})=0$, there exists at least one function $f\in \E^{\mbox{vc}}(\E_{j,k-1}^S, j, \operatorname{Sym})$. Any almost minimizer $f_1$ as in \eqref{e: def f1} preserves the symmetry by \autoref{def:E-admissible} (d) and charges no boundary measure on $H_0$ by \autoref{ar8} (d), proving \eqref{e: no mass on the boundary along iteration} for the intermediate step $(\E^S_{j,k})^\ast$. 

By \autoref{lem: symmetry of mean curvature}, the map $f_2$ defined in \eqref{e: def f2} satisfies $f_2(x)=\sigma(f_2(\sigma(x)))$. Moreover, we claim that $f_2(\R^{n+1}_+)\subseteq \R^{n+1}_+$ and $f_2(H_0)\subseteq H_0$. Indeed, if, for the sake of a contradiction, there is some $z\in \R^{n+1}_+$ such that $f_2(z)\in H_0$, then $f_2(\sigma(z)) =\sigma(f_2(z))= f_2 (z),$ violating the injectivity of $f_2$. By continuity, we get the claim. Since $f_2$ is a diffeomorphism, for any set $A$, $x\in \partial A$ if and only $f_2(x)\in \partial f(A)$; together with the Lipschitzianity of $f_2$ and the previous claim, this shows the thesis.

Finally, \eqref{e: symmetry of varifolds along iteration} is an immediate consequence of the symmetry of the grains and \eqref{e: no mass on the boundary along iteration}.
    \end{proof}
\end{proposition}

\section{Rectifiability of the symmetric flow}
In this section we prove the rectifiability of the symmetric varifolds introduced in \autoref{theorem: existence of limit symmetric measures}. Most of the arguments developed by \cite[Section 7]{KimTonegawa} can be reproduced in our setting, except for \cite[Proposition $7.2$]{KimTonegawa}. Indeed, this proposition requires the construction of a specific Lipschitz deformation which, roughly, drastically reduces the area of the varifold on a small scale, provided the mass is already small. We start by proving a symmetric variant thereof. See \autoref{fig:LipDefDirections} for a geometric intuition of the map. Finally, we describe how to use it to establish the rectifiability of the limit measures.
\begin{notation}
    Given $z\in \R^{n+1}$ and $\rho>0$, we define 
    \begin{gather*}
        B_\rho^S(z)=B_\rho(z)\cup B_\rho(\sigma(z))\\
        U_\rho^S(z)=U_\rho(z)\cup U_\rho(\sigma(z))\, .
    \end{gather*}
    \end{notation}
\noindent For future references, note that
    \begin{equation}
    \label{e: boundary BrS}
        \partial B_\rho^S(z) = \left(\partial B_\rho(z) \cap \R^{n+1}_+\right) \cup \left(\partial B_\rho(\sigma(z)) \cap \R^{n+1}_-\right)\cup \left(\partial B_\rho(z) \cap H_0\right)\, .
    \end{equation}
  and 
  \begin{equation}
  \label{e: measure BrS}
      \leb\left(U_\rho(z)\right) \leq \leb\left(U^S_\rho(z)\right)\leq  2 \,\leb\left(U^S_\rho(z)\right)\, .
  \end{equation}

\noindent We also state the monotonicity formula for future references.
\begin{lemma}
\label{lem: monotonicity formula}
Suppose $V\in {\bf V}_n(\mathbb R^{n+1})$, $0<r_1<r_2<\infty$, $x\in 
\mathbb R^{n+1}$, and for $0\leq s<\infty$,
\begin{equation*}
\norm{\delta V}(B_r(x))\leq s\norm{V}(B_r(x))
\end{equation*}
whenever $r_1< r< r_2$. Then
\begin{equation}
\label{e: monotonicity formula}
(\exp (sr))r^{-n} \|V\|(B_r(x))
\end{equation}
is nondecreasing in $r$ for $r_1< r <r_2$. 
\end{lemma}

\begin{proposition}
\label{prop: area reducing symmetric radial projection}
    There exist $c_1, c_2\in (0,\infty)$ depending only on $n$ with the following property. Let $\E^S=\setpiccolo{E_i^S}\in \opns(\R^{n+1})$ be a symmetric open partition with $\Ha^n(H_0 \cap \partial \E^S)=0$ and let $z\in  \spt \normpiccolo{\partial \E^S} \cap \closedhalf$ with $\normpiccolo{\partial \E^S}\bigl( B_R^S(z)\bigr)\leq c_1 R^n$. Then, there exist a symmetric $\E^S$-admissible function $f$ and $r\in\left[\frac{3}{4}R, R\right]$ such that 
    \begin{itemize}
        \item[(1)] $f(x)=x$ for $x\in \R^{n+1} \setminus U_r^S(z)$;
        \item[(2)] $f(x)\in B^S_r(x)$ for $x\in B_r^S(z)$;
        \item[(3)] $\normpiccolo{\partial f_\star \E^S}(B_r^S(z))\leq \frac12 \normpiccolo{\partial \E^S}(B_r^S(z))$;
        \item[(4)]  $\mathcal L^{n+1}(E_i^S \triangle \widetilde E_i^S) \leq c_2 \left( \normpiccolo{\partial \E^S}(B^S_r(z)) \right)^\frac{n+1}{n}$ for all $i$, where $\setpiccolo{\widetilde \E^S_i}_{i=1}^N=f_\star \E^S$;
        \item[(5)] $\Ha^n(H_0 \cap \partial \widetilde \E^S)=0$.
    \end{itemize}
\end{proposition}
\begin{proof}
If $\d(z)>R$, one can directly apply \cite[Proposition $7.2$]{KimTonegawa} to both halfspaces to get the thesis. We then focus on the case $0\leq \d(z)\leq R$, that is $B_R(z) \cap B_R(\sigma(z)) \neq \emptyset$.

For $r>0$ let 
\begin{equation*}
\eta(r)=\normpiccolo{\partial \E^S}(B_r^S(z))= \Ha^n\left(B_r^S(z)\cap \bigl(\cup_{i=1}^N \partial \E_i^S\bigr)\right)\, .
\end{equation*}
Let $Y=\set{z, \, \sigma(z)},$ so that $\eta$ can be written as 
\begin{equation*}
\eta(r)=\Ha^n\left( \set{\mathrm{dist}(\cdot, Y)\leq r} \cap \bigl(\cup_{i=1}^N \partial \E_i^S\bigr)\right)\, .
\end{equation*}
Since $z,\, \sigma(z)\in \spt
\normpiccolo{\partial \E^S}$, we have $\eta(r)>0$ for $r>0$ and $\eta(r)$ is a monotone increasing function, thus $\mbox{a.e.}$ differentiable. By the coarea formula, we also have  
\begin{equation}
\label{e: 7.3 KT}
\mathcal H^{n-1}\left(\partial B_r^S(z)\cap \left(\cup_{i=1}^N \partial E_i^S\right)\right)\leq \eta'(r)<\infty\, ,
\end{equation}
whenever $\eta$ is differentiable.

By the relative isoperimetric inequality \cite[p.152]{AFP}, there exists $\gamma$ depending only on $n$ such that
\begin{equation}
\label{rlm1.1}
\min\set{\leb\left(U_R(z) \cap E_i^S\right),\leb\left(B_R(z)\setminus E_i^S\right)}
\leq \gamma \left( \Ha^{n}\left(U_R(z) \cap \partial E_i^S\right)\right)^{\frac{n+1}{n} }\, .
\end{equation}
By symmetry, we have 
\begin{equation}
\label{e: proof prop 7.2 0}
\begin{split}
    \leb\left(U_R^S(z) \cap E_i^S\right)&\leq \mathcal L^{n+1}\left(U_R(z) \cap E_i^S\right) + \mathcal L^{n+1}\left(U_R(\sigma(z)) \cap E_i^S\right)\\
    &= 2\, \mathcal L^{n+1}\left(U_R(z) \cap E_i^S\right)\, ,
    \end{split}
\end{equation}
and, similarly,
\begin{equation}
\label{e: proof prop 7.2 1}
    \begin{split}
        \leb\left(U_R^S(z)\setminus E_i^S\right)&\leq \leb \left(U_R(z)\setminus E_i^S\right)+\mathcal L^{n+1} \left(U_R(\sigma(z))\setminus E_i^S\right)\\
        &=2 \, \leb \left(U_R(z)\setminus E_i^S\right)\, .
    \end{split}
\end{equation}
By \eqref{e: proof prop 7.2 0}, \eqref{e: proof prop 7.2 1} and \eqref{rlm1.1}, we have 
\begin{equation}
\label{e: 7.4 KT}
    \begin{split}
       & \min \set{\leb\left(U_R^S(z) \cap E_i^S\right),\,  \leb\left(U_R^S(z)\setminus E_i^S\right)}\\
       & \qquad \quad \leq 2\min\set{\mathcal L^{n+1}(U_R(z) \cap E_i^S),\leb(U_R(z)\setminus E_i^S)}\\
        &\qquad \quad \leq 2\, \gamma \left( \Ha^{n}\left(U_R(z) \cap \partial E_i^S\right)\right)^{\frac{n+1}{n} }\\
        &\qquad \quad  \leq 2\, \gamma \left( \Ha^{n}\left(U_R^S(z) \cap \partial E_i^S\right)\right)^{\frac{n+1}{n}} \, .
    \end{split}
\end{equation}
We define $c_2\define 2\, \gamma$. Hypothesis (3) can be written as $\eta(R)\leq c_1 \, R^n$. If we restrict $c_1$ depending only on $n$ so that
\begin{equation*}
    c_1\leq \left( \frac{\omega_{n+1}}{2^{n+2} c_2}\right)^\frac{n}{n+1}\, ,
\end{equation*}
we have 
\begin{equation}
\label{e: 7.5 KT}
    \eta(R)\leq c_1\, R^n\leq \left( \frac{\mathcal L^{n+1} (U_R(z))}{2^{n+2}c_2} \right)^\frac{n}{n+1}\leq \left( \frac{\mathcal L^{n+1} (U_R^S(z))}{2^{n+2}c_2} \right)^\frac{n}{n+1}\, .
\end{equation}
We claim the existence of an index $i_0\in \set{1,\ldots, N}$ such that 
\begin{equation}
\label{e: 7.6 KT}
    \mathcal L^{n+1}\left(U_R^S(z)\setminus E_{i_0}^S\right) \leq c_2 (\eta(R))^\frac{n+1}{n}\leq \frac{\leb (U_R^S(z)) }{2^{n+2}}\, .
\end{equation}
Note that the claim only concerns the first inequality of \eqref{e: 7.6 KT}, as the latter one is \eqref{e: 7.5 KT}. We prove this claim by distinguish in two cases: first, suppose there is $E_{i_0}^S$ such that 
\begin{equation}
\label{e: proof 7.2 KT Ei0 takes a major part}
    \frac{\leb \left(U_R^S(z)\cap E_{i_0}^S\right)}{\leb \left(U_R^S(z)\right)}\ge \frac34\, ,
\end{equation}
namely $E_{i_0}$ occupies a major part of $B_R^S(z)$. If this is the case, \eqref{e: 7.6 KT} follows by \eqref{e: 7.4 KT}, since \eqref{e: proof 7.2 KT Ei0 takes a major part} implies the minimum at the left-hand-side of \eqref{e: 7.4 KT} is achieved by $U_R^S(z)\setminus E_{i_0}^S$. Suppose instead that, for the sake of a contradiction, \eqref{e: proof 7.2 KT Ei0 takes a major part} fails for every index $i\in \set{1,\ldots,N}$. Since $U_R^S \cap \left(\cup_{i=1}^N E_i^S\right)$ is a full measure set, there is a combination $E_{i_1}^S,\ldots, E_{i_J}^S$ such that, denoted $\widehat E= \cup_{k=1}^J E_{i_k}^S$, we have 
\begin{equation*}
    \frac{\leb \left(U_R^S(z)\cap \widehat E\right)}{\leb \left(U_R^S(z)\right)}\in \left(\frac14, \frac34\right) \, .
\end{equation*}
By \eqref{e: 7.4 KT} applied to $\widehat E$, we have 
\begin{equation*}
    c_2 \left( \normpiccolo{ \nabla \chi_{\widehat E}}  (U_R^S(z))\right)^\frac{n+1}{n}\ge \frac{\leb (U_R^S(z))}{4}\, ,
\end{equation*}
while we have $\normpiccolo{ \nabla \chi_{\widehat E}}  (U_R^S(z))\leq \eta(R)$, contradicting \eqref{e: 7.5 KT}. The claim is proved, namely \eqref{e: proof 7.2 KT Ei0 takes a major part} holds true for some $i_0$.

Let $r\in\left[ \frac{3}{4}R, R\right]$. By \eqref{e: 7.6 KT} and \eqref{e: measure BrS}, we have
\begin{equation*}
    \begin{split}
        \leb\left(U_r^S(z) \setminus E_{i_0}^S\right)&\leq  \frac{\leb (U_R^S) }{2^{n+2}}\leq \frac{2\, \omega_{n+1}}{2^{n+2}} \, \left(\frac{R}{\frac43 r}\right)^{n+1} \left( \frac43 r\right)^{n+1}\\
        &\leq \left(\frac23\right)^{n+1}\leb\left(U_r(z)\right)\leq \frac12 \, \leb\left(U_r^S(z)\right)\, .
    \end{split}
\end{equation*}
Thus, \eqref{e: 7.4 KT} with $r$ in place of $R$ shows 
\begin{equation}
\label{e: 7.7 KT}
    \leb \left(U_r^S(z)\setminus E_{i_0}^S\right) \leq c_2 \left( \Ha^n\bigl(U_r^S(z) \cap \partial E^S_{i_0}\bigl)\right)^\frac{n+1}{n}
\end{equation}
for all $r\in\left[ \frac{3R}{4}, R\right]$. We introduce 
\begin{equation}
\label{e: proof 7.2 KT def A and tilde A}
    \widetilde A\define \set{r\in \left[ \frac{3}{4}R, R\right] \, \colon \, \Ha^n\bigl(\partial B_r^S(z) \setminus E_{i_0}\bigr) >\frac12 \Ha^n(\partial B_r^S(z))}\, , \quad A\define \left[ \frac{3}{4}R, R\right]\setminus \widetilde A\, .
\end{equation}
By definition of the set $Y$, the coarea formula, \eqref{e: boundary BrS} and \eqref{e: proof 7.2 KT def A and tilde A}, we get
\begin{equation}
\label{e: 7.8 KT}
\begin{split}
    &\leb\left( \bigl(B_R^S(z)\setminus U^S_{\frac34 R}(z)\bigr) \setminus E_{i_0}^S\right)= \int_{\set{\frac34 R\leq \mathrm{dist}(x,Y)\leq R}\setminus E_{i_0}^S} dx\\
    &\qquad =\int_{\frac34 R}^R \Ha^n\left(\set{\mathrm{dist}(\cdot, Y)=r}\setminus E_{i_0}^S\right)\, dr =\int_{\frac34 R}^R \Ha^n\left(\partial B_r^S(z)\setminus E_{i_0}^S\right)\, dr\\
    &\qquad \ge \frac12 \Ha^n\left(\partial B_{\frac34 R}^S(z)\right) \mathcal L^1(\widetilde A)\, .
    \end{split}
\end{equation}
Since $\Ha^n\bigl(\partial B_{\frac34 R}^S(z)\bigr)\in \left[ (n+1)\omega_{n+1} \left(\frac34 R\right)^n,2(n+1)\omega_{n+1} \left(\frac34 R\right)^n\right]$, \eqref{e: measure BrS},
\eqref{e: 7.6 KT} and \eqref{e: 7.8 KT} show
\begin{equation}
\label{e: 7.9 KT}
    \mathcal L^1(\widetilde A) \leq \left( \frac23\right)^n\frac{R}{(n+1)} \quad \mbox{and} \quad \mathcal L^1(A) \ge \left(\frac14- \left( \frac23\right)^n \frac{1}{n+1} \right)R \ge \frac{R}{12}\, . 
\end{equation}
In particular, \eqref{e: 7.9 KT} implies 
\begin{equation}
\label{e: 7.10 KT}
    \Ha^n\left(\partial B_r^S(z)\setminus E_{i_0}^S\right) \leq \frac12 \Ha^n\left(\partial B_r^S(z)\right) \ \mbox{for} \ r\in A\subseteq\left[\frac{3}{4}R, R\right]\ \mbox{with} \ \mathcal L^1(A) \ge \frac{R}{12}\, .
\end{equation}
Next, fix an arbitrary $r\in A$ which satisfies \eqref{e: 7.3 KT}, and let $G_i= E_i^S \cap \partial B_r^S(z)$ for $i\in \set{1,\ldots,N}$. Each $G_i$ is open with respect to the topology on $\partial B_r^S(z)$ and $\partial G_i\subset \partial B_r^S(z) \cap \partial E_i^S$. By definition, we also have $\partial B_r^S(z) \setminus E^S_i= \partial B_r^S(z)\setminus G_i$. By the relative isoperimetric inequality on $\partial B_r(z)$, there exists $\widetilde \gamma$ depending only on $n$ such that
\begin{equation}
\label{e: proof 7.2 relative isoperimetric inequality dimension n}
    \begin{split}
       \min\set{\Ha^n(\partial B_r(z)\cap G_{i_0}), \, \Ha^n\left(\partial B_r(z)\setminus G_{i_0}\right)} &\leq \widetilde \gamma \left( \Ha^{n-1}(\partial B_r(z) \cap \partial G_{i_0} ) \right)^\frac{n}{n-1}\\
       &=\widetilde \gamma \left( \Ha^{n-1}(\partial B_r^S(z) \cap \partial G_{i_0} ) \right)^\frac{n}{n-1}\, ,
    \end{split}
\end{equation}
where the last equality is due to $\partial G_{i_0}\subset \partial B_r^S(z)$.
Arguing similarly to \eqref{e: 7.4 KT}, from \eqref{e: proof 7.2 relative isoperimetric inequality dimension n} and by defining $c_3\define 2\, \widetilde \gamma$, we can get
\begin{equation}
\label{e: proof 7.2 relative isoperimetric inequality dimension n 2}
 \min\set{\Ha^n(\partial B^S_r(z)\cap G_{i_0}), \, \Ha^n(\partial B^S_r(z)\setminus G_{i_0})} \leq c_3 \left( \Ha^{n-1}(\partial B^S_r(z) \cap \partial G_{i_0}) \right)^\frac{n}{n-1}\, .
\end{equation}
By \eqref{e: 7.10 KT} and \eqref{e: proof 7.2 relative isoperimetric inequality dimension n 2},
\begin{equation}
\label{e: 7.11 KT}
    \Ha^n\left(\partial B_r^S(z) \setminus G_{i_0}\right)\leq c_3 \left( \Ha^{n-1} (\partial G_{i_0})\right)^\frac{n}{n-1}\, .
\end{equation}

Now we choose $x_0\in U_r^S(z) \cap E_{i_0}^S \cap \R^{n+1}_+$ be such that $B_{2r_0}(x_0)\subset B_r \cap E_{i_0}^S \cap \R^{n+1}_+$ for some sufficiently small $r_0$. The existence of $x_0$ and $r_0$ is guaranteed by $E_{i_0}^S$ being open, and $x_0$ can be chosen in the upper halfspace by symmetry. Let us define $f$ as follows. $f(x)=x$ if $x\in \R^{n+1}\setminus U_r^S(z)$. Fix a direction $\eta$ and consider the ray $L_\eta$ starting from $x_0$ having direction $\eta$. There are two possibilities (see Figure \ref{fig:LipDefDirections}):
\begin{itemize}
    \item if the ray does not intersect $H_0\cap B_r^S(z)$, dilate the segment $[x_0, x_0+r_0\,\eta]$ bijectively onto the segment $[x_0, x_0+ \widetilde s \, \eta]$, where $x_0+\widetilde s \, \eta\in \partial B_r^S(z)$ is the point where $L_\eta$ meets $\partial B_r^S(z)$. Then, radial project to the point $x_0+\widetilde s \, \eta\in \partial B_r$ any point of the form $x_0+\rho \, \eta$ for $r_0\leq \rho \leq \widetilde s$.
    \item if the ray intersects $H_0\cap B_r^S(z)$, we do the same thing, but stopping at $H_0$. This means, dilate the segment $[x_0, x_0+r_0\,\eta]$ bijectively onto the segment $[x_0, x_0+ \widehat s \, \eta]$, where $x_0+\widehat s \, \eta\in H_0$ is the point where $L_\eta$ meets $H_0$. Then, radial project to the point $x_0+\widehat s \, \eta\in H_0$ any point of the form $x_0+\rho \, \eta$ for $r_0\leq \rho \leq \widehat s$.
\end{itemize}
\begin{figure}
    \centering
    \includegraphics[width=0.6\linewidth]{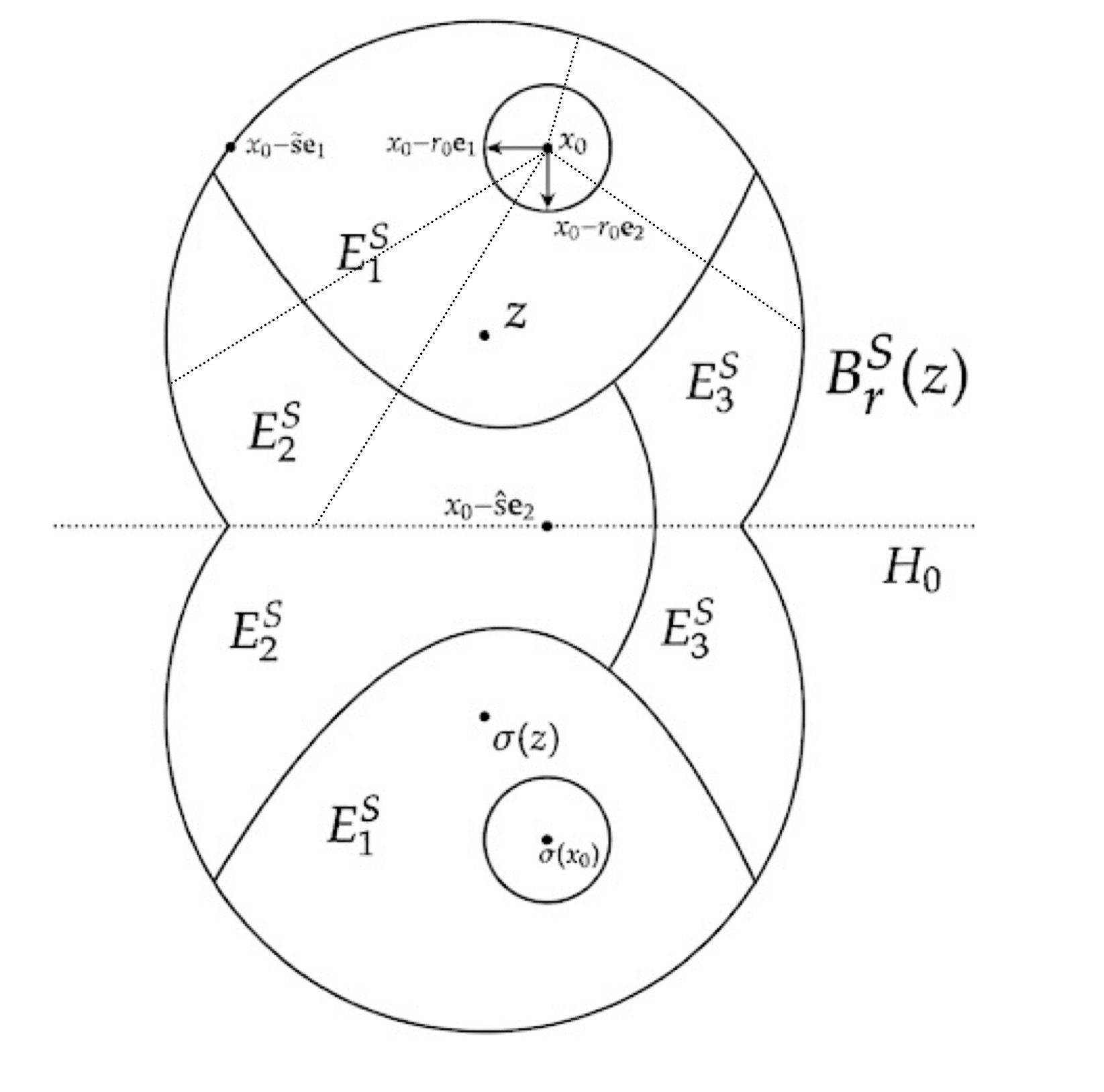}
    \caption{Description of the map}
        \label{fig:LipDefDirections}
\end{figure}

We claim that such $f$ is symmetric $\E^S$-admissible. The symmetry of the map is a consequence of the symmetery of the grains and of the definition. Let $\widetilde E_i^S\define{\rm int}(f(E_i^S))$. For $i\neq i_0$, $\widetilde E^S_{i}=E^S_i
\setminus B_r^S(z)$, because $f$ is the identity on $\R^{n+1}
\setminus B_r^S(z)$ and $f(E_i^S
\cap B_r^S(z))\subset \partial B_r^S(z)$. On the other hand, $\widetilde E_{i_0}^S=E^S_{i_0}\cup B_r^S(z)$ since $B_r^S(z)=f\left(B_{r_0}^S(x_0)\right) $ and  $B_{r_0}^S(x_0)\subset
E^S_{i_0}$, and any $x\in \partial B_r^S(z)\cap E^S_{i_0}$ is in $E^S_{i_0}\cup (U_r^S(z))$. 
For two open sets $A$ and $B$, we have $\partial (A\cap B)\subset(\partial A\cap
{\rm clos}\,B)\cup (\partial B\cap A)$ and $\partial (A\cup B)\subset
(\partial A\setminus {\rm clos}\,B)\cup (\partial B\setminus A)$. So 
\begin{equation}
\label{rlm6}
\begin{split}
\partial \widetilde E^S_{i}&=\partial \left(E_i^S\cap(\R^{n+1}\setminus
B_r^S(z))\right)\\
&\subset\left(\partial E^S_i\cap{\rm clos}\,(\R^{n+1}\setminus B_r^S(z)) \right)
\cup \left(\partial B_r^S(z)\cap E^S_i\right)\\
&=\left(\partial E^S_i\setminus U_r^S(z) \right)
\cup G_i
\end{split}
\end{equation}
for $i\neq i_0$, while 
\begin{equation}
\label{rlm7}
\begin{split}
\partial \widetilde E^S_{i_0}&=\partial \left(E^S_{i_0}\cup U_r^S(z)\right)\\
&\subset \left(\partial E_{i_0}^S
\setminus B_r^S(z)\right)\cup \left(\partial B_r^S(z)\setminus E^S_{i_0}\right)\\
&=\left(\partial E^S_{i_0}
\setminus B_r^S(z)\right)\cup \left(\partial B_r^S(z)\setminus G_{i_0}\right).
\end{split}
\end{equation}
We need to check $\R^{n+1}\setminus \cup_{i=1}^N \widetilde E_i^S\subset f\left(\cup_{i=1}^N
\partial E_i^S\right)$. Since $\R^{n+1}\setminus \cup_{i=1}^N\widetilde E_i^S$ does
not have any interior point, it is enough to prove $\cup_{i=1}^N \partial \widetilde E_i
\subset f\left(\cup_{i=1}^N \partial E_i^S\right)$. For $i\neq i_0$, $\partial E_i^S\setminus
U_r^S(z)\subset f(\partial E_i^S)$ since $f$ is the identity on $\R^{n+1}\setminus U_r^S(z)$.
For any $x\in G_i\cap \closedhalf$, consider a line segment $I$ with two ends, $x_0$ and $x$ (connect $x$ to $\sigma(x_0)$ instead in the negative halfspace).  
Since $x\in G_i=\partial B_r^S(z)\cap E_i^S$, there is some neighborhood of $x$ of $I$
belonging to $E_i^S$. On the other hand, we have $B_{r_0}(x_0)\subset E^S_{i_0}$, 
thus there must be 
some point $\widehat x\in I\cap \partial E^S_{i_0}$. Since $f$ on $B_r^S(z)\setminus 
B_{r_0}^S(x_0)$ is a (symmetrized) radial projection to $\partial B_r^S(z)$, $f(\widehat x)=x$. This proves
that $G_i\subset f(\partial E_{i_0}^S)$. Then \eqref{rlm6} shows $\partial
\widetilde E_i^S\subset f\left(\partial E_i^S\cup \partial E_{i_0}^S\right)$ for $i\neq i_0$. 
For $i=i_0$, $\partial
E^S_{i_0}\setminus B_r^S(z)=f\left(\partial E^S_{i_0}\setminus B_r^S(z)\right)$
since $f$ is the identity there. For any $x\in \partial B_r^S(z) \setminus G_{i_0}=\partial
B_r^S(z)\setminus E_{i_0}$, either
$x\in \partial E_i^S$ for some $i$ (including $i=i_0$), or $x\in E_i^S$ for some 
$i\neq i_0$. In the former case, since $f$ is the identity on $\partial B_r^S(z)$, 
$x\in f(\partial E_i^S)$. In the latter case, the line segment connecting $x_0$
and $x$ contains $\widehat x\in \partial E^S_{i_0}$ just as before, hence $x\in f(\partial
E^S_{i_0})$. Thus by \eqref{rlm7}, we have $\partial\widetilde E_{i_0}^S\subset 
f\left(\cup_{i=1}^N\partial E^S_{i}\right)$. In conclusion, we have proved that $\cup_{i=1}^N
\partial\widetilde E^S_i\subset f(\cup_{i=1}^N \partial E^S_i)$, and this proves that
$f$ is symmetric $\mathcal E^S$-admissible.

Let us show $(5)$. Any point $y\in U_r^S(z)\cap H_0$ is an interior point for $\widetilde E_{i_0}^S$; hence,
\begin{equation*}
    H_0 \cap \partial \widetilde \E^S\subseteq \left((H_0 \cap \partial \E^S) \setminus B_r^S(z) \right) \cup (H_0 \cap \partial B_r^S(z))\, ,
\end{equation*}
namely the map generates new boundary points on $H_0$ only in the intersection $H_0\cap \partial B_r^S(z)$, which is a set of dimension $n-1$. Therefore, since we are assuming that $\Ha^n( H_0\cap \partial \E^S)=0$, we get (5). 

With 
$\widetilde{\E}^S =f_{\star}\E^S=\{\widetilde E_i^S\}_{i=1}^N$, 
we have from \eqref{rlm6}, \eqref{rlm7} and $\cup_{i\neq i_0}G_i\subset \partial
B_r^S(z)\setminus G_{i_0}$ that
\begin{equation}
\label{e: 7.14 KT}
\begin{split}
\normpiccolo{\partial \widetilde{\E}^S}(B_r^S(z))&=\Ha^{n}\left(
\cup_{i=1}^N\partial \widetilde E_i^S\cap B_r^S(z)\right)\\
&\leq 
\Ha^{n}(\partial B_r^S(z)\setminus G_{i_0})+\sum_{i\neq i_0}
\Ha^{n}(\partial E_i^S\cap \partial B_r^S(z)) \\
&=\Ha^{n}(\partial B_r^S(z)\setminus G_{i_0})\, ,
\end{split}
\end{equation}
the last equality due to \eqref{e: 7.3 KT}. We next note that $E_i^S\triangle
\widetilde E_i^S=E_i^S\cap B_r^S(z)$ for $i\neq i_0$ and $E_i^S\triangle
\widetilde E_{i_0}^S=U_r^S(z)\setminus E^S_{i_0}$. Since both are included in $B_r^S(z)\setminus E_{i_0}$, 
\eqref{e: 7.7 KT} shows that the condition (4) is satisfied with  this 
$c_2$. Thus we conclude that the symmetric $\mathcal E^S$-admissible 
function $f$ satisfies conditions (1), (2), (4), (5) so far.

If the conclusion were not true, then, we must have $\normpiccolo{\partial\widetilde\E^S}
(B_r^S(z))> \frac12 \normpiccolo{\partial \mathcal E^S}(B_r^S(z))=\frac12\eta(r)$ if $r\in A$
with \eqref{e: 7.3 KT}. 
Combining \eqref{e: 7.14 KT}, \eqref{e: 7.11 KT} and \eqref{e: 7.3 KT}, we obtain
\begin{equation*}
\begin{split}
\frac12 \eta(r)&\leq  \Ha^n\left(\partial B_r^S(z)\setminus G_{i_0}\right) \leq c_3 \left( \Ha^{n-1} (\partial G_{i_0} ) \right)^\frac{n}{n-1}\\
&\leq c_3 (\eta'(r))^{\frac{n}{n-1}}\, .
\end{split}
\end{equation*} 
Since we have $\mathcal L^1(A)\geq
\frac{R}{12}$ by \eqref{e: 7.9 KT}, 
\begin{equation*}
\eta^{\frac{1}{n}}(R)\geq \int_{\frac34 R}^R \left(\eta^\frac1n(r)\right)'\, dr\ge \frac1n
\int_A \eta^{\frac{1-n}{n}}(r) \, \eta'(r)\,dr\geq n^{-1}(2c_3)^{\frac{1-n}{n}} \frac{R}{12}\, ,
\end{equation*}
and then 
\begin{equation*}
   \eta(R) \ge \frac{R^n}{(12\,n)^n\, (2\,c_3)^{n-1}} \, .
\end{equation*}
We would obtain 
a contradiction to $\normpiccolo{\partial\mathcal E^S}\bigl(B_R^S(z)\bigr)=\eta(R)\leq c_1\, R^{n}$ by
choosing an appropriately small $c_1$ depending only on $n$.
\end{proof}

\begin{theorem}
\label{theorem: rectifiability}
Suppose that $\{\mathcal E_j^S\}_{j=1}^{\infty}\subset\opn$ is a family of symmetric open partitions with $\normpiccolo{\partial \E^S_j}(H_0)=0$ and $\{\eps_j\}_{j=1}^{\infty}\subset (0,1)$
satisfy
\begin{enumerate}
\item
$\lim\limits_{j\rightarrow\infty}j^4 \eps_j =0$,
\item
$\sup_j \normpiccolo{\partial\mathcal E^S_j}(\R^{n+1})<\infty$,
\item $\displaystyle
\liminf\limits_{j\rightarrow\infty}
\int_{\R^{n+1}} \frac{\abspiccolo{\Phi_{\eps_j}\ast\delta(\partial\E^S_j)}^2
}{\Phi_{\eps_j}\ast\|\partial\E^S_j\|+\eps_j}\, dx
<\infty$,
\item
$\lim\limits_{j\rightarrow\infty} \Deltajvc \normpiccolo{\partial\mathcal E^S_j}(\R^{n+1})=0$.
\end{enumerate}
Then there exists a converging subsequence $\setpiccolo{\partial \mathcal E_{j_\l}^S}_{\l=1}^{\infty}$
whose limit $V^S\in{\bf V}_{n}(\R^{n+1})$ is symmetric and it satisfies, for some dimensional constant $c>0$,
\begin{equation*}
\theta^{\ast n}\left(\normpiccolo{V^S},x\right)\geq c \mbox{ for $\normpiccolo{V^S}$ a.e.~$x$}\, .
\end{equation*}
Furthermore, $V^S\in {\bf RV}_{n}(\mathbb R^{n+1})$. 
\begin{proof}
The interior regularity immediately follows by repeating the arguments of \cite[Theorem $7.3$]{KimTonegawa} and \cite[Appendix $A.2$]{STCanonical}), with no modifications. Indeed, the proof is local and for any $x\not \in H_0$, it sufficies to take $j$ sufficiently large so that $\abs{ \d}(x)\ge \frac{1}{j^2}$ to avoid any boundary effect. Let us focus on the boundary rectifiability. We basically follow again the proof of \cite[Theorem $7.3$]{KimTonegawa}, making use of our area-reducing symmetric Lipschitz map built in \autoref{prop: area reducing symmetric radial projection}. We briefly describe their proof, focusing on the required modifications.

Let us fix $x_0\in H_0$, which we can assume to be the origin after a translation. The existence of a converging subsequence $\partial \E^S_{j_\l}$ is guaranteed by the compactness of Radon measures. Define, for $R\in(0,1)$,
\begin{equation*}
    F_R\define \set{x\in B_1(x_0) \, \colon \, R^{-n}\normpiccolo{ V^S}(B_R)<\frac{c_1}{16} }\, ,
\end{equation*}
    where $c_1$ is the constant given by \autoref{prop: area reducing symmetric radial projection}. The aim is to show that $\lim\limits_{R\to 0^+} \normpiccolo{V^S}(F_R)=0$, implying the thesis. Next, for $m\in \Na$ we introduce 
    \begin{equation*}
        F_{R,m}\define \set{x\in F_R \, \colon \, \normpiccolo{\Phi_{\eps_{j_\l}} \ast \partial \E^S_{j_\l}} (B_R(x))< \frac{c_1}{16} \ \mbox{for all} \ \l\ge m}\, ,
    \end{equation*}
    and we note there exists $m_1 \gg 1$ with 
    \begin{equation*}
        \normpiccolo{V^S}(F_{R,m_1}) \ge \frac12 \normpiccolo{V^S}(F_R)\, .
    \end{equation*}
    Next, we define 
    \begin{equation*}
        G_R\define \set{x\in \R^{n+1}\, \colon \, \mathrm{dist}(x,F_{R,m_1})<(1-2^\frac{-1}{n})R}\, .
    \end{equation*}
    By \cite[Equation $(7.27)$]{KimTonegawa}, we have 
    \begin{equation*}
        \normpiccolo{\partial \E_{j_\l}^S}(G_R) \ge \frac14 \normpiccolo{V^S}(F_R)\, ,
    \end{equation*}
    thus it suffices to show that $\normpiccolo{\partial \E_{j_\l}^S}(G_R) \rightarrow 0$ as $R\rightarrow 0$. Define 
    \begin{gather}
    \label{e: def Grj1}
            G_{R,j_\l,1}\define \set{x\in G_R \, \colon \, \theta^n(\normpiccolo{\partial \E_{j_\l}^S},x)=1 \ \mbox{and} \ (2j_\l^2)^n\normpiccolo{\Phi_{\eps_{j_\l}} \ast \partial \E_{j_\l}^S}\left(B^S_{\sfrac{1}{2j_\l^2}}(x)\right) >\frac{c_1}{4} }\\
             \label{e: def Grj2}
           G_{R,j_\l,2}\define \set{x\in G_R \, \colon \, \theta^n(\normpiccolo{\partial \E_{j_\l}^S},x)=1 \ \mbox{and} \ (2j_\l^2)^n\normpiccolo{\Phi_{\eps_{j_\l}} \ast \partial \E_{j_\l}^S}\left(B^S_{\sfrac{1}{2j_\l^2}}(x)\right) \leq \frac{c_1}{4} }\, .
     \end{gather}
    The only differences with respect to the ones in \cite{KimTonegawa} is that we are using symmetric balls. Concerning $G_{R,j_\l,1}$, \eqref{e: def Grj1} implies 
    \begin{equation}
    \label{e: 7.29 KT}
        (2j_\l^2)^n\normpiccolo{\Phi_{\eps_{j_\l}} \ast \partial \E_{j_\l}^S}\left(B_{\sfrac{1}{2j_\l^2}}(x)\right) >\frac{c_1}{8}
    \end{equation}
    and we can then mostly reproduce the computations of \cite{KimTonegawa}. Indeed, in \cite[Equation $(7.28)$]{KimTonegawa}, the authors take $m_3\in \Na,$ $m_3\gg 1$ so that 
    \begin{equation*}
    \frac{1}{2j_{m_3}^2}< \frac{R}{2}\, 
    \end{equation*}
    whereas we need to take it slightly smaller by a constant factor, since \eqref{e: 7.29 KT} differs from the original one by a factor $\frac12$. This is the only required modification to estimate $G_{R,j_\l,1}$, since the other arguments only require the monotonicity formula \eqref{e: monotonicity formula} and the Besicovitch covering theorem. By \cite[Equation $(7.38)$]{KimTonegawa}, we have 
    \begin{equation*}
      \lim\limits_{R\to 0}  \limsup\limits_{\l \to \infty} \normpiccolo{\partial \E_{j_\l}^S} (G_{R,j_\l,1})=0\, .
    \end{equation*}
    
    Let $z\in G_{R,j_\l,2}$ now. Since $\eps_{j_\l} \ll \frac{1}{2j_\l^2}$ for all large $\l$ by assumption (1), we have 
    \begin{equation*}
        \Phi_{\eps_{j_\l}} \ast \chi_{B^S_{\sfrac{1}{2j_\l^2}}(z)}\ge \frac14 \quad \mbox{on} \ B^S_{\sfrac{1}{2j_\l^2}}(z) 
    \end{equation*}
    and 
    \begin{equation}
    \label{e: 7.30 KT}
        \normpiccolo{\partial \E^S_{j_\l}} \left(B^S_{\sfrac{1}{2j_\l^2}}(z)\right) \leq 4 \normpiccolo{\Phi_{\eps_{j_\l}} \ast \partial \E^S_{j_\l}}\left(B^S_{\sfrac{1}{2j_\l^2}}(z)\right)
    \end{equation}
    By \eqref{e: def Grj2} and \eqref{e: 7.30 KT}, we have 
    \begin{equation*}
        (2j_\l^2)^n\normpiccolo{\partial \E^S_{j_\l}} \left(B^S_{\sfrac{1}{2j_\l^2}}(z)\right) \leq c_1\, .
    \end{equation*}
    Therefore, for every $z\in G_{R,j_\l,2}$ there exist $r_z\in \left[ \frac{3}{8j_\l^2}, \frac{1}{2j_\l^2}\right] $ and a symmetric $\E^S_{j_\l}$-admissible map $f_z$ satisfying the conclusions of \autoref{prop: area reducing symmetric radial projection}. We apply the Besicovitch covering theorem to the family $\setpiccolo{B_{r_z}(z) \ \colon \ z\in G_{R,j_\l,2}\cap \closedhalf}$ to get a finite set $\set{z_k}_{k=1}^\Lambda\subset G_{R,j_\l,2}\cap \closedhalf$ such that $\setpiccolo{B_{r_{z_k}} (z_k)}_{k=1}^\Lambda$ is mutually disjoint and, writing $B_{r_{z_k}} (z_k)$ as $B_k$, we have 
    \begin{equation}
    \label{e: 7.41 KT not symmetric}
        \normpiccolo{\partial \E^S_{j_\l} }\left( \cup_{k=1}^\Lambda B(k) \right) \ge \frac{1}{\mathbf B_{n+1}} \normpiccolo{\partial \E^S_{j_\l} } \left(G_{R,j_\l,2} \cap \closedhalf\right)\, ,
    \end{equation}
    where $\mathbf B_{n+1}$ is a dimensional constant. Note that the finiteness of $\Lambda$ follows from $r_z\ge \frac{3}{8j_\l^2}$ and the fact $G_R$ is bounded.
    Consider the family $\setpiccolo{B^S(k)}_{k=1}^\Lambda \define \setpiccolo{ B(k) \cup \sigma (B(k))}_{k=1}^\Lambda$, which is still a family of pairwise disjoint sets. Note that, similarly to \eqref{e: measure BrS}, 
    \begin{equation}
    \label{e: proof thm 7.3 1}
        \normpiccolo{\partial \E^S_{j_\l} }\left( \cup_{k=1}^\Lambda B(k) \right)+ \normpiccolo{\partial \E^S_{j_\l} }\left( \cup_{k=1}^\Lambda \sigma(B(k)) \right)\leq 2\, \normpiccolo{\partial \E^S_{j_\l} }\left( \cup_{k=1}^\Lambda B^S(k) \right)\, ;
    \end{equation} 
    by \eqref{e: 7.41 KT not symmetric}, $\normpiccolo{\partial \E^S_j}(H_0)=0$ and \eqref{e: proof thm 7.3 1}, we get 
    \begin{equation}
    \label{e: 7.41 KT}
       2 \normpiccolo{\partial \E^S_{j_\l} }\left( \cup_{k=1}^\Lambda B^S(k) \right) \ge \frac{1}{\mathbf B_{n+1}} \normpiccolo{\partial \E^S_{j_\l} } \left(G_{R,j_\l,2}\right)\, .
    \end{equation}
    With this choice, define $f\colon \R^{n+1}\rightarrow\R^{n+1}$ by
    \begin{equation*}
        f(x) \define \begin{cases*}
            f_{z_k}(x) & if $x\in B^S(k)$ for some $k\in \set{1,\ldots,\Lambda}$\\
            x & otherwise\, .
        \end{cases*}
    \end{equation*}
    Since $f_{z_k}$ is $\E^S_{j_\l}$-admissible, so is $f$ due to the disjointness of $B^S(k)$. One can check $f\in \classEvc(\E, j,\operatorname{Sym})$ as in \cite{KimTonegawa,STCanonical}. By \eqref{ar1}, (3) of \autoref{prop: area reducing symmetric radial projection} and \eqref{e: 7.41 KT}, we have
    \begin{equation*}
        \begin{split}
            \Deltajvc\normpiccolo{\partial \E^S_{j_\l}}(\R^{n+1})& \leq \normpiccolo{\partial f_\star \partial \E^S_{j_\l}}(\R^{n+1})-\normpiccolo{\partial \E_{j_\l}^S}(\R^{n+1})\\
            & =\sum_{k=1}^\Lambda \normpiccolo{\partial f_\star \partial \E^S_{j_\l}}(B^S(k))-\normpiccolo{\partial \E_{j_\l}^S}(B^S(k))\\
            &\leq -\frac12 \sum_{k=1}^\Lambda \normpiccolo{\partial \E_{j_\l}^S}(B^S(k))\\
            &\leq \frac{1}{4 \mathbf B_{n+1}} \normpiccolo{\partial \E^S_{j_\l} } \left(G_{R,j_\l,2}\right)\, .
        \end{split}
    \end{equation*}
    Finally, assumption (4) implies the thesis.
\end{proof}
\end{theorem}

\section{Integrality of the symmetric flow}
\label{section: integrality}
The first part of the proof of the integrality proposed by \cite[Section 8]{KimTonegawa} aims at showing that, given an approximate solution to the MCF obtained by \eqref{app:0}-\eqref{app:future}, say $V$, whenever we fix a point $x$ and a set $Y$ consisting of density-$1$ points in the orthogonal complement of the approximate tangent space of $V$ at $x$, then the mass of $V$ around $Y$ has to be at least the one of $\Ha^0(Y)$ discs. In other words, $V$ should be made of as many sheets as $\Ha^0(Y)$. To establish this property, precisely stated in \cite[Lemma $8.5$]{KimTonegawa}, the authors partition $Y$ into many subsets $Y_1,\ldots,Y_J$ of small diameter, where they can use specific Lipschitz maps as competitors for the almost area minimizing process encoded in the first step of the algorithm. Similarly to the rectifiability, the interior integrality follows immediately, since for fixed $x$ we can take $j$ sufficiently large to neglect boundary effects. We therefore mainly focus on boundary regularity. The aforementioned Lipschitz maps are naturally symmetric whenever the sets $Y_k$ are symmetric, which requires us to make a symmetric partition of $Y$. However, if $Y$ contains points far from the boundary, no symmetric set containing such points will have a small diameter. We therefore replace the smallness of $\mathrm{diam}\, Y_k$ by the smallness of $\mathrm{diam}\, (Y_k \cap \R^{n+1}_+)=\mathrm{diam}\, (Y_k \cap \R^{n+1}_-)$. In any case, while performing the decomposition of the set $Y$, if something is sufficiently far from $H_0$ (at a distance grater than $\sfrac{1}{j^2}$), we can treat it with Kim-Tonegawa's results. Once this symmetric decomposition is done, the rest of the proof requires basically no modifications.
\vspace{0.5cm}

\noindent In the following, by analogy with the notation of \cite[Section $8$]{KimTonegawa}, we will denote by $T$ the matrix representing the projection onto the hyperplane $H_0$ and $T^\perp$ its orthogonal complement. Given $Y\subset H_0^\perp$ and $r_1,\, r_2\in (0,\infty)$, define a closed set 
\begin{equation*}
    E(r_1,r_2)\define \set{x\in \R^{n+1} \, \colon \, \abs{T(x)}\leq r_1, \ \mathrm{dist}\left(T^\perp (x), Y\right) \leq r_2}\, .
\end{equation*}
\begin{lemma} 
\label{lemma 8.1}
Corresponding to $n,\nu \in \Na$, $\alpha\in(0,1)$ and $\zeta\in (0,1)$, there exist $\gamma\in (0,1)$ and $j_0\in \Na$ with
the following property.
Assume 
\begin{enumerate}
\item
$\E^S=\{E^S_i\}_{i=1}^N \in \opns$ is a symmetric open partition, $j\in 
\Na$ with $j\geq j_0$, 
$R\in \left(0,\frac12 j^{-2}\right)$,  $\rho\in \left(0,\frac12 j^{-2}\right)$;
\item
$\rho\geq \alpha R$;
\item 
$Y\subset H_0^{\perp}$ is a symmetric set, it has no more than $2\, \nu$ elements, $Y \,\cap\,  H_0=\emptyset$, $\theta^{n}(\|\partial \mathcal E^S\|,y)=1$ for all $y\in Y$, $\mathrm{diam}\, Y_+=\mathrm{diam}\, Y_-<j^{-2}$ and $\normpiccolo{\partial \E^S}(H_0)=0$;
\item
$\int_{{\bf G}_{n}(E^\ast(r))}\|S-T\|\, d(\partial \mathcal E)
(x,S)\leq \gamma \|\partial\mathcal E\|(E^\ast(r))$ for all $r\in (0,R)$;
\item 
$\Deltajvc \|\partial \mathcal E^S\|(E^\ast(r))
\geq -\gamma \|\partial\mathcal E^S\|(E^\ast(r))$ for all $r\in (0,R)$,
\end{enumerate}
where we denoted $E^\ast(r)\define E(r,(1+R^{-1}r)\rho)$ for short.\\
Then we have
\begin{equation}
\normpiccolo{\partial \mathcal E^S}(E(R,2\rho))\geq (\mathcal H^0(Y)-\zeta)\, \omega_{n} R^{n}.
\label{itg2}
\end{equation}
\begin{remark}
The only difference with the statement of \cite[Lemma $8.1$]{KimTonegawa} is $(3)$.
\end{remark}
\begin{proof}
The proof adopted by \cite{KimTonegawa} argues by contradiction, showing that if \eqref{itg2} were false, then, in terms of measure, there are less than $\Ha^0(Y)$ discs and one could build a Lipschitz deformation drastically reducing the measure of $\partial \E^S$ by expanding an hole, yielding to a contradiction with $(5)$. The only crucial point is whether this map is symmetric and whether the smallness of the diameter of $Y_+$ and $Y_-$ can replace the smallness of the diameter of $Y$. In particular, \cite{KimTonegawa} assumed the smallness of the diameter of $Y$ to show the $\E^S$-admissibility of the map.

For the reader's convenience, we report some notations introduced by \cite[Lemma $8.1$]{KimTonegawa}, referring to it for detailed discussions. Let $a\in H_0$, $\delta>0$ small, $0<r_1<R$, $\rho_1=(1+R^{-1}r_1)\rho, \ \  \xi\in \left(0, \frac{\rho_1\, r_1}{R}\right) , \ \ a^\ast =\frac{ r_1}{r_1-\delta} \, a$ and 
\begin{gather*}
    C(T,a,\delta) \define \set{x\in \R^{n+1}\, \colon \, \abs{T(x)-a)}\leq \delta }\\
    E_1(a) \define \set{x\in C(T,a,\delta)\, \colon \, \abs{T(x)- a^\ast}\leq 2\, \delta \, \xi^{-1}(\rho_1- \mathrm{dist}(T^\perp(x), Y)}\\
    E_2(a) \define \set{x\in C(T,0,r_1) \setminus E_1(a) \, \colon \, \abs{T(x)-a^\ast}\leq 2\, r_1\, \xi^{-1} (\rho_1-\mathrm{dist}(T^\perp(x), Y)) }\, .
\end{gather*}
Note that, since $a,a^\ast\in H_0$ and $Y$ is symmetric, these sets are symmetric. The Lipschitz map that is needed to get a contradiction, $f_a$, is the identity outside $E_1(a) \cup E_2(a)$ and it radially expands, in the $n$ directions given by $H_0$, each ``horizontal slice’’ of $E_1(a)$. This shows that $f_a$ is actually symmetric. Moreover, since any horizontal slice of the form $H_0\times \set{t}$ is mapped in a subsets of itself, $f_a$ is Lipschitz and by assumption $\normpiccolo{\partial \E^S}(H_0)=0$, then 
\begin{equation*}
    \normpiccolo{\partial ({f_a}_\ast \E^S)}(H_0)=0.
\end{equation*}
that is (d) of \autoref{ar8} is satisfied. 

Next, we want to show that for all $\varphi \in \cA_j$, $\normpiccolo{\partial (f_a)_\star \E^S}(\varphi)\leq \normpiccolo{\partial \E^S}(\varphi)$. The authors notice that, in particular, $f_a$ is the identity outside $E(r_1, \rho_1)$. We distinguish in two cases:
\begin{itemize}
    \item Suppose $\mathrm{diam} \, Y<4\, j^{-2}$. With respect to the original proof, we just need to further restrict $j_0$ depending only on $\nu$ and $\zeta$ (precisely, see \cite[Equation $(8.4)$]{KimTonegawa}.
    \item Suppose $\mathrm{diam}\, Y\ge 4\, j^{-2}$, with still $\mathrm{diam}\, Y_+< j^{-2}$ by assumption. This means that we can decompose 
    \begin{equation*}
    E(r_1,\rho_1)= E(r_1, \rho_1)_+ \cup E(r_1, \rho_1)_-\,,
    \end{equation*}
    where the two compact sets, that we denote $E_+$ and $E_-$ for short, are disjoint. Indeed, let $y,z\in Y_+$ be the points of $Y_+$ achieving, respectively, the maximum and minimum distance from $H_0$. Due to the symmetry of $Y$, it follows that $\mathrm{diam} \, Y= 2\,\d(y)$ and $\mathrm{diam}\, Y^+=\abs{y-z}$. The assumption then rewrites as 
    \begin{equation*}
        \d(y) \ge 2\, j^{-2}, \quad \abs{y-z} < j^{-2}\,,
    \end{equation*}
    which implies 
    \begin{equation*}
    \d(z) >j^{-2}\, ,
    \end{equation*}
    since $\d(y) = \d(z) + \abs{y-z}$. If $x\in E(r_1,\rho_1)$, then either $\d(\sigma(y))-\rho_1\leq\d(x)\leq\d(\sigma(z))+\rho_1$ or $\d (z)-\rho_1\leq \d(x) \leq \d(y)+\rho_1$ and $\rho_1\in (0,j^{-2})$, showing that $E_+$ and $E_-$ are disjoint. The thesis then follows by \autoref{lem: massive drop area implies admissible} together with the computations of \cite{KimTonegawa}.
\end{itemize}
Finally, (b) of \autoref{ar8}, namely the property of being a change of volume-controlled deformation was investigated by \cite[Appendix $A.2$]{STCanonical} and it can be adapted to our assumption (3) similarly to the above described case. 
\end{proof}
\end{lemma}
\begin{remark}
\label{remark: lemma 8.1 Y far from H0}
    Consider all the assumptions of \autoref{lemma 8.1} but (3), and replace it with 
    \begin{itemize}
        \item [3')] $Y\subset H_0^\perp$, it has no more than $\nu$ elements, $\mathrm{dist} (Y,H_0)> 2\rho$ and $\mathrm{diam}\, Y<j^{-2}$.
    \end{itemize}
    Then, \eqref{itg2} still holds, possibly with slightly different $\gamma$ and $j_0$, but still depending only on the same parameters. Indeed, assuming $Y\subset \R^{n+1}_+$ for simplicity, we can see that for each $i=1,2$, 
    \begin{equation*}
    \mathrm{dist} (E_i(a), H_0) > \mathrm{dist} (Y,H_0)-2\rho_1>0\, ,
    \end{equation*}
    since $2\rho_1 < 2\rho < \mathrm{dist}(Y,H_0)$ by assumption. This means that we can build a symmetric  function $\widetilde f_a$ which is $f_a$ on $\R^{n+1}_+$ (in the notation of the previous proof), yielding the thesis.
\end{remark}

The next is \cite[Lemma $6.1$]{AllardFirstVariation} (see also \cite[Lemma $8.3$]{KimTonegawa}). We will make use of this for sets contained in a single halfspace which are far enough from $H_0$. We do not present the proof, but the main ideas will be showed in the next Lemma, which is a symmetric version of this.
\begin{lemma}
\label{lemma4.21 KT}
Suppose
\begin{enumerate}
\item
$\nu\in \Na$, $\xi\in (0,1)$, $M\in (1,\infty)$, $0<r_0<R<\infty$, and $V\in {\bf V}_n({\R}^{n+1})$,
\item
$Y\subset H_0^{\perp}$ has no more than $\nu+1$ elements,
\item
$(M+1){\rm diam}\, Y\leq R$,
\item
$r_0<\frac{\mathrm{diam} \,Y}{3(\nu+1)}$,
\item
$R\|\delta V\|(B_r(y))\leq \xi \|V\|(B_r(y))$ for all $y\in Y$ and $r\in (r_0,R)$,
\item
$\int_{{\bf G}_{n}(B_r(y))}\|S-T\|\, dV(x,S)\leq \xi \|V\|(B_r(y))$ for all $y\in Y$ and $r
\in (r_0,R)$.
\end{enumerate}
Then there are $V_1,\, V_2 \in {\bf V}_n({\R}^{n+1})$ and a partition of $Y$ 
into subsets $Y_0,\, Y_1,\, Y_2,$ such that 
\begin{equation}
V\geq V_1+V_2 \, ,
\label{4-21-7 KT}
\end{equation}
\begin{equation}
\mbox{neither } Y_1\mbox{ nor } Y_2\mbox{ has more than }\nu\mbox{ elements},
\label{4-21-8 KT}
\end{equation}
\begin{equation}
(M\,{\rm diam}\, Y)\norm{\delta V_k}(B_r(y))\leq 2M(\nu+1)(3\nu\, M)^{n+1}(\exp \xi)\xi \norm{V_k}(B_r(y))
\label{4-21-9 KT}
\end{equation}
\centerline{for all $y\in Y_k$, $r\in (r_0,M\, {\rm diam}\, Y)$ and $k=1,2$\, ,}
\begin{equation}
\int_{{\bf G}_n(B_r(y))}\norm{S-T}\, dV_k(x,S)\leq M(3\nu\, M)^n (\exp\xi)\xi \norm{V_k}(B_r(y))
\label{4-21-10 KT}
\end{equation}
\centerline{for all $y\in Y_k$, $r\in (r_0,M\, {\rm diam}\, Y)$ and $k=1,2$\, ,}
\begin{equation}
V_k\geq V\mres\{x\in {\mathbb R}^{n+1}:{\rm dist}\, (T^{\perp}(x),Y_k)\leq r_0\} \times \bG(n+1,n)\,\, \mbox{ for } k=1,2 \, 
\label{4-21-11 KT}
\end{equation}
\begin{equation}
\begin{split}
&\set{\left(1+\frac1M\right)^n+\frac{\nu+1}
{M}}  (\exp\xi)\frac{\norm{V}(\set{x:{\rm dist}\, (x,Y)\leq R})}{\omega_n R^n} \\
&\geq \sum_{y\in Y_0} \frac{\norm{V}(B_{r_0}(y))}{\omega_n r_0^n} 
+\sum_{k=1,2} \frac{\norm{V_k}(\set{x :{\rm dist}\,(x,Y_k)\leq M\, {\rm diam}\, Y})}{\omega_n 
(M\, {\rm diam}\, Y)^n}\, .
\end{split}
\label{4-21-12 KT}
\end{equation}
\end{lemma}
\begin{remark}
\label{remark: lemma 8.3 KT for symmetric sets}
    Looking at the original proof by \cite{AllardFirstVariation}, one can see that the described partition can be made in such a way that whenever $V$ is a symmetric varifold and $Y\subset H_0^\perp$ is a symmetric set of at most $2\, \nu+2$ elements, such that both $Y_+$ and $Y_-$ satisfy the assumption of \autoref{lemma4.21 KT}, then the partitions obtained, say $Y_+=Y_0\cup Y_1\cup Y_2$ and $Y_-=Z_0\cup Z_1\cup Z_2$, satisfy $\sigma(Y_k)=Z_k$ for $k=1,2,3$. Similarly, if $V_1,V_2$ and $W_1,W_2$ are the varifold associated to respectively $Y_+$ and $Y_-$, then $\sigma_\sharp V_k= W_k$ for $k=1,2$.
\end{remark}

\begin{lemma}
\label{lemma4.21}
Suppose
\begin{enumerate}
\item
$\nu\in \Na$, $\xi\in (0,1)$, $M\in (1,\infty)$, $0<r_0<R<\infty$, and $V\in {\bf V}_n({\R}^{n+1})$ is a symmetric varifold,
\item
$Y\subset H_0^{\perp}$ is symmetric, it has no more than $2\,\nu+2$ elements and $Y\cap H_0 =\emptyset$,
\item
$(M+1){\rm diam}\, Y\leq R$,
\item
$r_0<\frac{\mathrm{diam} \,Y}{3(2\,\nu+1)}$,
\item
$R\|\delta V\|(B_r(y))\leq \xi \|V\|(B_r(y))$ for all $y\in Y$ and $r\in (r_0,R)$,
\item
$\int_{{\bf G}_{n}(B_r(y))}\|S-T\|\, dV(x,S)\leq \xi \|V\|(B_r(y))$ for all $y\in Y$ and $r
\in (r_0,R)$.
\end{enumerate}
Then there are $V_1,\, V_2, \, V_3 \in {\bf V}_n({\R}^{n+1})$ and a partition of $Y$ 
into subsets $Y_0,\, Y_1,\, Y_2, \, Y_3$ such that 
\begin{equation}
\label{e: lemma 6.1 geometry Yk}
   Y_1=\sigma(Y_2) \  \mbox{and} \ V_1=\sigma_\sharp V_2
\end{equation}
\begin{equation}
\label{e: lemma 6.1 geometry Y3}
    Y_0 \ \mbox{and} \ Y_3 \ \mbox{are symmetryc}\footnote{some of these sets may be empty, in which case we still call them symmetric for simplicity} \ \mbox{and} \ V_3 \ \mbox{is symmetric} \, ,
\end{equation}
\begin{equation}
V\geq V_1+V_2+V_3 \, ,
\label{4-21-7}
\end{equation}
\begin{equation}
\label{4-21-8}
\mbox{neither } Y_1\mbox{ nor } Y_2\mbox{ has more than }\nu+1 \mbox{ elements}\, ,
\end{equation}
\begin{equation}
\label{4-21-8-2}
Y_3 \mbox{ has no more than } 2\,\nu \mbox{ elements}\, ,
\end{equation}
\begin{equation}
(M\,{\rm diam}\, Y)\|\delta V_j\|(B_r(y))\leq 2M(2\,\nu+2)(3(2\,\nu+1) M)^{n+1}(\exp \xi)\xi \|V_j\|(B_r(y))
\label{4-21-9}
\end{equation}
\centerline{for all $y\in Y_k$, $r\in (r_0,M\, {\rm diam}\, Y)$ and $k=1,2,3$\, ,}
\begin{equation}
\int_{{\bf G}_n(B_r(y))}\|S-T\|\, dV_k(x,S)\leq M(3(2\,\nu+1) M)^n (\exp\xi)\xi \|V_k\|(B_r(y))
\label{4-21-10}
\end{equation}
\centerline{for all $y\in Y_k$, $r\in (r_0,M\, {\rm diam}\, Y)$ and $k=1,2,3$\, ,}
\begin{equation}
V_k\geq V\mres\set{x\in {\mathbb R}^{n+1}:{\rm dist}\, (T^{\perp}(x),Y_k)\leq r_0} \times \bG(n+1,n)\,\, \mbox{ for } k=1,2,3 \, 
\label{4-21-11}
\end{equation}
\begin{equation}
\begin{split}
&\set{\left(1+\frac1M\right)^n+\frac{2(\nu+1)}
{M}}  (\exp\xi)\frac{\norm{V}(\{x:{\rm dist}\, (x,Y)\leq R\})}{\omega_n R^n} \\
&\geq \sum_{y\in Y_0} \frac{\norm{V}(B_{r_0}(y))}{\omega_n r_0^n} 
+\sum_{k=1,2,3} \frac{\norm{V_k}(\set{x :{\rm dist}\,(x,Y_k)\leq M\, {\rm diam}\, Y})}{\omega_n 
(M\, {\rm diam}\, Y)^n}.
\end{split}
\label{4-21-12}
\end{equation}
\begin{remark}
    Note that \eqref{4-21-7} and \eqref{4-21-11} imply that 
    \begin{equation}
    \label{e: lemma 6.1 Y1 and Y2 are far}
    \begin{split}
        &\norm{V} \left( \set{x\in \R^{n+1} \, \colon \, \mathrm{dist}\left(T^\perp(x), (Y_1\cup Y_2)\right)\leq r_0)} \right)\\
        & \qquad= \norm{V} \left(\set{x:{\rm dist}\, (x,Y_1)\leq r_0}\right)+ \norm{V} \left(\set{x:{\rm dist}\, (x,Y_2)\leq r_0}\right)\, .
        \end{split}
    \end{equation}
    Indeed, 
    \begin{equation*}
        \begin{split}
           &\norm{V} \left( \set{x\in \R^{n+1} \, \colon \, \mathrm{dist}\left(T^\perp(x), (Y_1\cup Y_2)\right)\leq r_0)} \right)\\
           &\qquad \ge \norm{V_1+V_2} \left( \set{x\in \R^{n+1} \, \colon \, \mathrm{dist}\left(T^\perp(x), (Y_1\cup Y_2)\right)\leq r_0)} \right) \\
           &\qquad  \ge \sum_{k=1,2} \norm{V_k} \left(\set{x:{\rm dist}\, (x,Y_k)\leq r_0}\right)\\
           &\qquad \ge \sum_{k=1,2} \norm{V} \left(\set{x:{\rm dist}\, (x,Y_k)\leq r_0}\right)\\
           &\qquad \ge \norm{V} \left( \set{x\in \R^{n+1} \, \colon \, \mathrm{dist}\left(T^\perp(x), (Y_1\cup Y_2)\right)\leq r_0)} \right) \, .
        \end{split}
    \end{equation*}
\end{remark}
\begin{proof}
    Due to assumption $(5)$ and \autoref{lem: monotonicity formula}, we have 
    \begin{equation}
    \label{e: lemma 6.1 monotonicity}
        \frac{\norm{V} (B_r(y))}{\omega_n \, r^n}\leq (\exp \xi)\,  \frac{\norm V(B_s(y))}{\omega_n \, s^n}
    \end{equation}
    for every $y\in Y$ and $r_0<r<s<R$.  Let 
    \begin{equation*}
        \eta\define M\, \mathrm{diam}\, Y, \quad \widetilde \nu\define 2\,\nu+1, \quad\rho\define\frac{\mathrm{diam}\, Y}{3\widetilde    \nu} \, .
    \end{equation*}
   Note that now (4) reads as $r_0<\rho$ and that 
   \begin{equation}
   \label{e: lemma 6.1 eta over rho}
       \frac{\eta}{\rho}=3\widetilde \nu M\, .
   \end{equation}
   We define 
    \begin{equation}
    \label{e: lemma 6.1 definition Y0}
        Y_0\define \set{y\in Y\, \colon\, \frac{\norm{V} (B_\eta(y))}{ \eta^n}> M \frac{\norm{V}(B_{r_0}(y))}{ r_0^n}}\, ,
    \end{equation}
    which is a symmetric set by symmetry of $V$. By \eqref{e: lemma 6.1 definition Y0} and \eqref{e: lemma 6.1 monotonicity} we have 
    \begin{equation}
    \label{e: lemma 6.1 estimate mass Y0}
    \begin{split}
        \sum_{y\in Y_0} \frac{\norm V(B_{r_0}(y))}{\omega_n \, r_0^n}&< \frac{1}{M} \sum_{y\in Y_0} \frac{\norm V(B_\eta(y))}{\omega_n \, \eta^n}\leq \frac{2\,\nu+2}{M}\, \max_{y\in Y_0} \frac{\norm V(B_\eta(y))}{\omega_n \, \eta^n}\\
        &\leq \frac{(2\,\nu+2)\exp \xi}{M} \max_{y\in Y_0}\frac{\norm V(B_R(y))}{\omega_n \, R^n}\\
        &\leq \frac{(2\,\nu+2)\exp \xi}{M} \ \frac{\norm V (\set{x\, \colon \, \mathrm{dist}(x,Y)\leq R})}{\omega_m \, R^n}\, .
        \end{split}
    \end{equation}
    We now consider the other points. Let $\alpha\leq 2\,\nu+2$ be the cardinality of $Y\setminus Y_0$. Note that $\alpha$ is even by $(2)$ and the symmetry of $Y_0$. Let also $\beta=\sfrac\alpha2$. We write 
    \begin{equation*}
        Y\setminus Y_0= \set{y_1, \ldots, y_\beta, \, y_{\beta+1},\ldots, y_\alpha}\,
    \end{equation*}
    in such a way that $\d(y_i)\leq \d(y_{i+1})$ for each $i$. Moreover, by symmetry, we have 
    \begin{equation*}
        y_i=\sigma(y_{\alpha+1-i})\, ,
    \end{equation*}
    for $i=1,\ldots, \beta$. Geometrically, we are then labeling $Y\setminus Y_0$ from bottom to top. We claim the existence of an index $\l$ such that $\abs{y_{\l+1}-y_\l}\ge 3\, \rho$. Indeed, if this were false, we would have 
    \begin{equation*}
        \mathrm{diam}\, Y\leq \abs{y_\alpha -y_1}=\abs{y_\alpha-y_{\alpha-1} } + \ldots + \abs{y_2-y_1}< 3(\alpha-1) \, \rho =\frac{\alpha-1}{2\, \nu+1} \, \mathrm{diam}\, Y\, ,
    \end{equation*}
    yielding a contradiction as $\alpha\leq 2\,\nu+2 $. Let then
    \begin{equation}
    \label{e: lemma 6.1 l bar}
        \bar \l= \mathrm{argmin} \, \set{\l \, \colon \, \abs{y_{\l+1}-y_\l} \ge 3 \rho}\, ,
    \end{equation}
    and again, by symmetry, $\bar \l\leq \beta$. Geometrically, each such $\l$ represents a \virgolette{hole} and $\bar \l$ is the first \virgolette{hole} encountered moving along $Y\setminus Y_0$.\\

    We must distinguish between two cases, based on whether $\bar \l$ is $\beta$ or not. See \autoref{fig:lemma 6.1} for a geometric idea.
\begin{figure}[htbp]
  \centering
  \begin{minipage}{\textwidth}
    \centering
    \includegraphics[width=0.3\textwidth]{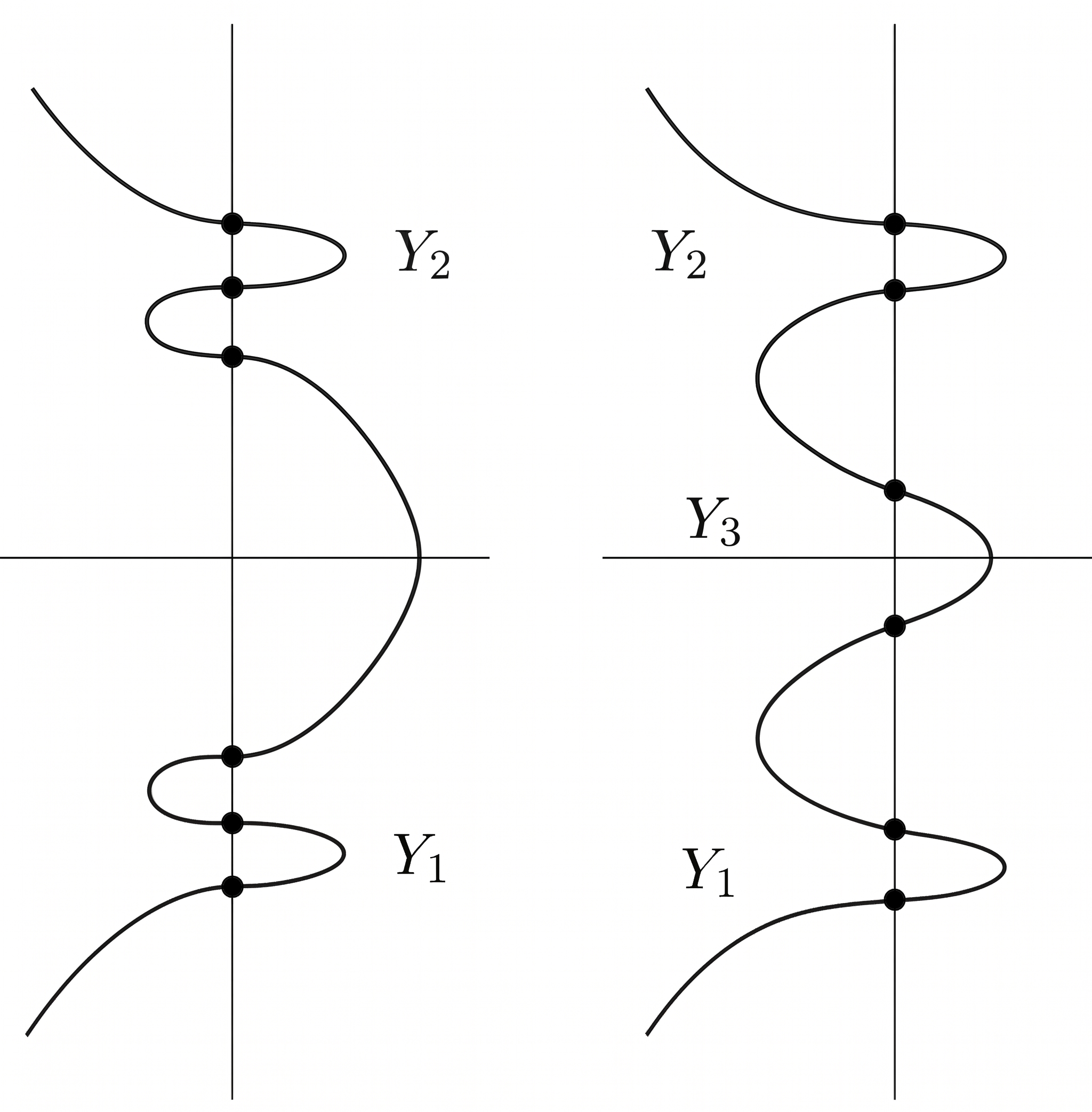}
    \caption{This is the picture to have in mind to distinguish from $\bar \l=\beta$ (left-hand-side) to $\bar \l<\beta$ (right-hand-side). Here we assumed $Y_0=\emptyset$ for simplicity.}
     \label{fig:lemma 6.1}
  \end{minipage}
\end{figure}
        
\vspace{0.5cm}   
\noindent\textbf{First case: suppose $\bar \l<\beta$.} We define 
    \begin{align*}
        &Y_1\define\set{y_\l \ \mathrm{with} \ \l\leq \bar \l}\\
        &Y_2\define\set{y_\l \ \mathrm{with} \ \l\ge \overline \alpha+1-\bar \l} =\sigma(Y_1)\\
        &Y_3\define \set{y_\l, \, \sigma(y_\l) \ \mathrm{with} \ \bar \l<\l \leq \beta}\, .
    \end{align*}
It is clear from the definition that $Y_1$ and $Y_2$ have at most $\beta-1$ elements each, so that $Y_1\cup Y_2$ has at most $\alpha-2$ elements, showing \eqref{4-21-8} and \eqref{4-21-8-2} in this case. Similarly, $Y_3$ has also at most $\alpha-2$ elements. For each $\tau>0$, let 
\begin{equation}
\label{e: lemma 6.1 definition V j tau}
\begin{gathered}
    V_{1,\tau}\define V \mres \left(\set{d<-\tau}\times \bG(n+1,n)\right), \quad V_{2,\tau}\define V \mres \left(\set{d>\tau}\times \bG(n+1,n)\right) ,\\ V_{3,\tau}\define V \mres \left(\set{-\tau< d<\tau}\times \bG(n+1,n)\right)\, .
\end{gathered}
\end{equation}
Using \cite[Theorem $4.10(2)$, page 442]{AllardFirstVariation} together with (5)-(6) (see \cite[page 455]{AllardFirstVariation} for more details), we have that for every $k=1,2,3$ and $y\in Y_k$,
\begin{equation*}
    \int_{\d(y_{\bar \l} )+ \rho}^{\d(y_{\bar \l}) + 2\rho} \norm{\delta V_{k,\tau}} (B_\eta(y))\, d\tau \leq 2\xi \, \norm{ V} (B_\eta (y))\, .
\end{equation*}
Thus,
\begin{equation*}
     \int_{\d(y_{\bar \l} )+ \rho}^{\d(y_{\bar \l}) + 2\rho} \sum_{k=1}^3 \sum_{y\in Y_k} \frac{\norm{\delta V_{k,\tau}} (B_\eta(y))}{\norm{ V} (B_\eta (y))}\, d\tau \leq 2\xi(2\,\nu+2)
\end{equation*}
and then there is $t\in \left(\d(y_{\bar \l} )+ \rho,\d(y_{\bar \l}) + 2\rho\right))$ such that 
\begin{equation*}
    \sum_{k=1}^3 \sum_{y\in Y_k} \frac{\norm{\delta V_{k,t}} (B_\eta(y))}{\norm{ V} (B_\eta (y))} \leq \frac{2\xi(2\,\nu+2)}{\rho}
\end{equation*}
for any $y\in \cup_{k=1}^3Y_k$. In particular, for any $k=1,2,3$ and $y\in Y_k$, we have 
\begin{equation}
\label{e: lemma 6.1 uniform variation estimate}
    \norm{\delta V_{k,t}}(B_\eta(y))\leq \frac{2\xi(2\,\nu+2)}{\rho} \norm{V}(B_\eta(y))\, .
\end{equation}
We define $V_k\define V_{k,t}$ for $ k=1,2,3$. By \eqref{e: lemma 6.1 definition V j tau} it is clear that \eqref{4-21-7} holds true. By the choice of $t$, it follows that for any $y\in Y_1$ (similar arguments can be done for $k=2,3$),
\begin{equation*}
    V_1\mres \bG_n(B_\rho(y))=V \mres \bG_n(B_\rho(y))\, .
\end{equation*}
This implies that (5)-(6) hold true for $V_1$ in place of $V$ for $r\in(r_0,\rho)$, that is 
\begin{gather}
   \label{e: lemma 6.1 proof variation small radius} \eta \norm{\delta V_1}(B_r(y))\leq R \norm{\delta V_1}(B_r(y))\leq \xi \norm{V_1}(B_r(y))\\
   \label{e: lemma 6.1 proof tilt small radius} \int_{\bG_n(B_r(y))} \norm{S-T}dV_1(x,S)\leq \xi \norm{V_1}(B_r(y))\, .
\end{gather}
Since $y\not \in Y_0$ and $r_0<\rho$, \eqref{e: lemma 6.1 monotonicity} implies 
\begin{equation}
\label{e: lemma 6.1 points not in Y0}
    \norm{V} (B_{\eta}(y))\leq M \norm{V}(B_{r_0}(y)) \left(\frac{\eta}{r_0}\right)^n\leq M (\exp \xi)\norm{ V} (B_\rho(y)) \left(\frac{\eta}{\rho}\right)^n\, .
\end{equation}
For every $r\in[\sigma,\eta)$, by \eqref{e: lemma 6.1 uniform variation estimate}, \eqref{e: lemma 6.1 points not in Y0}, \eqref{e: lemma 6.1 eta over rho} we have 
\begin{equation}
\label{e: lemma 6.1 proof variation large radius}
    \begin{split}
        \eta \norm{\delta V_1}(B_r(y)) &\leq \eta \norm{\delta V_1} (B_\eta(y)) \leq \frac{2\xi \eta (2\,\nu+2)}{\rho} \norm{ V}(B_\eta(y))\\
        &\leq\frac{2\xi \eta^{n+1}(2\,\nu+2)}{\rho^{n+1}} M (\exp \xi) \norm V(B_\rho(y)) \\
        &=2(2\,\nu+2)M (3\widetilde \nu M)^{n+1}(\exp \xi) \xi \norm{V_1} (B_\rho(y))\, .
    \end{split}
\end{equation}
In addition, for $r\in[\sigma,\eta)$, by (6) and using again \eqref{e: lemma 6.1 points not in Y0} and \eqref{e: lemma 6.1 eta over rho}, we also have 
\begin{equation}
\label{e: lemma 6.1 proof tilt large radius}
    \begin{split}
     \int_{\bG_n(B_r(y))} \norm{S-T}dV_1(x,S) &\leq \xi\norm{V}(B_\eta(y))  \\
     &\leq  M (\exp \xi) \xi \left(\frac{\eta}{\rho}\right)^n \norm{V}(B_\rho(y))\\
     &\leq M (\exp \xi) \xi (3(2\,\nu+2)M)^n\norm{V_1}(B_r(y))\, .
    \end{split}
\end{equation}
\eqref{e: lemma 6.1 proof variation small radius}-\eqref{e: lemma 6.1 proof variation large radius} and \eqref{e: lemma 6.1 proof tilt small radius}-\eqref{e: lemma 6.1 proof tilt large radius} show, respectively, \eqref{4-21-9} and \eqref{4-21-10}.

Let $y\in Y$. Note that
\begin{equation*}
    \bigcup_{k=1}^3\set{x\, \colon \mathrm{dist} (x, Y_k)\leq \eta}\subseteq B_{(M+1)\mathrm{diam} Y}(y)\, ,
\end{equation*}
and that $(M+1)\, \mathrm{diam}\, Y=\frac{(M+1)\, \eta}{M}$. This, together with \eqref{4-21-7} and \eqref{e: lemma 6.1 monotonicity} shows that 
\begin{equation}
\label{e: lemma 6.1 estimate mass Y123}
\begin{split}
    \frac{\sum_{k=1,2,3} \norm{V_k} \set{x\, \colon \mathrm{dist} (x, Y_k)\leq \eta}}{\omega_n\, \eta^n}&\leq \frac{\norm{V} (B_{(M+1)\mathrm{diam} Y}(y))}{\omega_n \, \eta^n}\\
    &\leq \left(1+\frac1M\right)^n\exp \xi \, \frac{ \norm{ V} (\set{x\, \colon \, \mathrm{dist}(x,Y)\leq R)}}{\omega_n \, R^n}\, .
    \end{split}
\end{equation}
Finally, \eqref{4-21-11} follows by \eqref{e: lemma 6.1 estimate mass Y0} and \eqref{e: lemma 6.1 estimate mass Y123}.

\vspace{0.5cm}
\noindent\textbf{Second case: suppose $\bar \l=\beta$.} We define 
\begin{equation*}
    Y_1\define (Y\setminus Y_0)_-=(Y\setminus Y_0) \cap \R^{n+1}_-, \quad Y_2\define (Y\setminus Y_0)_+=(Y\setminus Y_0)\cap \R^{n+1}_+, \quad Y_3\define \emptyset.
\end{equation*}
We then define $V_1$ and $V_2$ similarly to the previous case, whereas we put $V_3= 0$. The thesis follows by the same computations.
\end{proof}
\end{lemma}
\begin{lemma}
\label{lemma 8.4}
Let $n, \, \nu \in\Na$, $\widetilde \nu=2\,\nu+1$ and $\lambda\in (1,2)$. Corresponding to $n,\nu$ and $\lambda$, there exist $\gamma\in (0,1)$ 
and ${\widetilde M}\in (1,\infty)$ 
with the following property. Suppose 
\begin{enumerate}
\item
$0<r_0<R<\infty,\, T\in {\bf G}(n+1,n),\, V\in {\bf V}_n(\R^{n+1})$ is a symmetric varifold;
\item
$Y\subset H_0^{\perp}$ is symmetric, it has  no more than $2\,\nu+2$ elements and $Y\cap H_0=\emptyset$;
\item
$\left((1+3(2\,\nu+1))^2+{\widetilde M}^2\right)^{\frac12} r_0 <R$;
\item
${\rm diam}\, Y\leq \gamma R$;
\item
$R\|\delta V\|(B_r(y))\leq \gamma \|V\|(B_r(y))$ for all $y\in Y$ and $r\in (r_0,R)$;
\item
$\int_{{\bf G}_n(B_r(y))}\|S-T\|\, dV(x,S)\leq \gamma \|V\|(B_r(y))$ for all $y\in Y$ and
$r\in (r_0,R)$.
\end{enumerate}
Then there exists a partition of $Y$ into symmetric subsets $Y_0,Y_1,\ldots, Y_J$ such that
\begin{itemize}
\item[(i)] for all $j\in \set{1,\ldots, J}$,  either $\mathrm{diam} \, Y_j\leq 3 \widetilde\nu\, r_0$, or $\mathrm{diam} \, Y_j^+=\mathrm{diam} \, Y_j^-\leq 3 \widetilde\nu\, r_0$;
\item[(ii)] we have
\begin{equation}
\begin{split}
\lambda &\frac{\norm{V}(\set{x : {\rm dist}\,(x,Y)\leq R})}{\omega_n R^n}\geq
\sum_{y\in Y_0} \frac{\|V\|(B_{r_0}(y))}{\omega_n r_0^n} \\
&+\sum_{j=1}^J \frac{\norm{V}\left(\set{x: {\rm dist}\,(T^{\perp}(x),Y_j)\leq 
r_0,\, \abs{T(x)}\leq {\widetilde M}r_0}\right)}{\omega_n ({\widetilde M}r_0)^n}\, .
\end{split}
\label{4-21-18}
\end{equation}

\label{thm4.21.1}
\end{itemize}
\begin{proof}
We adopt the same notation of the proof of \autoref{lemma4.21}. Before going into the details, let us give a brief idea of the proof. We distinguish in two cases, based on $Y$ having a small diameter or not. In the first case, the thesis easily follows. In the latter, we use \autoref{lemma4.21} to obtain $Y_1,Y_2, Y_3$ (with $Y_3$ possibly empty). We then check again the diameter of the elements of the partition; if some of their diameter is small, we can estimate the mass of $V$ around it. Otherwise, we need to partition them again. If $\mathrm{diam}\, Y_1=\mathrm{diam}\, Y_2$ happens to be large, we partition both $Y_1$ and $Y_2$ by \autoref{lemma4.21 KT}, obtaining symmetric results as observed in \autoref{remark: lemma 8.3 KT for symmetric sets}; the reason why we apply this partition method is because $Y_1$ is built to be sufficiently far from $H_0$. If instead $\mathrm{diam}\, Y_3$ is large, then we partition it using again \autoref{lemma4.21}. Note that the two described scenarios can happen simultaneously. We repeat this process on each newly generated partition at most $\nu+1$ times, until we get subsets of small diameter (to be more precise, until we get sets whose restrictions to the upper and lower halfspace have small diameter).
\vspace{0.7cm}

\noindent\textbf{$Y$ has Small diameter}: suppose that $\mathrm{diam}\, Y\leq 3 \widetilde\nu\,  r_0$. Then, set $J=1$, $Y_1=Y$ and $Y_0=\emptyset$. This implies that for any $y\in Y$
\begin{equation}
\label{e: proof Lemma 8.4 0}
    \set{x\, \colon \, \mathrm{dist}(T^\perp(x), Y_1)\leq r_0, \ \abs{T(x)}\leq \widetilde M r_0}\subseteq B_{r_0((1+3\tilde \nu)^2+\tilde M^2)^\frac12}(y)\, ;
\end{equation}
indeed, if $x$ belongs to the set at the left-hand-side of \eqref{e: proof Lemma 8.4 0}, there exists $\widetilde y\in Y$ such that $\abs{T^\perp(x)- \widetilde y}\leq r_0$ and $\abs{y-\widetilde y}\leq \mathrm{diam}\, Y$, yielding 
\begin{equation*}
    \abs{T^\perp(x)-y}=\abs{T^\perp(x-y)}\leq r_0+3 \widetilde\nu\, r_0\, .
\end{equation*}
Hence,
\begin{equation*}
    \abs{x-y}=\abs{T(x-y)+T^\perp(x-y)}\leq \sqrt{(r_0+3 \widetilde\nu \,r_0)^2+(\widetilde M r_0)^2}\, .
\end{equation*}
We write 
\begin{equation}
\label{e: proof Lemma 8.4 1}
    \frac{\norm{V}(B_{r_0((1+3\widetilde \nu)^2+\widetilde M})^\frac12(y)}{\omega_n (r_0^n\widetilde M)^n}=   \frac{\norm{V}(B_{r_0((1+3\widetilde \nu)^2+\widetilde M})^\frac12(y)}{\omega_n \, r_0^n\,  r_0((1+3\widetilde \nu)^2+\widetilde M)^\frac{n}{2}}\, \left( 1+ \frac{(1+3\widetilde \nu)^2}{\widetilde M^2}\right)^\frac{n}{2}\, .
\end{equation}
Assumptions (3) and (5) allow us to use the monotonicity formula \autoref{lem: monotonicity formula}, and with \eqref{e: proof Lemma 8.4 1} we have 
\begin{equation}
\label{e: proof Lemma 8.4 2}
\frac{\|V\|(B_{ r_0((1+3\widetilde \nu)^2+{\widetilde M}^2)^{\frac12}}(y))}{\omega_n (r_0{\widetilde M})^n}
\leq (\exp \gamma)\left(1+\frac{(1+3\nu)^2}{{\widetilde M}^2}\right)^\frac{n}{2} \frac{\|V\|(B_R(y))}{\omega_n R^n}\, .
\end{equation}
\eqref{e: proof Lemma 8.4 0} and \eqref{e: proof Lemma 8.4 2} imply that 
\begin{equation}
\label{e: lemma 8.4 estimate small diamater}
\begin{split}
    &\frac{\norm{V} \left( \set{x\, \colon \, \mathrm{dist}(T^\perp(x), Y_1)\leq r_0, \ \abs{T(x)}\leq \widetilde M r_0}\right)}{\omega_n (\widetilde M \, r_0)^n}\\
    &\qquad \leq (\exp \gamma)\left(1+\frac{(1+3\widetilde \nu)^2}{{\widetilde M}^2}\right)^\frac{n}{2} \frac{\norm{V}(\set{x\, \colon \, \mathrm{dist}(x,Y)\leq R })}{\omega_n R^n}\, 
    \end{split}
\end{equation}
which is the thesis by taking sufficiently small $\gamma$ and large $\widetilde M$ depending only on $n,\nu$ and $\lambda$.
\vspace{0.7cm}

\noindent\textbf{$Y$ has large diameter}: suppose $\mathrm{diam}\, Y> 3\,  \widetilde \nu \,r_0$. We can then apply \autoref{lemma4.21} with $\xi=\gamma$ and $M$ sufficiently large in a way that $\gamma \ge \frac{1}{M+1}$ (so that (4) of the present Lemma implies (3) of \autoref{lemma4.21}) to obtain $Y_0,Y_1,Y_2,Y_3$ and $V_1,V_2,V_3$ satisfying \eqref{e: lemma 6.1 geometry Yk}-\eqref{4-21-12}. As in the proof of \autoref{lemma4.21}, we denote $Y\setminus Y_0=\set{y_1,\ldots,y_\beta,y_{\beta+1},\ldots,y_\alpha}$ and we define $\bar \l$ as there, namely as the first index $\l$ such that there is a hole above $y_\l$ (see \eqref{e: lemma 6.1 l bar} for details). We start by focusing on the following case.
\vspace{0.5cm}

\noindent\textbf{$Y$ has large diameter and $\bar \l=\beta$}: in this case, $Y_3=\emptyset$ and $V_3=0$. We separate into multiple cases, depending on $Y_1$ and $Y_2$ having small diameter or not.
\begin{itemize}
    \item If $\mathrm{diam}\, Y_1=\mathrm{diam}\, Y_2\leq 3 \widetilde\nu \, r_0$, then $J=1$ and the desired partition is $Y_0, \, Y_1\cup Y_2$. We now show why. We repeat the argument proposed when $Y$ had small diameter; more precisely, $V_1$ by \eqref{4-21-9} satisfies, for $r\in (r_0,M\mathrm{diam}\, Y)$ and $y\in Y_1,$
    \begin{equation*}
        (M\,{\rm diam}\, Y)\|\delta V_1\|(B_r(y))\leq 2M(2\,\nu+2)(3(2\,\nu+1) M)^{n+1}(\exp \gamma)\gamma \|V_1\|(B_r(y))\, .
    \end{equation*}
     We want to apply the monotonicity formula to get a counterpart to \eqref{e: lemma 8.4 estimate small diamater}, with a factor different to $\exp \gamma$. To do so, we need to have
    \begin{equation*}
        r_0 \left((1+3\widetilde \nu)^2+\widetilde M^2\right)^\frac12 < M\mathrm{diam}\, Y\, .
    \end{equation*}
    Note that we can set $M$ greater than $\widetilde M$ by a factor depending only on $\nu$ so that
    \begin{equation*}
        r_0 \left((1+3\widetilde \nu)^2+\widetilde M^2\right)^\frac12\leq r_0(1+3\widetilde \nu + \widetilde M)<3M\, \widetilde \nu \, r_0< M\, \mathrm{diam}\, Y\, 
    \end{equation*}
    where we used the hypothesis of $Y$ having large diameter for the last inequality. We can then apply the monotonicity formula \eqref{e: lemma 6.1 monotonicity} to get
    \begin{equation}
    \label{e: proof lemma 8.4 3}
        \begin{split}
    &\frac{\norm{V_1} \left( \set{x\, \colon \, \mathrm{dist}(T^\perp(x), Y_1)\leq r_0, \ \abs{T(x)}\leq \widetilde M r_0}\right)}{\omega_n (\widetilde M \, r_0)^n}\\
    &\leq \left(\exp \left( 2M(2\,\nu+2)(3(2\,\nu+1) M)^{n+1}(\exp \gamma)\gamma\right) \right)\left(1+\frac{(1+3\widetilde\nu)^2}{{\widetilde M}^2}\right)^\frac{n}{2}\cdot  \\
    &\qquad \cdot \frac{\norm{V_1}(\set{x\, \colon \, \mathrm{dist}(x,Y_1)\leq M\, \mathrm{diam}\, Y })}{\omega_n (M\, \mathrm{diam}\, Y)^n}\\
    &=\vcentcolon c(\gamma, \widetilde M)\frac{\norm{V_1}(\set{x\, \colon \, \mathrm{dist}(x,Y_1)\leq M\, \mathrm{diam}\, Y })}{\omega_n \, (M\, \mathrm{diam}\, Y)^n}\, ,
    \end{split}
    \end{equation}
    together with an analogous estimate for $V_2$ and $Y_2$. Moreover, by \eqref{4-21-11}, we have 
    \begin{equation}
    \label{e: lemma 8.4 V leq V1}
    \begin{split}
        &\norm{V} \left( \set{x\, \colon \, \mathrm{dist}(T^\perp(x), Y_1)\leq r_0, \ \abs{T(x)}\leq \widetilde M r_0}\right)\\
        & \qquad\leq \norm{V_1} \left( \set{x\, \colon \, \mathrm{dist}(T^\perp(x), Y_1)\leq r_0, \ \abs{T(x)}\leq \widetilde M r_0}\right)
        \end{split}
    \end{equation}
    Hence, by \eqref{e: lemma 6.1 Y1 and Y2 are far},\eqref{e: lemma 8.4 V leq V1}, \eqref{e: proof lemma 8.4 3}, the fact that $c(\gamma, \widetilde M)>1$ and \eqref{4-21-12} we get 
    \begin{equation*}
        \begin{split}
& \sum_{y\in Y_0} \frac{\norm V (B_{r_0}(y))}{\omega_n \, r_0^n} + \frac{\norm{V}\left(\set{x: {\rm dist}\,\left(T^{\perp}(x),(Y_1\cup Y_2)\right)\leq 
r_0,\, \abs{T(x)}\leq {\widetilde M}r_0}\right)}{\omega_n ({\widetilde M}r_0)^n}  \\  
& \leq\sum_{y\in Y_0} \frac{\norm V (B_{r_0}(y))}{\omega_n \, r_0^n} +\sum_{k=1,2} \frac{\norm V\left(\set{x: {\rm dist}\,(T^{\perp}(x),Y_k)\leq 
r_0,\, \abs{T(x)}\leq {\widetilde M}r_0}\right)}{\omega_n ({\widetilde M}r_0)^n}\\
&\leq \sum_{y\in Y_0} \frac{\norm{V} (B_{r_0}(y))}{\omega_n \, r_0^n} +\sum_{k=1,2} \frac{\norm{V_k}\left(\set{x: {\rm dist}\,(T^{\perp}(x),Y_k)\leq 
r_0,\, \abs{T(x)}\leq {\widetilde M}r_0}\right)}{\omega_n ({\widetilde M}r_0)^n}\\
&\leq c(\gamma, \widetilde M) \left( \sum_{y\in Y_0} \frac{\norm{V} (B_{r_0}(y))}{\omega_n \, r_0^n} + \sum_{k=1,2}\frac{\norm{V_k}\left(\set{x\, \colon \, \mathrm{dist}(x,Y_k)\leq M\, \mathrm{diam}\, Y }\right)}{\omega_n \, (M\, \mathrm{diam}\, Y)^n}\right)\\
&\leq c(\gamma, \widetilde M)\set{\left(1+\frac1M\right)^n+\frac{2(\nu+1)}
{M}}  (\exp\gamma)\frac{\norm V(\set{x:{\rm dist}\, (x,Y)\leq R})}{\omega_n R^n} \, ,
        \end{split}
    \end{equation*} 
which is \eqref{4-21-18} upon taking sufficiently large $\widetilde M$ and $M$ and small $\gamma$ (recall that we assumed that $\gamma\ge \frac{1}{M+1})$ depending only on $n$ and $\nu$.
    \item Suppose instead that $\mathrm{diam}\, Y_1=\mathrm{diam}\, Y_2 > 3 \widetilde\nu \, r_0$. In principle, we could iterate \autoref{lemma4.21} on $Y_1\cup Y_2$, but this would not be helpful, as it would return as result the same partition, since $\mathrm{diam} \, (Y_1\cup Y_2)= \mathrm{diam}\, (Y\setminus Y_0)$. More precisely, we do not want to apply \autoref{lemma4.21} with $\nu$ being the cardinality of the given step partition, but with the $\nu$ fixed at the start. This also explains why in the statement of \autoref{lemma4.21} we require that $Y$ must have \emph{at most} some elements, instead of a precise quantity. 
    
    We wish to apply \autoref{lemma4.21 KT} instead to the pair $Y_1, \, V_1$ (and similarly to $Y_2, \, V_2$). $Y_1$ has at most $\nu+1$ elements because it is a subset of $Y$, so (2) of \autoref{lemma4.21 KT} holds. (5) and (6) hold true due to \eqref{4-21-9} and \eqref{4-21-10}, with $R=M\, \mathrm{diam}\, Y$ and  $\xi= 2M(2\,\nu+2) (3(2\,\nu+1)M)^{n+1} (\exp \gamma) \gamma$. (3) holds true using $M-1$ as new $M$\footnote{Although this may sound a bit confusing, we are just saying that new constant $M$ is the same constant $M$ of the previous iteration minus $1$}, since we have $((M-1)+1)\mathrm{diam}\, Y\leq M\, \mathrm{diam Y}$. Finally, (4) holds true since $\mathrm{diam}\, Y>3 \widetilde\nu \, r_0 > 3(\nu +1)\,  r_0$. Thus, there exist a partition of $Y_1$ into $Z_0,Z_1,Z_2$ and varifolds $W_1,W_2$ satisfying \eqref{4-21-7 KT}-\eqref{4-21-12 KT}. Moreover, by \autoref{remark: lemma 8.3 KT for symmetric sets}, $Y_2$ admits a partition $\sigma(Z_0),\, \sigma(Z_1) , \, \sigma(Z_2)$ together with varifolds $\sigma_\sharp W_1, \, \sigma_\sharp W_2$ satisfying similar properties. Note that $V_k+\sigma_\sharp V_k$ satisfies a counterpart to \eqref{e: lemma 6.1 Y1 and Y2 are far}, namely, taking $V_1$ for simplicity,
    \begin{equation*}
       \begin{split}
        &\norm{V_1+\sigma_\sharp V_1} \left( \set{x\in \R^{n+1} \, \colon \, \mathrm{dist}\left(T^\perp(x), (Z_k\cup \sigma(Z_k))\right)\leq r_0)} \right)\\
        & \qquad= \norm{V_1} \left(\set{x:{\rm dist}\, (x,Z_k)\leq r_0}\right)+ \norm{\sigma_\sharp V_1} \left(\set{x:{\rm dist}\, (x,Z_k)\leq r_0}\right)\, \\
        &\qquad =\norm{W_k} \left(\set{x:{\rm dist}\, (x,Z_k)\leq r_0}\right)+ \norm{\sigma_\sharp W_k} \left(\set{x:{\rm dist}\, (x,Z_k)\leq r_0}\right)\, ,
        \end{split}
    \end{equation*}
    for $k=1,2$. Indeed, in the notation of \autoref{lemma4.21}, $Y_1$ is built to be further than $3\,\rho>3\,r_0$ from $H_0$, yielding that $r_0$ neighborhoods of $Y_1$ cannot cross $H_0$. Similarly, whenever we iterate the procedure, we cut the given set corresponding to a point above which there is a hole of size $3r_0$.

    Now we can just iterate the procedure to each $Z_k$ for $k=1,2$: if they have a small diameter, we get counterparts to \eqref{e: proof lemma 8.4 3}-\eqref{e: lemma 8.4 V leq V1}, otherwise we apply again \autoref{lemma4.21}. Each time we iterate, we need to take larger $\widetilde M$ and smaller $\gamma$, but this iteration lasts for at most $\nu+1$ times, as each element of the new partition has (strictly) fewer elements than its predecessor.
\end{itemize}
\vspace{0.7cm}

\noindent\textbf{Y has large diameter and $\bar \l<\beta$}. Suppose that $\mathrm{diam}\, Y_k\leq 3\,  \widetilde \nu\, r_0$ for some $k\in \set{1,2,3}$ (note that $\mathrm{diam}\, Y_1=\mathrm{diam}\, Y_2$). Then, we can repeat the argument of \eqref{e: proof lemma 8.4 3}-\eqref{e: lemma 8.4 V leq V1}. If $\mathrm{diam}\, Y_1>3 \widetilde\nu \,r_0$, we  apply \autoref{lemma4.21 KT} to $Y_1, V_1$ and to $Y_2, V_2$ and we argue similarly to what we described before. If instead  $\mathrm{diam}\, Y_3>3 \widetilde\nu\, r_0,$ we use again \autoref{lemma4.21}; say that $Y_3=\set{y_1,y_\mu,y_{\mu+1},\ldots, y_{2\mu}}$ for some $\mu\in \Na$, $\mu\leq \nu$. If the $\bar \l$ associated to $Y_3$ is $\mu$, then we fall in the previous case and we repeat all the computations we described. Otherwise, we repeat the present step until we can. We will stop after at most $\nu+1$ iterations, since this procedure cuts some elements.

\end{proof}
\end{lemma}
The next Lemma is the symmetric version of \cite[Lemma $8.5$]{KimTonegawa}.
\begin{lemma}
\label{lemma 8.5}
    Corresponding to $n, \nu \in \Na$ and $\lambda\in (1,2)$, there exist $\gamma, \eta\in (0,1)$, $\widetilde M\in (1,\infty)$ and $j_0\in \Na$ with the following property. Suppose 
    \begin{enumerate}
\item $\E^S\in \opns$ is a symmetric open partition,
$j\in \Na$ with $j\geq j_0$;
\item $\eps\leq \gamma j^{-4}$;
\item $\eta j^{-2}<R$;
\item $Y\subset H_0^{\perp}$ is symmetric, it has no more than $2\, \nu$ elements, $Y\cap H_0=\emptyset$ and $\theta^{n}
\left(\normpiccolo{\partial\E^S},y\right)=1$ for each $y\in Y$;
\item ${\rm diam}\, Y\leq \gamma R$; 
\item $\displaystyle R\normpiccolo{\delta(\Phi_{\eps}\ast \partial\E^S)}(B_r(y))\leq \gamma \normpiccolo{\Phi_{\eps}\ast \partial \mathcal E^S}(B_r(y))
$ for all $y\in Y$ and $r\in (\eta^2 j^{-2},R)$, 
\item $ \displaystyle \int_{{\bf G}_{n}(B_r(y))} \norm{S-T}\, d(\Phi_{\eps}\ast \partial \mathcal E^S)(x,S)\leq \gamma \normpiccolo{\Phi_{\eps}
\ast \partial \mathcal E^S}(B_r(y))$ for all $y\in Y$ and $r\in (\eta^2 j^{-2},R)$; 
\end{enumerate}
Moreover, we introduce
\begin{enumerate}
\item[(a)] $\widetilde R_1\define\eta^2 j^{-2}\lambda^{-\frac{1}{4n}}$;
\item[(b)] $\widetilde R_2\define\tilde M \eta^2 j^{-2}\lambda^{-\frac{1}{4n}}$;
\item[(c)] $\rho\define\frac12\eta^2j^{-2}(1-\lambda^{-\frac{1}{4n}})$;
\end{enumerate}
and, for any symmetric subset $Y'\subset Y$,
\begin{enumerate}
\item[(d)] $E^\ast_1(r,Y')\define\set{x\in\R^{n+1} \, \colon\,  \abs{T(x)}\leq r, {\rm dist}(Y',T^{\perp}(x))
\leq (1+\widetilde R_1^{-1}r)\rho}$,
\item [(e)]$E^\ast_2(r,Y')\define\set{x\in\R^{n+1} \, \colon \,  \abs{T(x)}\leq r, {\rm dist}(Y',T^{\perp}(x))
\leq (1+\widetilde R_2^{-1}r)\rho}$.
\end{enumerate}
Further suppose that for all symmetric $Y'\subset Y$ with either ${\rm diam}\,Y'<j^{-2}$ or $\mathrm{diam}\, Y'_+=\mathrm{diam}\, Y'_-<j^{-2}$, $i=1,2$
and $r\in (0,j^{-2})$ that
\begin{enumerate}
\item [(8)] $\displaystyle \int_{{\bf G}_{n}(E^\ast_i(r,Y'))}
\norm{S-T}\, d(\partial \E^S)(x,S)\leq \gamma \|\partial\mathcal E^S\|(E^\ast_i(r,Y'))\, ,$
\item [(9)] $\Delta_j \normpiccolo{\partial \E^S}(E^\ast_i(r,Y'))
\geq -\gamma \normpiccolo{\partial\E^S}(E^\ast_i(r,Y')) \, .$
\end{enumerate}
Then we have
\begin{equation}
\lambda \normpiccolo{\Phi_{\eps}\ast \partial \E^S}\left(\set{x \, \colon \,  {\rm dist}\,(x,Y)\leq R}\right)\geq 
\omega_{n} R^{n}\Ha^0(Y)\, .
\label{afitg7}
\end{equation}
\begin{proof}
    The proof is analogous to \cite[Lemma $8.5$]{KimTonegawa}; indeed, after fixing some parameters, we apply \autoref{lemma 8.4} to the varifold $\Phi_\eps \ast \partial \E^S$, whose symmetry is guaranteed by the linearity of the convolution and by the symmetry of $\E^S$. We then get a partition $Y_0,\ldots, Y_J$ and we apply \autoref{lemma 8.1} to $\partial \E^S$ with each element of $Y_0$ and to $\partial \E^S$ with every other set $Y_k$, which are sets whose restriction to both halfspaces have small diameter and then eligible for the Lemma.
\end{proof}
\end{lemma}
The next Lemma is used to treat sets $Y$ which are not symmetric, but are very far from $H_0$.
\begin{lemma}
\label{lemma 8.5 Y far from H0}
Corresponding to $n,\nu\in \Na$ and $\lambda\in (1,2),$ there exist $\widehat \gamma, \widehat \eta\in (0,1)$ and $\widehat M\in (1,\infty)$ with the following property. Suppose (1)-(3), (5)-(7) of \autoref{lemma 8.5} and 
\begin{enumerate}
\item[(4')]: $Y\subset H_0^\perp$ has no more than $\nu$ elements, $\theta^n\left(\normpiccolo{\partial \E^S}, y\right)=1$ for each $y\in Y$ and $\mathrm{dist}(Y,H_0) > j^{-2}$.
\end{enumerate}
Furthermore, assume also (8)a and (9), but assume no symmetry on $Y'$ and that just $\mathrm{diam}\, Y'<j^{-2}$. Then we have
\begin{equation*}
    \lambda \normpiccolo{\Phi_\eps \ast \partial \E^S}( \set{x\, \colon \, \mathrm{dist}(x,Y)\leq R}) \ge \omega_n \, R^n \, \Ha^0(Y)\, .
\end{equation*}
\begin{proof}
    It is proved following \textit{verbatim} \cite[Lemma $8.5$]{KimTonegawa}. At the end, the authors use \cite[Lemma $8.1$]{KimTonegawa}, which is the non-symmetric version of our \autoref{lemma 8.1}, with $\rho=\frac12 \eta^2 j^{-2}(1-\lambda^{-\frac{1}{4n}})<\frac12 j^{-2}$. However, as we observed in \autoref{remark: lemma 8.1 Y far from H0}, if $Y$ is not symmetric but far enough from $H_0$, we can still prove the thesis of \autoref{lemma 8.1}.
\end{proof}
\end{lemma}
\begin{theorem}
\label{theorem: integrality}
    Suppose that $\{\mathcal E_j^S\}_{j=1}^{\infty}\subset 
\opn$ and $\{\eps_j\}_{j=1}^{\infty}\subset (0,1)$ satisfy
\begin{itemize}
\item[(1)] $\lim_{j\rightarrow\infty} j^4 \eps_j=0$,
\item[(2)] $\sup_j \normpiccolo{\partial\mathcal E_j^S}(\R^{n+1})<\infty$,
\item[(3)] $\displaystyle \liminf_{j\rightarrow\infty} \bigint{\R^{n+1}} \frac{\big| \Phi_{\eps_{j}} \ast \delta (\partial \E_{j}^S(t))\big|^2}{\Phi_{\eps_j} \ast \normpiccolo{\partial \E_{j}^S(t)} + \eps_{j}} \, dx <\infty$,
\item[(4)] $\lim_{j\rightarrow\infty} j^{2(n+1)}\Deltajvc\normpiccolo{\partial\E_j^S}(\R^{n+1})=0$.
\end{itemize}
Then there exists a converging subsequence $\{\partial\mathcal E^S_{j_\l}\}_{\l=1}^{\infty}$ whose limit satisfies $V^S\in {\bf IV}_{n}(\R^{n+1})$. Moreover, the density is even valued for $\normpiccolo{V^S}\mbox{-a.e.} \  x\in H_0$. 
\begin{proof}

As mentioned in the introduction to this section,for any given $x\in \R^{n+1}\setminus H_0$, we can just take $j$ sufficiently large such that $B_{\sfrac{1}{j^2}}(x)\cap H_0=\emptyset$ and we work there following the original proof.

Let us focus on the integrality on $H_0$. We basically follow again the proof \cite[Theorem $8.6$]{KimTonegawa}, making use of our Lemmas for symmetric sets rather than theirs. We briefly describe their proof, focusing on the required modifications.

We fix a sequence $\set{j_\l}$ along which the quantities in (2) and (3) are uniformly bounded by some constant $M$ and such that $\set{\partial \E^S_{j_\l}}_{\l=1}^\infty$ converges to a symmetric varifold $V^S\in \RV_n(\R^{n+1})$, whose existence in guaranteed by \autoref{theorem: rectifiability}. After some approximations, we fix $x\in H_0$ and we write
    \begin{equation*}
        d\define \theta^n\left(\normpiccolo{V^S}, x\right), \ T\define \mathrm{Tan}^n\left(\normpiccolo{V^S},x\right)\, .
    \end{equation*}
    Due to the locality of the tangent plane (see \cite[Proposition $10.5$]{MaggiSets}), we have that for $\Ha^n \mbox{-a.e.} \ x\in H_0$, $T=H_0$. After a translation, we can also assume $x=0$. Set $r_\l=\l^{-1}$ and choose another subsequence so that
    \begin{equation}
\label{xitg21}
\lim\limits_{\l\rightarrow\infty} (f_{(r_\l)})_{\sharp} \partial \E^S_{j_\l}
=\lim\limits_{\l\rightarrow\infty} (f_{(r_\l)})_{\sharp} (\Phi_{\eps_{j_l}}\ast\partial \E_{j_\l})
=d\abs{H_0}\, ,
\end{equation}
and 
    \begin{equation}
    \label{e: proof thm 8.6 ratio jl rl}
        \lim\limits_{\l \to \infty} \frac{{j_\l}^{-1}}{r_\l}=0\, ,
    \end{equation}
    where $(f_{(r)})(y)\define r^{-1} \, y$.
    
    Suppose $2\,\nu$ is the smallest positive even number greater than $d$, that is
    \begin{equation}
    \label{e: proof thm 8.6 0}
        \nu \in \Na \ \mbox{and} \ 2\,\nu\in (d, d+2]\, .
    \end{equation}
    Choose $\lambda\in (1,2)$ such that 
    \begin{equation}
    \label{e: proof thm 8.6 ratio d nu}
        \lambda^{n+1} d < 2\,\nu\, .
    \end{equation}
Corresponding to $n, \, \nu$ and $\lambda$, there are $\gamma, \eta\in (0,1)$ and $\widetilde M\in (1,\infty)$ such that \autoref{lemma 8.5} holds. For each $\l$, we define the set $G_\l^{\ast}$ as the points $x$ in $B_{(\lambda-1)r_\l}$ which satisfy both conditions (6) and (7) of \autoref{lemma 8.5}, taken with respect to $j=j_\l$ and $R=r_\l$, and such that $\partial \E^S_{j_\l}$ is unit density at $x$. We can also further redefine $G_\l^\ast$ as $G_\l^\ast \setminus H_0$, since for any $\l$ we have $\normpiccolo{\partial \E^S_{j_\l}}(H_0)=0$ by \eqref{e: no mass on the boundary along iteration}.
    
We next define, as in \autoref{lemma 8.5} (a)-(c), 
\begin{equation*}
\widetilde R_{1,l}\define\eta^2 j_\l^{-2}\lambda^{-\frac{1}{4n}}\, , \ 
\widetilde R_{2,l}\define\tilde M\eta^2 j_\l^{-2}\lambda^{-\frac{1}{4n}}\, , \ 
\rho_\l\define\frac12 \eta^2 j_\l^{-2}(1-\lambda^{-\frac{1}{4n}})\, .
\end{equation*}
For each $x\in G_\l^\ast$, take an arbitrary symmetric finite set $Y'=\set{y_1,\ldots, y_{2m}}\subset G_\l^\ast$, where $m\in \Na$, with $y_1=x$, $T(x-y_i)=0$ for $i\in\set{2,\ldots,2m}$ and $\mathrm{diam}\, (Y')_+=\mathrm{diam}\, (Y')_-<j_\l^{-2}$. Define
\begin{equation*}
E^\ast_{i,\l}(r,Y')\define\{ z\in \mathbb R^{n+1} : \abs{T(z-x)}\leq r,\, {\rm dist}\,(T^{\perp}(Y')
,T^{\perp}(z))\leq  (1+\widetilde R_{i,\l}^{-1}
r)\rho_\l\}
\end{equation*}
for $i=1,2$. We define $G_\l^{\ast \ast}$ as the set of points $x\in G_\l^\ast$ such that, for arbitrary such $Y'$ described above, conditions (8) and (9) of \autoref{lemma 8.5} hold true, namely for all $r\in (0,j_\l^{-2})$ and $i=1,2$, we have 
\begin{equation*}
\begin{split}
& \int_{{\bf G}_{n}(E^\ast_{i,\l}(r,Y'))}
\|S-T\|\, d(\partial\mathcal E_{j_\l})\leq \gamma\|\partial\mathcal E_{j_\l}\|(E^\ast_{i,\l}(r,Y'))
\mbox{ and } \\ & \Delta_{j_l}\|\partial\mathcal E_{j_\l}\|(E^\ast_{i,\l}(r,Y'))\geq -\gamma 
\|\partial\mathcal E_{j_\l}\|(E^\ast_{i,\l}(r,Y')).
\end{split}
\end{equation*}
The same computations used by \cite{KimTonegawa} yields to (see \cite[Equation $(8.154)$]{KimTonegawa})
\begin{equation}
\label{e: proof thm 8.6 8.154 KT}
    \lim\limits_{\l \to \infty}r_\l^{-n} \normpiccolo{\partial \E_{j_\l}} (B_{(\lambda-1)r_\l}\setminus G_\l^{\ast \ast})=0\, .
\end{equation}
Given $s\in (0,\frac14)$ and $x\in G_\l^{\ast \ast}$ such that $\abs{\d} (x)\leq \frac{\gamma \, r_\l \, s}{2}$, we want to use \autoref{lemma 8.5} with $R=r_\l\, s$ for $Y=\set{T^\perp(x), \, \sigma(T^\perp(x))}$. The case of $x\in G_\l^{\ast \ast}$ and far from $H_0$ will be treated afterwards. We can see that $Y$ is constructed in such a way to satisfy all the assumptions of \autoref{lemma 8.5} (in particular, note that we assumed the bound on $\abs \d(x)$ to make sure that $Y$ has the correct diameter to satisfy assumption (5)). Thus, by \eqref{afitg7} we have 
\begin{equation}
\label{e: proof thm 8.6 1}
    \lambda \normpiccolo{\Phi_\eps \ast \delta \E^S_{j_\l}}( B^S_{r_\l s}(x) )\ge 2\, \omega_n(r_\l \, s)^n\, .
\end{equation}
We bound the left-hand-side by 
\begin{equation}
\label{e: proof thm 8.6 2}
\begin{split}
    \lambda \normpiccolo{\Phi_\eps \ast \delta \E^S_{j_\l}}( B^S_{r_\l s}(x))&\leq \lambda\left( \normpiccolo{\Phi_\eps \ast \delta \E^S_{j_\l}}( B_{r_\l s}(x)) + \normpiccolo{\Phi_\eps \ast \delta \E^S_{j_\l}}( B_{r_\l s}(\sigma(x)))\right)\\
    &=2\lambda  \normpiccolo{\Phi_\eps \ast \delta \E^S_{j_\l}}( B_{r_\l s}(x))\, ,
    \end{split}
\end{equation}
where we used the symmetry. By \eqref{e: proof thm 8.6 1}-\eqref{e: proof thm 8.6 2}, we have 
\begin{equation}
\label{e: proof thm 8.6 3}
    \lambda  \normpiccolo{\Phi_\eps \ast \delta \E^S_{j_\l}}( B_{r_\l s}(x))\ge \omega_n (r_\l \, s)^n\, ,
\end{equation}
for $\abs \d (x)\leq \frac{\gamma \, r_\l \, s}{2}$. Suppose instead that $x\in G_\l^{\ast \ast}$ and $\abs \d (x) > j_\l^{-2}$. We then use \autoref{lemma 8.5 Y far from H0} with $R=r_\l \, s$ and $Y=\set{ T^\perp (x)}$ to obtain 
\begin{equation*}
\label{e: proof thm 8.6 4}
      \lambda  \normpiccolo{\Phi_\eps \ast \delta \E^S_{j_\l}}( B_{r_\l s}(x))\ge \omega_n (r_\l \, s)^n\, .
\end{equation*}
Note that
\begin{equation}
\label{e: proof thm 8.6 5}
    j_\l^{-2}\leq \frac{\gamma\, r_\l\, s}{2}\, ,
\end{equation}
since it is equivalent to $s\, \gamma \ge 2 \frac{j_\l^{-1}}{r_\l} \, j_\l^{-1}$, whose right-hand-side goes to $0$ by \eqref{e: proof thm 8.6 ratio jl rl}. Hence, by \eqref{e: proof thm 8.6 3}, \eqref{e: proof thm 8.6 4} and \eqref{e: proof thm 8.6 5}, we have 
\begin{equation}
\label{e: proof thm 8.6 6}
     \lambda  \normpiccolo{\Phi_\eps \ast \delta \E^S_{j_\l}}( B_{r_\l s}(x))\ge \omega_n (r_\l \, s)^n \ \mbox{for every} \ x\in G_\l^{\ast \ast}\, .
\end{equation}
Arguing as in \cite[Equation $(8.156)$]{KimTonegawa}, we can see that \eqref{e: proof thm 8.6 6} implies 
\begin{equation}
\label{e: proof thm 8.6 7}
    G_\l^{\ast \ast} \subset B_{(\lambda-1)r_\l}\cap \set{x\, \colon \, \abs{\d}(x) \leq 3r_\l \, s}\, .
\end{equation}
We next show that, for all sufficiently large $j_\l$,
\begin{equation}
\label{e: proof thm 8.6 8}
    \Ha^0\left( \set{x\in G_\l^{\ast \ast} \, \colon \, T(x)=a}\right) \leq 2\,\nu-2
\end{equation}
for all $a\in B_{(\lambda-1)r_\l}\cap H_0$. For the sake of a contradiction, suppose there is a sequence of points $a_\l\in B_{(\lambda-1)r_\l}\cap H_0$ such that \eqref{e: proof thm 8.6 8} fails. Thus, there exist a family of sets $Y_\l$ with $\Ha^0(Y_\l)=2\,\nu-1$ such that for each for every $x\in Y_\l$, $T(x)=a_\l$. However, there is also such $Y_\l$ which is symmetric and with $\Ha^0(Y_\l)=2\,\nu$, due to the symmetric definition of $G_\l^{\ast \ast}$ and to the fact that (by definition) $G_\l^{\ast \ast}\cap H_0=\emptyset$. We want to apply \autoref{lemma 8.5} to $Y_\l$ and $R=r_\l$. The diameter requirement (assumption (5) of \autoref{lemma 8.5}) holds true due to \eqref{e: proof thm 8.6 7} by taking $s=\sfrac{\gamma}{6}$. One can check that the other assumptions are satisfied too. Thus, we have
\begin{equation}
\label{e: proof thm 8.6 9}
    \lambda \normpiccolo{\Phi_{\eps_{j_\l}} \ast \partial \E^S_{j_\l}} \left( \set{x\, \colon \, \mathrm{dist}(x, Y_\l)\leq r_\l}\right) \ge \omega_n r_\l^n \,2 \nu\, .
\end{equation}
We may assume after choosing a subsequence
that $r_\l^{-1}a_\l$ converges to $\bar a\in B_{\lambda-1}\cap H_0$. By \eqref{xitg21}, we have
\begin{equation}
\label{e: proof thm 8.6 10}
\begin{split}
    \lambda^n \, \omega_n \, d&=\lim\limits_{\l \to \infty} \normpiccolo{(f_{(r_\l)})_\sharp (\Phi_{\eps_{j_\l}} \ast \partial \E^S_{j_\l})} (B_\lambda(\bar a))\\
    &= \lim\limits_{\l \to \infty} r_\l^{-n} \normpiccolo{\Phi_{\eps_{j_\l}} \ast \partial \E^S_{j_\l} } (B_{\lambda r_\l } (r_\l \, \bar a))\, .
    \end{split}
\end{equation}
If we take \eqref{e: proof thm 8.6 7} with $s=\frac{\sqrt{\lambda}-1}{6}$, we have 
\begin{equation*}
    Y_\l \subset B_{(\lambda-1)r_\l} \cap \set{\abs \d(x) \leq \frac{r_\l (\sqrt{\lambda}-1)}{2}}
\end{equation*}
and $T(y_\l)=a_\l$ for each $y_\l \in Y_\l$, yielding 
\begin{equation*}
    \abs{x - a_\l}\leq \frac{r_\l (\sqrt{\lambda}-1)}{2}+ r_\l \leq \sqrt{\lambda}\, r_\l 
\end{equation*}
for each $x\in \set{z\, \colon \, \mathrm{dist}(z, Y_\l)\leq r_\l}$. Since $\abs{r_\l^{-1}a_\l -\bar a}$ converges to $0$, so does $\abs{a_\l- r_\l \bar a}$. Thus, if $x\in B_{\sqrt{\lambda} r_\l}(a_\l)$, then
\begin{equation*}
    \abs{x-r_\l \bar a}\leq \abs{x- a_\l} + \abs{a_\l -r_\l \bar a}\leq \lambda\,  r_\l \, ,
\end{equation*}
for sufficiently large $\l$. Thus,
\begin{equation}
\label{e: proof thm 8.6 11}
    \set{x\, \colon \, \mathrm{dist}(x,Y_\l) \leq r_\l}\subset B_{\lambda r_\l}(r_\l \, \bar a)\, .
\end{equation}
\eqref{e: proof thm 8.6 9}, \eqref{e: proof thm 8.6 10} and \eqref{e: proof thm 8.6 11} show $\lambda^{n+1}\, d\ge 2\,\nu$, contradicting \eqref{e: proof thm 8.6 ratio d nu} and yielding \eqref{e: proof thm 8.6 8}. 

Finally, we note that
\begin{equation}
\label{gs15}
\lim\limits_{\l\rightarrow\infty}r_\l^{-n}\normpiccolo{T_{\sharp}\partial\E^S_{j_\l}}(B_{(\lambda-1)r_\l}\setminus G_\l^{\ast \ast})
\leq \lim_{l\rightarrow\infty} r^{-n}_\l\normpiccolo{\partial\E^S_{j_\l}}(B_{(\lambda-1)r_\l}\setminus G_\l^{\ast \ast})=0
\end{equation}
due to \eqref{e: proof thm 8.6 8.154 KT}, while 
\begin{equation}
\label{gs16}
\begin{split}
    \normpiccolo{T_\sharp \partial \E_{j_\l}^S} (G_\l^{\ast \ast})&=\int_{B_{(\lambda-1)r_\l}\cap H_0 } \sum_{\{x\in G_\l^{\ast \ast} \, \colon \, T(x)=a\}} \theta^n(\normpiccolo{\partial \E_{j_\l}^S}, x)\, d\Ha^n(a)\\
    &\leq \omega_n ( (\lambda-1)r_\l)^n (2\,\nu-2) 
    \end{split}
\end{equation}
by \eqref{e: proof thm 8.6 8} for all large $j_\l$. By \eqref{xitg21}, 
\begin{equation}
\begin{split}
\lim\limits_{\l\rightarrow\infty} r_\l^{-n}\normpiccolo{T_{\sharp}\partial\mathcal E^S_{j_\l}}(B_{(\lambda
-1)r_\l})&=\normpiccolo{T_{\sharp}\left(d\abs{H_0}\right)}(B_{\lambda-1}) \\
&=\omega_{n}(\lambda-1)^{n}d
\end{split}
\label{gs17}
\end{equation}
and \eqref{gs15}, \eqref{gs16} and \eqref{gs17} show $d\leq 2\,\nu-2$. By \eqref{e: proof thm 8.6 0}, this proves $d=2\,\nu-2$.
\end{proof}
\end{theorem}
\section{Properties of symmetric varifolds and existence of a free boundary Brakke flow}
\label{section: symmetric Brakke flow}
In this section we show how to relate the properties of a  varifold to the properties of its symmetrization. We start by showing how to compare their variations.
\begin{lemma}
\label{lem: variation of reflection}
    Let $V\in \V_n\left(\closedhalf\right)$ and $V^S =V+\sigma_\sharp V\in \V_n(\R^{n+1})$ its symmetrization. Let $g\in C^1_c(\R^{n+1}, \R^{n+1})$ be a vector field and let $g^S$ be its reflection (see \autoref{def: symmetryc objects} (b)). Then,
\begin{equation}
\label{e: variation of symm varifold}
    \delta V^S(g) =\delta V(g) + \delta V(g^S)\, .
\end{equation}
    In particular, if $g=g^S$, that is $\sigma(g(x))=g(\sigma(x))$, we have 
    \begin{equation}
    \label{e: variation of symm varifold with symm vector field}
        \delta V^S(g) = 2 \, \delta V(g)\, .
    \end{equation}
    \begin{proof}
        By linearity we have
        \begin{equation}
        \label{e: computation variation of symm varifold}
        \delta V^S (g)= \delta V(g) + \delta (\sigma_\sharp V) (g)
        \end{equation}
        and 
        \begin{equation}
        \label{e: variation reflection}
            \begin{split}
                \delta (\sigma_\sharp V) (g)
        = \int_{\bG_n(\R^{n+1})} S' \cdot \nabla g(x) \, d(\sigma_\sharp V)(x,S')        = \int_{\bG_n(\R^{n+1})} (Q\,S\, Q )\cdot \nabla g(\sigma(x)) \, dV(x,S)\, .
            \end{split}
        \end{equation}
Since $g^S(x)=\sigma(g(\sigma(x)))$, we have 
\begin{equation}
 \label{e: computation variation of symm of vector field}
    \begin{split}
        \delta V(g^S) &=\int_{\bG_n(\R^{n+1})} S \cdot\left( Q  \, \nabla g(\sigma(x)) \, Q\right) dV(x,S)\\
        &= \int_{\bG_n(\R^{n+1})} \mbox{trace} \left( S \, Q \, \nabla g(\sigma(x)) \, Q\right)\, dV(x,S)\\
        &=\int_{\bG_n(\R^{n+1})} \mbox{trace}  \left( Q \,S \, Q \, \nabla g(\sigma(x)) \right)  dV(x,S)\\
        &= \delta (\sigma_\sharp V)(g)\, ,
            \end{split}
        \end{equation}
        where we used the invariance of the trace under permutations. \eqref{e: computation variation of symm varifold}, \eqref{e: variation reflection} and \eqref{e: computation variation of symm of vector field} show \eqref{e: variation of symm varifold}. \eqref{e: variation of symm varifold with symm vector field} follows immediately.
       
    \end{proof}
\end{lemma}
We show that a symmetric varifold has symmetric approximate mean curvature, whenever it exists.
\begin{lemma}
        Let $V\in \V_n\left(\closedhalf\right)$ and $V^S =V+\sigma_\sharp V\in \V_n(\R^{n+1})$ its symmetrization. If $V^S$ admits generalized mean curvature $\h=\h(\cdot,V^S)$, then $\h$ is symmetric, namely for every $x\in \R^{n+1}$
    \begin{equation}
    \label{e: symmetry mean curvature of a symmetric varifold}
        \sigma (\h(x))=\h(\sigma(x))\,  .
    \end{equation}
    \begin{proof}
    Let $g\in C^1_c(\R^{n+1}, \R^{n+1})$ and let $g^S(\cdot)= \sigma(X(\sigma(\cdot)))$ be its reflection. Note that 
    \begin{equation*}
        (g^S)^S(x)= \sigma (g^S(\sigma(x)))= \sigma(\sigma(g(\sigma(\sigma(x)))))=g(x)\,;
    \end{equation*}
    Together with \eqref{e: variation of symm varifold}, this implies that 
    \begin{equation}
    \label{e: proof lemma symmetry mean curvature 1}
        \delta V^S(X)= \delta V^S(X^S)\, .
    \end{equation}
    If we write both sides of \eqref{e: proof lemma symmetry mean curvature 1} in terms of the generalized mean curvature we get
    \begin{equation}
    \label{e: proof symmetry mean curvature 1}
        \int_{\R^{n+1}} X(x)\cdot \h(x)\, d\normpiccolo{V^S}(x) = \int_{\R^{n+1}} \sigma(X(\sigma(x))) \cdot \h(x)\, d\normpiccolo{V^S}(x)\, .
    \end{equation}
    Taking the change of variable $x=\sigma(y)$ and using the orthogonality of the reflection, we get 
    \begin{equation}
    \label{e: proof symmetry mean curvature 2}
    \begin{split}
        \int_{\R^{n+1}} \sigma(X(\sigma(x))) \cdot \h(x) \, d\normpiccolo{V^S}(x)& = \int_{\R^{n+1}} \sigma(X(x)) \cdot \h(\sigma (x)) \, d\normpiccolo{V^S}(x)\\
        &=\int_{\R^{n+1}} X(x) \cdot \sigma( \h (\sigma (x))) \, d\normpiccolo{V^S}(x)\, .
        \end{split}
     \end{equation}
     \eqref{e: proof symmetry mean curvature 1} and \eqref{e: proof symmetry mean curvature 2} give 
     \begin{equation*}
         \int_{\R^{n+1}} X(x) \cdot \left( \h(x)-\sigma(\h(\sigma(x))) \right) \, d\normpiccolo{V^S}(x)=0\, ,
     \end{equation*}
     yielding \eqref{e: symmetry mean curvature of a symmetric varifold}.
     \end{proof}
\end{lemma}

\begin{lemma}
\label{lem: restriction of a symmetric varifold is free boundary}
      Let $V\in \V_n\left(\closedhalf\right)$ and $V^S =V+\sigma_\sharp V\in \V_n(\R^{n+1})$ its symmetrization. If $V^S$ admits generalized mean curvature $\h$, then $V$ is a varifold with free boundary on $H_0$, having $\h^\fb(x,V)= \h(x,V^S)$ for $x\in \closedhalf$.
      \begin{proof}
          Let $g\in C^1_c(\R^{n+1}; \R^{n+1}) \cap\T{H_0}$; since $\spt (V) \subseteq \clos{ \R^{n+1}_+}$, $\delta V(g)$ is independent of the values of $g$ in $\R^{n+1}_-$ and we may therefore assume that $g$ satisfies $\sigma(g(\sigma(x)))=g(x)$. 
          Thus, by \eqref{e: variation of symm varifold with symm vector field} and \eqref{e: symmetry mean curvature of a symmetric varifold},
          \begin{equation*}
          \begin{split}
              \delta V(g)&= \frac12 \, \delta V^S(g) = -\frac12 \int_{\R^{n+1}} g\cdot \h \, d\normpiccolo{V^S}\\
              &=-\frac12\left( \int_{\R^{n+1}} g(x) \cdot \h(x,V^S)+ g(\sigma(x)) \cdot \h(\sigma(x))\right)\, d\normpiccolo{V}(x)\\
              &=-\frac12\left( \int_{\R^{n+1}} g(x) \cdot \h(x,V^S)+ \sigma(g(\sigma(x))) \cdot \sigma(\h(\sigma(x)))\right)\, d\normpiccolo{V}(x)\\
              &=-\frac12\left( \int_{\R^{n+1}} g(x) \cdot \h(x)+ g(x) \cdot \h(x)\right)\, d\normpiccolo{V}(x)\\
              &=-\int_{\R^{n+1}} g(x)\cdot \h(x) \, d\normpiccolo{V}(x)\, ,
              \end{split}
          \end{equation*}
          that is the thesis.
      \end{proof}
\end{lemma}
\begin{lemma}
\label{lemma: rectifiability of restriction}
        Let $V\in \V_n\left(\closedhalf\right)$ and $V^S =V+\sigma_\sharp V\in \V_n(\R^{n+1})$ its symmetrization. If $V^S\in \RV(\R^{n+1})$, then $V\in \RV(\R^{n+1})$. More precisely, if there exist a countably $n$-rectifiable set $\Gamma$ and a function $\theta\in L^1_{\operatorname{loc}}(\Ha^n\mres\Gamma)$ such that for every $\varphi \in C_c(\bG_n(\R^{n+1}))$
            \begin{equation*}
                V^S(\varphi)= \int_{\Gamma} \varphi(x, T_x\Gamma) \, \theta(x)\, d\Ha^n(x)\, , 
                \end{equation*}
        then there exists a function $g$ such that $g\equiv 1$ on $\R^{n+1}_+\cap \Gamma$, $g\equiv \frac12$ on $H_0\cap \Gamma$ and 
        \begin{equation*}
            V(\varphi) =\int_{\Gamma \cap \closedhalf} \varphi(x, T_x\Gamma) \, g(x) \, \theta(x) \, d\Ha^n(x)\, .
        \end{equation*}
    In particular, if $V^S\in \IV(\R^{n+1})$ with $\theta(x)\in 2\, \Na$ for $\normpiccolo{V^S}-\mbox{a.e.}\  x\in H_0$, then $V\in \IV(\closedhalf)$.
        \begin{proof}

    Since $V\ll V^S$, by Radòn-Nykodym there exists a function $\widetilde g$ such that 
            \begin{equation*}
                V(\varphi)= \int_{\bG_n(\R^{n+1})} \varphi(x,S')\, g(x,S')\, dV^S(x,S')= \int_{\Gamma} \varphi(x,T_x\Gamma) \, \widetilde g(x,T_x\Gamma)\, \theta(x)\, d\Ha^n(x)\, .
            \end{equation*}
Define $g(x)\define \widetilde g(x, T_x\Gamma)$; since $\spt \norm V\subseteq \closedhalf$, if $\varphi\in C_c(\bG_n(\R^{n+1}_+))$, it follows that 
    \begin{equation*}
        V(\varphi)= V^S(\varphi)=\int_{\Gamma} \varphi(x, T_x\Gamma) \, \theta(x)\, d\Ha^n(x)\, ,
    \end{equation*}
yielding $g\equiv 1$ on $\R^{n+1}_+$. It is clear that $g\equiv 0$ on $\R^{n+1}_-.$ Moreover, by definition of pushforward of a rectifiable varifold, 
           \begin{equation*}
               \sigma_\sharp V = \var\left( \sigma \left( \Gamma \cap \closedhalf\right), (g \, \theta)(\sigma^{-1})\right)\,.
           \end{equation*}
Hence, for any $x\in H_0\cap \Gamma$ the multiplicity of $\sigma_\sharp V$ at $x$ is $g(x)\theta(x)$. Thus, the multiplicity of $V^S=V+\sigma_\sharp V$ at $x\in H_0\cap \Gamma$ is $2\, g(x)\theta(x)= \theta(x)$, yielding that $g\equiv \frac12$ on $H_0\cap \Gamma$. 
        \end{proof}
\end{lemma}
\begin{lemma}
\label{lemma: FB brakke flow and symmetric brakke flow}
     Let $\set{V_t}_{t\ge 0}$ be a family of varifolds and let $\setpiccolo{V^S_t}_{t\ge 0}=\set{V_t+\sigma_\sharp (V_t)}_{t\ge 0}$ be its symmetrization. If for $\mbox{a.e.}\  t$, $V^S_t$ has generalized mean curvature $\h(\cdot)=\h(\cdot, V^S)$ and $\setpiccolo{V_t^S}_{t\ge 0}$ satisfies the standard Brakke's inequality, that is for all $0\leq t_1<t_2 <\infty$ and $\varphi\in C^1_c(\R^{n+1}\times  \R_+; \R_+)$,
     \begin{equation}
     \label{e: lem BI}
         \norm{V_t^S}(\varphi(\cdot, t))\Big|_{t=t_1}^{t_2}\leq \int_{t_1}^{t_2} \int_{\R^{n+1}} \left(\nabla\varphi(\cdot, t)-\h(\cdot, V^S)\right)\cdot \h(\cdot, V^S)+ \frac{\partial \varphi}{\partial t}(\cdot, t) \, d\normpiccolo{V^S}\, dt\, ,
     \end{equation}
     then, for $\mbox{a.e.}\ t$, $V_t$ is a varifold with free boundary (we denote $\h^\fb$ its generalized mean curvature) in $H_0$ and for all $0\leq t_1<t_2 <\infty$ and $\varphi\in C^1_c(\R^{n+1}\times  \R_+, \R_+)$ with $\nabla \varphi \in \T{H_0}$ we have 
     \begin{equation}
     \label{e: lem FB BI}
         \norm{V_t}(\varphi(\cdot, t))\Big|_{t=t_1}^{t_2}\leq \int_{t_1}^{t_2} \int_{\R^{n+1}} \left(\nabla\varphi(\cdot, t)-\h^\fb(\cdot, V)\right)\cdot \h^\fb(\cdot, V)+ \frac{\partial \varphi}{\partial t}(\cdot, t) \, d\normpiccolo{V}\, dt\, .
     \end{equation}
     \begin{proof}
$V_t$ has a free boundary on $H_0$ for $\mbox{a.e.} \ t$ by \autoref{lem: restriction of a symmetric varifold is free boundary}. We have to show the validity of the free boundary Brakke's inequality. Since $\spt(V_t)\subset \clos{ \R^{n+1}_+}$, we can assume that $\varphi$ satisfies $\varphi(\sigma(x),t)=\varphi(x,t)$ for all $x\in \R^{n+1}$ and $t\ge 0$. This implies that $\normpiccolo{V_t^S}(\varphi(\cdot, t)=2 \, \normpiccolo{V_t}(\varphi(\cdot, t))$. Similarly, we use the symmetry of $\varphi$ and $\h$ (by \eqref{e: symmetry mean curvature of a symmetric varifold}) to write the RHS of \eqref{e: lem BI} as twice the RHS of \eqref{e: lem FB BI}, obtaining the thesis.
     \end{proof}
\end{lemma}
\begin{lemma}
\label{lem: BV identity reflection}
Suppose that for each $i=1,\ldots,N$, there exists a family of symmetric open sets $\{ E^S_i(t)\}_{t\in[0,T)}$, coupled with symmetric scalar functions $v_i$ satisfying
    \begin{equation}
    \label{e: canonical relation symmetry assumption}
        \left.\int_{E_i^S(t)} \varphi(x,t) \, dx\right|_{t=t_1}^{t_2} =  \int_{t_1}^{t_2} \int_{E_i^S(t)} \frac{\partial\varphi}{\partial t}(x,t) \, dx\,dt   + \int_{t_1}^{t_2} \int_{\partial^\ast E_i^S(t)} \varphi(x,t) \, v_i(x,t)\,d\Ha^n(x)\,dt 
    \end{equation}
    for $\mbox{a.e.}\ 0\leq t_1<t_2<T$ and for all $\varphi \in C^1_c(\R^{n+1}\times [0,T))$.\\
Then, if we let $E_i(t) \define E^S_i(t)\cap \R^{n+1}_+$, we have
     \begin{equation}
     \label{e: canonical relation symmetry}
        \left.\int_{E_i(t)} \varphi(x,t) \, dx\right|_{t=t_1}^{t_2} =  \int_{t_1}^{t_2} \int_{E_i(t)} \frac{\partial\varphi}{\partial t}(x,t) \, dx\,dt   + \int_{t_1}^{t_2} \int_{\partial^\ast E_i(t)\cap \R^{n+1}_+} \varphi(x,t) \, v_i(x,t)\,d\Ha^n(x)\,dt 
    \end{equation}
    for $\mbox{a.e.} \ 0\leq t_1<t_2<T$ and for all $\varphi \in C^1_c(\R^{n+1}\times [0,T))$ with $\nabla \varphi(\cdot, t) \in \T{H_0}$.
    \begin{proof}
We can assume $\varphi$ to be symmetric in the space variable, as \eqref{e: canonical relation symmetry} depends only the values of $\varphi$ in $\R^{n+1}_+$. Since $H_0$ is a set of zero Lebesgue measure, we have 
\begin{equation}
\label{e: proof canonical relation symmetry 1}
    \int_{E_i^S(t)} \varphi(x,t) \, dx = 2 \int_{E_i(t)} \varphi(x,t) \, dx
\end{equation}
and, similarly,
\begin{equation}
\label{e: proof canonical relation symmetry 2}
    \int_{E_i^S(t)} \frac{\partial \varphi}{\partial t}(x,t) \, dx = 2 \int_{E_i(t)} \frac{\partial \varphi}{\partial t} (x,t)\, dx\, .
\end{equation}
Note that $x\in \R^{n+1}_+ \cap \partial^\ast E_i^S(t)$ if and only if $x\in \R^{n+1}_+\cap  \partial^\ast E_i(t)$; thus, using again the symmetry of $\varphi$ and the symmetry of $v_i$, we have
\begin{equation}
\label{e: proof canonical relation symmetry 3}
    \int_{\partial^\ast E_i^S(t)\setminus H_0}\varphi(x,t)v_i(x,t) \, d\Ha^n(x)= 2 \int_{\partial^\ast E_i(t)\cap \R^{n+1}_+}\varphi(x,t)v_i(x,t) \, d\Ha^n(x)\, .
\end{equation}
By \eqref{e: canonical relation symmetry assumption}, \eqref{e: proof canonical relation symmetry 1}, \eqref{e: proof canonical relation symmetry 2} and \eqref{e: proof canonical relation symmetry 3}, we have 
\begin{equation}
\label{e: proof canonical relation symmetry 4}
    \begin{split}
          \left.\int_{E_i(t)} \varphi(x,t) \, dx\right|_{t=t_1}^{t_2} &=   \int_{t_1}^{t_2} \int_{E_i(t)} \frac{\partial\varphi}{\partial t}(x,t) \, dx\,dt   +  \int_{t_1}^{t_2} \int_{\partial^\ast E_i(t)\cap \R^{n+1}_+} \varphi(x,t) \, v_i(x,t)\,d\Ha^n(x)\,dt\\
          &+ \frac12\, \int_{\partial^\ast E_i^S(t) \cap H_0} \varphi(x,t)v_i(x,t) \, d\Ha^n(x)\, dt\, .
    \end{split}
\end{equation}
Once we show that the last term of \eqref{e: proof canonical relation symmetry 4} is $0$, we get the thesis. By the locality of the approximate tangent space to a rectifiable set (see \cite[Proposition $10.5$]{MaggiSets}), for $\Ha^n$-$\mbox{a.e.} \ x\in \partial^\ast E^S_i(t)\cap H_0$, we have $T_x \partial^\ast E^S_i(t)=H_0$. By \cite[Corollary $16.1$]{MaggiSets}, this implies that for $\Ha^n$-$\mbox{a.e.} \ x\in \partial^\ast E_i(t)\cap H_0$, we have $H_0=\left(\nu_{E^S_i(t)}(x)\right)^\perp$, namely $\nu_{E^S_i(t)}$ points in the vertical direction. However, \cite[Exercise $15.10$]{MaggiSets} implies that $\nu_{E^S_i(t)}$ is a symmetric vector field, yielding that 
\begin{equation*}
    \Ha^n( \partial^\ast E^S_i(t) \cap H_0)=0\, .
\end{equation*}
    \end{proof}
\end{lemma}
\begin{lemma}
\label{lem: BV identity reflection 2}
    Under the same assumptions of \autoref{lem: BV identity reflection}, let $I_{i,j}^S(y)\define \partial^\ast E_i^S(t)\cap \partial^\ast E_j^S(t$) and suppose that the symmetric scalar functions $v_i$ further satisfy 
\begin{equation*} 
        \sum_{i\neq j} \int_0^T \int_{I_{i,j}^S(t)} {\rm div}\,g - (\nu_i \otimes \nu_i) \cdot \nabla g\,d\Ha^n\,dt = - \sum_{i \neq j} \int_0^T \int_{I_{i,j}^S(t)} v_i \, \nu_i \cdot g \,  d\Ha^n\,dt
    \end{equation*}
    for every vector field $g\in C^1_c(\R^{n+1}\times [0,T];\R^{n+1})$.\\
Then, if we let $I_{i,j}(t)\define \partial^\ast E_i(t) \cap \partial^\ast E_j(t) \cap \R^{n+1}_+$, we have 
\begin{equation*} 
        \sum_{i\neq j} \int_0^T \int_{I_{i,j}(t)} {\rm div}\,g - (\nu_i \otimes \nu_i) \cdot \nabla g\,d\Ha^n\,dt = - \sum_{i \neq j} \int_0^T \int_{I_{i,j}(t)} v_i \, \nu_i \cdot g \,  d\Ha^n\,dt
    \end{equation*}
    for every vector field $g\in C^1_c(\R^{n+1}\times [0,T];\R^{n+1})\cap \T{H_0}$.
\end{lemma}
The proof is analogous to that of \autoref{lem: BV identity reflection}.
\section{Conclusions}
We finally prove \autoref{theorem: existence of limit symmetric measures}.
    \begin{proof}[Proof of \autoref{theorem: existence of limit symmetric measures}.]
       \autoref{prop: symmetry along iteration} shows that at any step of the iteration we have symmetric partitions. One can repeat the arguments of \cite[Section $6$]{KimTonegawa} to see that the iteration yields the existence of a subsequence $\set{j_\l}_{\l=1}^\infty$ and a family of Radon measures $\set{\mu^S_t}_{t\ge 0}$ such that 
       \begin{equation*}
           \lim\limits_{\l\to \infty}\normpiccolo{\partial \E^S_{j_\l} (t)}=\mu^S_t
       \end{equation*}
        in the sense of measures. By \autoref{lemma: limit of symmetric is symmetric}, $\mu_t^S$ has to be symmetric too. The rectifiability and the integrality of these measures are proved by, respectively, \autoref{theorem: rectifiability} and \autoref{theorem: integrality}. Concerning Brakke's inequality, we can prove it with respect to symmetric functions (which can be suitably approximated by the class $\cA_j$) as in \cite[Section $9$]{KimTonegawa}. Once this is done, for a fixed test function $\varphi \in C^1_c(\R^{n+1}\times [0,\infty)\, ; \R_+)$, we test Brakke's inequality with $\varphi (\cdot,t) + \varphi(\sigma(\cdot),t)$ and, arguing similarly to the proof of \autoref{lemma: FB brakke flow and symmetric brakke flow}, we can show the validity of Brakke's inequality for $\varphi$. Finally, the results concerning the BV flow can be shown by the arguments of \cite[Theorem $2.11$-$2.12$]{STCanonical}.

    \end{proof}
    We then use the results of the previous section to infer the existence of a Brakke flow with free boundary on $H_0$.
    \begin{proof}[Proof of \autoref{thm: main theorem precise}]
    It is an immediate consequence of \autoref{theorem: existence of limit symmetric measures} together with \autoref{lemma: rectifiability of restriction}, \autoref{lemma: FB brakke flow and symmetric brakke flow}, \autoref{lem: BV identity reflection} and \autoref{lem: BV identity reflection 2}.
        
    \end{proof}

\begin{proof}[Proof of \autoref{theorem: unit density precise}]
    The assumption \eqref{e: assumption density ratio} is 
    \begin{equation*}
          \sup_{x\in \R^{n+1}, 0<r<r_0} \frac{\Ha^n(\Gamma_0\cap B_r(x))+ \Ha^n(\Gamma_0\cap B_r(\sigma(x))}{\omega_n \, r^n}<2-\delta_0\, ,
    \end{equation*}
    which is equivalent to 
    \begin{equation*}
         \sup_{x\in \R^{n+1}, 0<r<r_0} \frac{\Ha^n(\Gamma_0^S\cap B_r(x))}{\omega_n \, r^n}<2-\delta_0\, .
    \end{equation*}
    Under this assumption, by the same arguments of \cite[Theorem $2.13$]{STCanonical}, one can see that $V_t^S$ is unit density for some short time $T_0=T_0(n,r_0,\delta_0, \Ha^n(\Gamma_0))$. On the other hand, \autoref{theorem: integrality} shows that for $\mbox{a.e.} \ t\in \R^+$, we have that for $\normpiccolo{V^S_t}\mbox{-a.e.} \ x\in H_0$, the density is even. This yields the thesis. 
   
\end{proof}
    \appendix
    \section{A regularity result at the boundary}
    We state a up-to-the-boundary regularity result, which is simply obtained by \virgolette{halving} the hypothesis of \cite{KasaiTonegawa, EndTimeST} across $H_0$ and reflecting. As in \autoref{section: integrality}, we will denote by $T$ the matrix representing the projection onto the hyperplane $H_0$ and $T^\perp$ its orthogonal complement. Fix $\phi \in C^\infty([0,\infty))$ such that $0\leq \phi \leq 1$,
    \begin{equation*}
            \phi(x)=\begin{cases*}
                1 & $ \mbox{for} \ 0\leq x \leq \left(\frac23\right)^\frac1n \, ,$\\
                0 & $\mbox{for} \  x \ge \left(\frac56\right)^\frac1n \, .$
        \end{cases*}
    \end{equation*}
    For $R>0$ and $x\in \R^{n+1}$, define 
    \begin{equation*}
       \phi_{T,R}(x) \define \phi\left(\frac{\abs{T(x)}}{R}\right), \quad  \phi_T(x)\define \phi(\abs{T(x)})
    \end{equation*}
    and set 
    \begin{equation*}
        \mathbf{c} \define \int_{H_0} \phi_T^2(x)\, d\Ha^n(x)\, .
    \end{equation*}
    For $a\in H_0$ and $r\in (0,\infty),$ we define the cylinder 
    \begin{equation*}
        C(T,a,r)\define \set{x\in \R^{n+1} \, \colon \, \abs{T(x)-a}< r}\, .
    \end{equation*}
    To ease a bit the notation, in the following we will assume the flow to be defined in some negative neighborhood of the time $0$. In the following, for a space-time function $f$, we will denote
    \begin{equation*}
        [f]_\alpha\define \sup \set{
    \frac{\abs{f(y_1,s_1)-f(y_2,s_2)}}{\max\{\abs{y_1-y_2},\abs{s_1-s_2}^{\frac12}\}^\alpha} \, \colon \, (y_1,s_1),(y_2,s_2)\in D\,, \;  (y_1,s_1)\neq(y_2,s_2)} \, .
    \end{equation*}
    \begin{theorem}
Corresponding to $\nu\in(0,1)$ and $E_1\in\left[\frac12, \infty\right),$ there exist $\eps_1\in (0,1)$ and $c_1\in(1,\infty)$ depending only on $n, \nu$ and $E_1$ with the following property. Let $\set{V_t}_{t\ge -\Lambda}$ be the free boundary Brakke flow constructed in \autoref{thm: main theorem precise}. Let $R \in (0,\infty)$ and $a\in H_0$. Suppose that 
        \begin{enumerate}
            \item For $\mbox{a.e.} \ t \in [-R, 0]$, $V_t$ is unit density in $C(T,a,2R)\,$;
            \item $\displaystyle \norm{V_t}(B_r(x))\leq \omega_n \, r^n \, E_1 \quad \mbox{for all} \ B_r(x)\subset C(T,a,2R) \ \mbox{and} \ t\in[-R,0]\, ;$
            \item $\displaystyle \normpiccolo{V_{-\frac{4R^2}{5}}}(\phi_{T,R^2}^2)\leq \left(1-\frac{\nu}{2}\right)\, \mathbf c \, R^n \, ;$
            \item $\displaystyle    \left(C(T,a,\nu R) \times \set{0}\right) \cap \spt(\norm{V_t} \times dt) \neq \emptyset\,$;
            \item $\displaystyle  \mu\define\left(R^{-(k+4)}\int_{-R^2}^0\int_{\rC(T,a,2R)}
    \abs{T^{\perp}(x)-a}^2\,d\norm{V_t}dt\right)^{\frac12}<\eps_1\,$.
        \end{enumerate}
 Let $D\define (B_{\frac{R}{2}}(a)\cap H_0) \times \left[-\frac{R^2}{4},0\right]$. Then there are $C^{2,\alpha}$ functions
$f: D\rightarrow H_0^\perp$ and $F: D\rightarrow \mathbb R^{n+1}$ such
that $T(F(y,t))=y$ and $T^\perp(F(y,t))=f(y,t)$ for all $(y,t)\in D$,
\begin{equation*} 
    {\rm spt}\norm{V_t}\cap \rC(T,a, R/2)={\rm image}\,F(\cdot,t)
    \mbox{ for all }t\in\left[-\frac{R^2}{4},0\right),
\end{equation*}
\begin{equation*}
   R^{-1} \norm{f}_0+\norm{\nabla f}_0
   +R(\normpiccolo{\nabla^2 f}_0+\norm{f_t}_0)+R^{2}
   ([\nabla^2 f]_\alpha+[f_t]_\alpha)
   \leq c_1\max\set{\mu,\abs{u}_\alpha}\, .
\end{equation*}
    \end{theorem}
These assumptions make sure that the reflected flow $\setpiccolo{V_t^S}_{t\ge -\Lambda}$ satisfies all the assumptions of \cite[Theorem $2.3$]{EndTimeST}, implying the thesis.

    \section{Initial datum with infinite area}
    \label{appendix: weight}
    In \autoref{ass:main} we assumed that $\Ha^n(\Gamma_0)<\infty$, but this was not necessary. Indeed, we can assume $\Ha^n(\Gamma_0)=\infty$, as long as the area grows at most exponentially at infinity. Let $\Omega\in C^2(\R^{n+1})$ satisfying 
    \begin{equation*}
        0<\Omega(x) \leq 1, \quad \abs{\nabla \Omega(x)}\leq c_1 \Omega(x), \quad \normpiccolo{\nabla^2 \Omega(x)}\leq c_1\Omega(x), \quad \Omega(x)=\Omega(\sigma(x))
    \end{equation*}
    for all $x\in \R^{n+1}$ and for some constant $c_1$ (the reader should not confuse this with the constant $c_1$ of \autoref{prop: area reducing symmetric radial projection}). The assumption $\Ha^n(\Gamma_0)<\infty$ is then replaced by $\norm{\Gamma_0}(\Omega)<\infty$ and everything works the same. To be more precise, some results are slightly weaker in this setting, for instance the energy dissipation rule (3) of \autoref{thm: main theorem precise} does not hold. We refer to \cite{KimTonegawa, STCanonical} for details. We just note that the symmetry of $\Omega$ is required here since it is necessary to have $\Omega\in \cA_j$ for sufficiently large $j$.

    \section{Non flat domains}
    \label{appendix: non flat domains}
    We briefly discuss why the presented approach cannot be directly extended to general domains. Suppose that $\Omega\subseteq \R^{n+1}$ is a bounded domain of class $C^2$. Let 
    \begin{equation*}
        \sigma(x)=x-2\d(x)\, \nabla \d(x)
    \end{equation*}
    be the reflection map, defined on some neighborhood of $\partial \Omega$. $\sigma$ is not a global isometry and, roughly speaking, for any $x$ close to $\partial \Omega$, $\sigma$ is far from being an isometry by a factor of order $K \, \abs{\d}(x)$, where $K$ is some global bound on the curvature of the (compact) hypersurface $\partial \Omega$. In other words, we have that $D \sigma \approx Q + R$, where $Q$ is the same of this article and with $\abs{R} (x)\approx K\, \abs{d}(x).$ Therefore, if we have $V\in \V_n(\overline\Omega)$ and its symmetrization $V^S=V+\sigma_\sharp V$, whenever $V^S$ has approximate mean curvature, it is symmetric along $\partial \Omega$. However, the difference is that in the iteration we rather use the \emph{smoothed} mean curvature, defined using a convolution. Therefore, this vector field at $x\in \partial \Omega$ is influenced by the values of the first variation around $x$, where the symmetry is weaker due to the curvature. As we stated in \eqref{e:h in L infty}, $\abs{\h_\eps} (x,V)\leq \eps^{-2}$, but actually one can get a sharper bound, that is 
    \begin{equation*}
        \abs{\h_\eps}(x,V)\leq \eps^{-1-\delta}
    \end{equation*}
    for any given $\delta >0$. By reflecting the varifold, we get a much lower normal component of $\h_\eps$, but we only gain a factor of order $\eps^{1-\delta}$, which is the size seen by the convolution kernel together with the error term $R$ we wrote above. Thus, what we can get is 
    \begin{equation*}
        \abs{\h_\eps (x,V^S)\cdot \nu_\Omega}\leq \eps^{-2\delta}
    \end{equation*}
    for some $\delta>0$. Along the iteration, when we apply the diffeomorphism 
    \begin{equation*}
    f_2(x)=x+\Delta t_j \, \h_\eps(x, V^S)\, ,
    \end{equation*}
    we would get a new domain $\widetilde \Omega$ whose distance from the original $\Omega$ is only of a factor $\Delta t_j \, \eps^{-\delta}$, but after $\frac{j}{\Delta t_j}$ iterations this quantity explodes. One could try to fix this issue by defining a new curvature by cutting off the normal component as
    \begin{equation*}
        \widehat \h_\eps = \h_\eps(V^S) - \eta(\d)(\h_\eps(V^S)\cdot \nabla \d)\, \nabla \d\, ,
    \end{equation*}
    but this would destroy the convolution structure of the smoothed mean curvature and would make most of the computations of \cite{KimTonegawa} very difficult to replicate.
\bibliography{refs}

\end{document}